%% file: instantons_and_odd_khovanov_homology_std.tex
\documentclass [11pt]{amsart}
\usepackage {amsmath, amssymb, amscd, graphicx, color, url, epsf, dbnsymb, multirow}
\usepackage[all,knot,arc]{xy}
\usepackage[text={6.5in,9in},centering,letterpaper,dvips]{geometry}
\usepackage{hyphenat}
\usepackage{multirow}
\usepackage{booktabs}
\usepackage{mathrsfs}

\newtheorem{theorem}{Theorem}[section]
\newtheorem{lemma}[theorem]{Lemma}
\newtheorem{corollary}[theorem]{Corollary}
\newtheorem{prop}[theorem]{Proposition}
\newtheorem{definition}{Definition}[section]

\newcommand{\y}{Y}
\newcommand{\x}{X}
\newcommand{\s}{S}

\newcommand{\e}{E}
\newcommand{\dt}{D}
\newcommand{\lt}{L}
\newcommand{\T}{T}

\newcommand{\lbb}{\mathbb{L}}
\newcommand{\sbb}{\mathbb{S}}
\newcommand{\ybb}{\mathbb{Y}}
\newcommand{\ubb}{\mathbb{U}}
\newcommand{\xbb}{\mathbb{X}}
\newcommand{\vbb}{\mathbb{V}}
\newcommand{\wbb}{\mathbb{W}}
\newcommand{\ebb}{\mathbb{E}}
\newcommand{\tbb}{\mathbb{T}}

\newcommand{\so}{\text{SO}(3)}
\newcommand{\su}{\text{SU}(2)}
\newcommand{\sone}{S_1}
\newcommand{\stwo}{S_2}
\newcommand{\sthree}{T}

\newcommand{\sthreebb}{\mathbb{T}}

\newcommand{\vothreebb}{\mathbb{V}}

\newcommand{\eothreebb}{\mathbb{E}}
\newcommand{\ih}{I}
\newcommand{\chain}{\text{C}}
\newcommand{\chainbf}{\textbf{C}}
\newcommand{\partialbf}{\boldsymbol{\partial}}
\newcommand{\moduli}{M}
\newcommand{\ind}{\mu}
\newcommand{\knot}{K}
\newcommand{\kbb}{\mathbb{K}}
\newcommand{\hbb}{\mathbb{H}}
\newcommand{\modx}{\moduli_{G(\sthree)}(a,\xbb_{03},b)}

\newcommand{\lime}{\lambda}

\newcommand{\cp}{-\mathbb{C}\mathbb{P}^2}
\newcommand{\kh}{\overline {\text{Kh}'}}
\newcommand{\hf}{\widehat{\text{HF}}}
\newcommand{\hm}{\widetilde{\text{HM}}}
\newcommand{\conna}{a}
\newcommand{\connag}{\mathfrak{a}}
\newcommand{\connb}{b}
\newcommand{\connbg}{\mathfrak{b}}
\newcommand{\connc}{c}
\newcommand{\conncg}{\mathfrak{c}}
\newcommand{\conndg}{\mathfrak{d}}
\newcommand{\connA}{A}

\newcommand{\triv}{\mathfrak{t}}
\newcommand{\ring}{F}
\newcommand{\relp}{\mathscr{P}}
\newcommand{\qgr}{q}
\newcommand{\hgr}{t}
\newcommand{\ext}{{\textstyle{\bigwedge}}}

\title{Instantons and odd Khovanov homology}

\author{Christopher W. Scaduto}

\begin{document}
\maketitle

\begin {abstract} 
	We construct a spectral sequence from the reduced odd Khovanov homology of a link
	converging to the framed instanton homology of the double cover branched over the link,
	with orientation reversed.
	Framed instanton homology 
	counts certain instantons on the cylinder of a 3-manifold 
	connect-summed with a 3-torus.
	En route, we provide a new proof of Floer's surgery exact triangle 
	for instanton homology using metric stretching maps,
	and generalize the exact triangle to a link surgeries spectral sequence. 
	Finally, we relate framed instanton homology to Floer's instanton homology 
	for admissible bundles.
\end {abstract}

\input{intro}
\vspace{10px}
\input{main}
\vspace{10px}
\input{topology}
\vspace{10px}
\input{instantons}
\vspace{10px}
\input{triangle}
\vspace{10px}
\input{links}
\vspace{10px}
\input{framed}
\vspace{10px}
\input{oddkhov}
\vspace{10px}
\input{homology3spheres}
\vspace{10px}
\input{euler}
\vspace{10px}
\bibliographystyle{plain}

\end{document}

%% file: intro.tex
\maketitle

\section{Introduction}

Given a closed, connected, oriented 3-manifold $\y$, we study the framed instanton 
homology $\ih^\#(\y)$, an absolutely $\mathbb{Z}/4$-graded abelian group 
which is an invariant of the oriented homeomorphism type of $\y$.
This group is obtained by counting, in a suitable sense, 
$\so$-instantons on a bundle over $\mathbb{R}\times(\y\# T^3)$ 
which is non-trivial when restricted to the 3-tori. These groups are a 
special case of Floer's instanton homology for admissible bundles from \cite[\S 1]{f2}
and are considered by Kronheimer and Mrowka in \cite[\S 4.1]{kmki} and \cite[\S 4.3]{kmu}.
The terminology {\em framed} is from \cite{kmki}.

Given an oriented link $L$ in $S^3$, we relate the framed instanton homology
of $\Sigma(L)$, the double cover of $S^3$ branched over $L$, 
to the reduced odd Khovanov homology of $L$. This latter invariant, 
written $\kh(L)$, was defined by Ozsv\'ath, Rasmussen and Szab\'o in \cite{ors}. 
It is an abelian group bigraded by a quantum grading, $\qgr$, and a homological 
grading, $\hgr$. It is a variant of Khovanov homology, defined originally by Khovanov in \cite{kh}.\\

\begin{theorem}\label{thm:1}
Given an oriented link $L$ in $S^3$, there is a spectral sequence whose second page 
is $\kh(L)$ that converges to $\ih^\#(\overline{\Sigma(L)})$. 
Each page of the spectral sequence comes equipped with a $\mathbb{Z}/4$-grading,
which on $\kh(L)$ is given by
\begin{equation}
	\frac{3}{2}\qgr -\hgr +\frac{1}{2}\left(\sigma+\nu\right) \mod 4,\label{eq:oddkhgr}
\end{equation}
where $\sigma$ and $\nu$ are the signature and nullity of $L$, respectively,
and the induced $\mathbb{Z}/4$-grading on $\ih^\#(\overline{\Sigma(L)})$ is the usual one.\\
\end{theorem}

\noindent Our convention is that the signature of the right-handed trefoil is $+2$. 
The theorem immediately implies the four rank inequalities
\[
	\text{rk}_\mathbb{Z}\kh(\lt)_j \geq  \text{rk}_\mathbb{Z}\ih^\#(\overline{\Sigma(L)})_j
\]
where $j\in\mathbb{Z}/4$ and the gradings are as in the Theorem. Non-split alternating links, and 
more generally quasi-alternating links as introduced in \cite{os}, have the 
property that $\kh(\lt)$ is supported in the even gradings of (\ref{eq:oddkhgr}) and is a free 
abelian group of rank $\text{det}(L)$; see \cite[Thm. 1]{man}, \cite[\S 5]{ors} and the remarks 
in \cite[\S 9.3]{ls}. In these cases the spectral sequence collapses.

\begin{corollary}
	If $L$ is a quasi-alternating link, then $\ih^\#(\Sigma(L))$ is free
	abelian of rank $\text{det}(L)$ and is supported in even gradings. 
	The rank in grading $j\in\{0,2\}\subset\mathbb{Z}/4$ is given by
	\[
		\frac{1}{2}\left[\text{det}(L) + (-1)^{j/2}2^{\# L - 1}\right]
	\]
	where $\# L$ is the number of components of $L$. \label{cor:2}
\end{corollary}

\noindent 
If rational coefficients are assumed, 
of the 250 prime knots that have at most 10 crossings, 
only 7 of them have potentially non-trivial differentials 
after the $E^2$-page of Theorem \ref{thm:1}. This follows 
from the computations of odd Khovanov homology in \cite{ors}.

To put Theorem \ref{thm:1} into context, we relate framed instanton 
homology to previously studied instanton invariants. To start,
framed instanton homology is a special case of 
Floer's instanton homology for admissible bundles. 
An $\so$-bundle $\ybb$ over a connected 3-manifold $\y$
is {\em admissible} if either $\y$ is a homology 3-sphere,
in which case $\ybb$ is trivial, or there is an oriented surface
$\Sigma\subset\y$ with $\ybb|_\Sigma$ non-trivial. 
This latter condition guarantees that $\ybb$ does not
support any reducible flat connections.

A \textit{geometric representative} for $\ybb$
is an unoriented, closed 
1-manifold $\omega\subset\y$ with $[\omega]\in H_1(\y;\mathbb{F}_2)$ 
Poincar\'{e} dual to $w_2(\ybb)$.
The non-trivial admissibility condition 
for $\ybb$ amounts to the existence of an oriented surface $\Sigma\subset\y$ 
that intersects $\omega$ in an odd number of points, or,
equivalently, to the conditions that $[\omega]$ is non-zero and lifts to a non-torsion 
class in $H_1(\y;\mathbb{Z})$.

For admissible $\ybb$, Floer defined in \cite{f1,f2} 
a relatively $\mathbb{Z}/8$-graded abelian group $\ih(\ybb)$. 
When $\y$ is a homology 3-sphere and $\ybb$ is trivial, 
we write $\ih(\y)$ for this group, 
and it comes equipped with an absolute $\mathbb{Z}/8$-grading.
The isomorphism class of $\ih(\ybb)$ depends only on the 
oriented homeomorphism type of $\y$ and $w_2(\ybb)$.

Now let $\ybb$ be any $\so$-bundle over a closed, connected, oriented 3-manifold $\y$ 
geometrically represented by $\lambda$.
Making some inessential choices, we can construct a bundle $\ybb^\#$ over $\y\# T^3$ by gluing 
together $\ybb$ and a non-trivial bundle over $T^3$. The bundle $\ybb^\#$ 
is always admissible, and the group $\ih(\ybb^\#)$ is always 4-periodic. 
The {\em framed instanton homology twisted by $\lambda$}, 
written $\ih^\#(\y;\lambda)$,
is relatively $\mathbb{Z}/4$-graded 
and isomorphic to four consecutive gradings 
of $\ih(\ybb^\#)$. When $\lambda$ is mod 2 null-homologous, 
we recover the framed instanton homology $\ih^\#(\y)$.

When $\y$ is a homology 3-sphere, we relate $\ih^\#(\y)$ to 
Floer's $\mathbb{Z}/8$-graded $\ih(\y)$.
It is convenient to employ 
Fr{\o}yshov's reduced groups $\widehat{\ih}(\y)$ from \cite{froy},
which are obtained from $\ih(\y)$ by considering interactions 
with the trivial connection.
They come equipped with an absolute $\mathbb{Z}/8$-grading and 
a degree 4 endomorphism $u$. Fr{\o}yshov's Theorem 
10 says $(u^2-64)^n=0$ for some $n>0$, when 
the coefficient ring used contains 
an inverse for $2$.

If $\ybb$ is non-trivial and admissible,
the situation is simpler, as there is no trivial connection 
to worry about. In this case $u$ is a degree $4$ endomorphism 
defined on $\ih(\ybb)$, and Fr{\o}yshov's Theorem 9 of\ \cite{froy} says
$(u^2-64)^n=0$ for some $n>0$.
The proof of the following is essentially 
an application of Fukaya's connected sum theorem of \cite{fukaya}.

\begin{theorem}\label{thm:integerhom}
Let $F$ be a field with char$(F)\neq 2$, and 
suppose all homology groups are taken with $F$-coefficients, 
unless indicated otherwise. 
If $H_1(\y;\mathbb{Z})=0$, then
\[
		\ih^\#(\y) \simeq  \text{{\em ker}}(u^2-64) \otimes H_\ast(S^3) \oplus H_\ast(\text{pt.})
\]
as $\mathbb{Z}/4$-graded $F$-modules, where $u^2-64$ 
is acting on $\bigoplus_{j=0}^3\widehat{\ih}(\y)_j$.
If $\ybb$ is non-trivial and admissible 
with geometric representative $\lambda$, then 
\[
		\ih^\#(\y;\lambda) \simeq \text{{\em ker}}(u^2-64)\otimes H_\ast(S^3)
\]
as relatively $\mathbb{Z}/4$-graded $F$-modules, where
$u^2-64$ is acting on four consecutive gradings of the relatively $\mathbb{Z}/4$-graded $F$-module $\ih(\ybb)$.
\end{theorem}

\noindent When $L$ is the $(3,5)$ torus knot, the double cover of $S^3$ branched over $L$
is the Poincar\'{e} homology sphere $\Sigma(2,3,5)$. 
By the results in \cite{froy}, Fr{\o}yshov's reduced group 
for $\Sigma(2,3,5)$ is trivial. Theorem \ref{thm:integerhom} 
implies that $\ih^\#(\Sigma(2,3,5))$ has rank 1, supported in grading 0. 
This provides an example where the 
spectral sequence of Theorem \ref{thm:1} does not collapse, 
as the reduced odd Khovanov homology of $L$ has rank 3, 
as computed in \cite{ors}.

As another application of Theorem \ref{thm:integerhom}, a simple
inductive argument involving the exact triangle, 
which we present in \S \ref{sec:euler}, yields
the following.

\begin{corollary} For any $\y$ and $\lambda$, we have $\chi(\ih^\#(\y;\lambda)) = |H_1(\y;\mathbb{Z})|$.
\label{cor:euler}
\end{corollary}

\noindent For a set $S$, the notation $|S|$ is to be interpreted as the cardinality 
of $S$ if it is finite, and $0$ otherwise.

Theorem \ref{thm:integerhom} suggests that 
knowledge of the smallest positive integer $n$ such that $(u^2-64)^n=0$ 
is useful in understanding the 
relationships between the various instanton groups. It is known, 
cf. \cite[\S 6]{froy}, that if $\ybb$ is non-trivial 
admissible and restricts non-trivially 
to a surface of genus $\leq 2$, then one can take $n=1$. 
For the following, let $h:\Theta_\mathbb{Z}^3\to\mathbb{Z}$ 
be Fr{\o}yshov's homomorphism from \cite{froy}, where $\Theta_\mathbb{Z}^3$ is 
the integral homology cobordism group.

\begin{corollary}\label{cor:3}
Let $\y$ be the result of $\pm 1$-surgery on a knot $K\subset S^3$ with genus $\leq 2$. 
Let $F$ be a field with char$(F)\neq 2$. Then, with all homology 
taken with $F$-coefficients, we have an isomorphism
\[
	\ih^\#(\y) \simeq H_\ast(\text{pt.})\oplus H_\ast(S^3)\otimes \bigoplus_{j=0}^3\widehat{\ih}(\y)_j 
\]
as $\mathbb{Z}/4$-graded $F$-modules. In particular, if in addition $h(\y)=0$, then the groups $\widehat{\ih}(\y)_j$ on the right can be replaced by $\ih(\y)_j$. 
\end{corollary}

\noindent We apply this result using computations from the literature. 
Fr{\o}yshov proves in \cite{froy} that for the family of 
manifolds $\Sigma(2,3,6k+1)$ with $k$ a positive integer
we have $h=0$. 
Furthermore, $\Sigma(2,3,6k+1)$ can be realized as $1$-surgery on 
the twist knot $(2k+2)_1$ with $k$ full twists, which is a knot of genus 1. Using 
the computation of $\ih(\Sigma(2,3,6k+1))$ from \cite{fs} we obtain

\begin{corollary}\label{cor:4} With coefficients in a field $F$ with char$(F)\neq 2$, 
we have isomorphisms
\[
	\ih^\#(\Sigma(2,3,6k+1)) \simeq F_0^{\lfloor k/2\rfloor+1}\oplus
	F_1^{\lfloor k/2\rfloor}\oplus F_2^{\lceil k/2\rceil}\oplus F_3^{\lceil k/2\rceil} 
\]
where $k$ is a positive integer and the subscripts indicate the gradings.
\end{corollary}

\noindent We also consider the manifolds $\Sigma(2,3,6k-1)$ with $k$ a positive integer, 
which are obtained from $-1$-surgery on twist knots with $2k-1$ half twists. 
In this case $h(\Sigma(2,3,6k-1))=1$, and we can again use the computations of \cite{fs} to obtain

\begin{corollary}\label{cor:5} With coefficients in a field $F$ with char$(F)\neq 2$, 
we have isomorphisms
\[
	\ih^\#(\Sigma(2,3,6k-1)) \simeq F_0^{\lceil k/2\rceil}\oplus
	F_1^{\lceil k/2\rceil-1}\oplus F_2^{\lfloor k/2\rfloor}\oplus F_3^{\lfloor k/2\rfloor}
\]
where $k$ is a positive integer and the subscripts indicate the gradings.
\end{corollary}

\noindent As $\Sigma(2,3,6k\pm 1)$ is the branched double cover over the $(3,6k\pm 1)$ torus knot, Corollaries \ref{cor:4} and \ref{cor:5} provide more examples in which the spectral sequence of Theorem \ref{thm:1} does not collapse. In \cite{bloom}, Bloom considers a spectral sequence
in monopole Floer homology with $\mathbb{F}_2$-coefficients 
analogous to that of Theorem \ref{thm:1}.
For the $(3,6k\pm 1)$ torus knot, he conjectures that the spectral sequence collapses at the fourth page, 
and guesses the higher differentials. 
We note that if his speculation were to hold in our setting (also with $\mathbb{F}_2$-coefficients), 
then we would recover, using Bloom's computations and our grading (\ref{eq:oddkhgr}), 
the formulae of Corollaries \ref{cor:4} and \ref{cor:5}, but 
with $F=\mathbb{F}_2$.\\

\subsection{Background \& Motivation}

In the setting of Heegaard Floer homology, Ozsv\'ath and Szab\'o in \cite{os} 
constructed a spectral sequence with $\mathbb{F}_2$-coefficients 
from the reduced Khovanov homology of a link $L$ to the Heegaard Floer hat homology
of $\Sigma(L)$, with orientation reversed. This was the first instance of a
structural relation between a Floer homology and a combinatorial link homology. 
For this result we write
\[
	\overline{\text{Kh}}(L;\mathbb{F}_2)\;\; \leadsto \;\;\hf(\overline{\Sigma(\lt)};\mathbb{F}_2).
\]
The notation $A\leadsto B$ is an abbreviation for the existence of a spectral 
sequence with some starting page $A$, converging to $B$. 
An essential ingredient for their construction 
was a link surgeries spectral sequence. 
See \S \ref{sec:main} for the instanton analogue.
Subsequently, Bloom \cite{bloom}
constructed a link surgeries spectral sequence in the setting 
of monopole Floer homology, and from this obtained a spectral sequence 
\[
	\overline{\text{Kh}}(L;\mathbb{F}_2)\;\; \leadsto \;\;\hm(\overline{\Sigma(\lt)};\mathbb{F}_2).
\]
It has since 
been shown, after much work, that the monopole Floer group $\hm(\y)$ is isomorphic to $\hf(\y)$, cf. \cite{klt,hgn,tech}.

Ozsv\'ath and Szab\'o speculated in \cite{os} 
that their spectral sequence, if lifted to $\mathbb{Z}$-coefficients, 
would not have reduced Khovanov homology as the $E^2$-page, 
but would have some other link homology with altered signs in the differentials. 
With this is mind, Ozsv\'ath, Rasmussen and Szab\'o \cite{ors} defined an abelian 
group $\text{Kh}'(L)$ called the {\em odd Khovanov homology of} $L$. With 
$\mathbb{F}_2$-coefficients, it coincides with Khovanov homology; but they 
are very different with $\mathbb{Z}$-coefficients.

Odd Khovanov homology is bigraded
by a homological grading, called $\hgr$, and a quantum grading, 
called $\qgr$.
There is a splitting
\[
	\text{Kh}'(\lt)_{\hgr,\qgr}\simeq \kh(\lt)_{\hgr,\qgr-1}\oplus \kh(\lt)_{\hgr,\qgr+1},
\]
where $\kh(\lt)$ is called the {\em reduced} odd Khovanov homology. 
The bigraded group $\kh(\lt)_{\hgr,\qgr}$ categorifies the Jones 
polynomial $J_L$ in the sense that
\[
	J_{L}(x) = \sum_{\hgr,\qgr} (-1)^\hgr\text{rk}_\mathbb{Z}(\kh(L)_{\hgr,\qgr})x^\qgr.
\]
Here, $J_\text{unknot}(x)=1$. It was conjectured in \cite{ors} 
that there is a spectral sequence
\begin{equation*}
	\kh(L) \leadsto \hf(\overline{\Sigma(L)})
\end{equation*}
with $\mathbb{Z}$-coefficients. Our Theorem \ref{thm:1} provides such a spectral sequence with instanton homology in place of Heegaard Floer homology.

The correct version of instanton homology for the task 
was suggested in the work of Kronheimer and Mrowka.
The framed instanton homology $\ih^\#(\y)$ is 
isomorphic to the sutured instanton group $\text{SHI}(M,\gamma)$ 
introduced in \cite{kms}, where $M$ is the complement of an open 3-ball in $\y$ and 
$\gamma$ is a suture on the 2-sphere boundary. 
Restating a conjecture of Kronheimer and Mrowka, transferred from the sutured setting, 
cf. \cite[Conj. 7.24]{kms}, it is expected that
\[
	\hf(\y;\mathbb{C})\simeq \ih^\#(\y;\mathbb{C})\simeq \hm(\y;\mathbb{C}).
\]
Theorem \ref{thm:1} was inspired by this conjecture, 
and offers some validation for it, as does Corollary \ref{cor:2}. 
Note that Corollary \ref{cor:euler} confirms that 
the Euler characteristics of the three Floer groups 
under consideration agree.

Recently, Kronheimer and Mrowka \cite{kmu} introduced singular
instanton homology groups. We only mention the groups $\ih^\#(\y,L)$, 
where $L$ is a link in $\y$, and from which $\ih^\#(\y)$ is obtained
by taking $L$ to be empty. The construction of this group involves counting instantons 
on $\mathbb{R}\times\ybb^\#$ singular along $\mathbb{R}\times L$. Writing $\ih^\#(L)=\ih^\#(S^3,L)$,
Kronheimer and Mrowka produced a spectral sequence
\[
\text{Kh}(\lt) \;\;\leadsto \;\; \ih^\#(L),
\]
where the left side is unreduced Khovanov homology. 
Their spectral sequence respects $\mathbb{Z}/4$-gradings. 
Using related singular instanton groups, 
they proved in \cite{kmu} that Khovanov homology detects the unknot. 
This paper adapts many details from the aforementioned paper.\\

\subsection{Outline \& Acknowledgments}
Many of the proofs in this paper are partially adapted from 
the papers so far mentioned, especially \cite{kmu}, and 
are credited throughout. For example, the proof of the exact triangle 
in \S 5 is inspired mostly by \cite{kmu}, although
the analysis of instanton counts is somewhat different.

To work out the signs in the differential 
of the spectral sequence in Theorem \ref{thm:1},
we introduce an algebro-topological way of composing homology orientations 
in \S \ref{sec:homor}. 
The composition rule is associative and has distinguished units. 
Indeed, homology orientations are part of the morphism data in 
an appropriate category on which framed instanton homology is a functor.

The proof of Theorem \ref{thm:integerhom} follows from versions 
of Fukaya's connected sum theorem of \cite{fukaya} where non-trivial admissible 
bundles are involved, which to the author's knowledge have not appeared 
in the literature. These versions are simpler than Fukaya's original theorem, 
as there are fewer trivial connections with which to deal. 
We outline Donaldson's proof of Fukaya's theorem from \cite[\S 7.4]{d},
more or less verbatim (except for notation), and show how the versions 
of interest to us follow, making the necessary modifications.

The structure of this paper is as follows. In \S 2, 
we outline the proof of Theorem \ref{thm:1}, 
introducing the surgery exact triangle and the link 
surgeries spectral sequence. In \S 3,
we construct the bundles that are relevant to 
the surgery exact triangle. In \S 4,
we review the construction of instanton homology 
for admissible bundles. 
In \S 5, we prove Floer's exact triangle, and 
in \S 6, we prove the link surgeries spectral sequence. 
In \S 7, we define framed instanton homology and 
discuss its basic properties.
In \S 8, we define reduced odd Khovanov homology and 
complete the proof of Theorem \ref{thm:1} and Corollary \ref{cor:2}.
In \S 9, we prove Theorem \ref{thm:integerhom} and Corollaries \ref{cor:3}, \ref{cor:4} and \ref{cor:5}. 
In \S 10, we prove Corollary \ref{cor:euler}.

The author would like to thank Ciprian Manolescu 
for suggesting the problem that lead to the results in this paper, 
and for his encouragement and enthusiasm. The author thanks Tye Lidman for helpful conversations, 
and Jianfeng Lin for providing an argument given in \S \ref{sec:main} and for helpful discussions regarding \S \ref{sec:homor}. The author has learned that some of the results in this paper 
have been obtained independently by Aliakbar Daemi.

%% file: main.tex
\section{The Surgery Story}\label{sec:main}

In this section we outline the proof of Theorem \ref{thm:1}. 
To begin, we recall the instanton surgery exact sequence, or exact triangle, 
introduced by Floer \cite{f2}.
Let $K$ be a framed knot in $\y$. That is, $K$ has a preferred meridian 
and longitude. Let $\omega$ be a geometric representative for $\ybb$ disjoint from 
$K$. Denote by $\y_0$ and $\y_1$ the results of performing $0$- and 
$1$-surgery on $K$, respectively. We can view $\omega$ inside $\y_0$ and 
$\y_1$ by keeping it away from the surgery neighborhood. Let $\omega_0=\omega\cup\mu\subset \y_0$ 
where $\mu$ is a core for the induced framed knot in $\y_0$,
and let $\omega_1=\omega\subset\y_1$. Finally, for $i=0,1$ choose a 
bundle $\ybb_i$ over $\y_i$ geometrically represented by $\omega_i$. If the 
ordered triple of bundles 
$\ybb,\ybb_0,\ybb_1$ can be geometrically represented in this way, 
we say they form a \textit{surgery triad}.

\begin{theorem}[Floer]\label{thm:floer} There is an exact sequence
\[
\cdots \ih(\ybb) \to \ih(\ybb_0)\to \ih(\ybb_1)\to \ih(\ybb)\cdots
\]
provided all three bundles are admissible and form a surgery triad.
\end{theorem}

\begin{figure}[t]
\includegraphics[scale=.60]{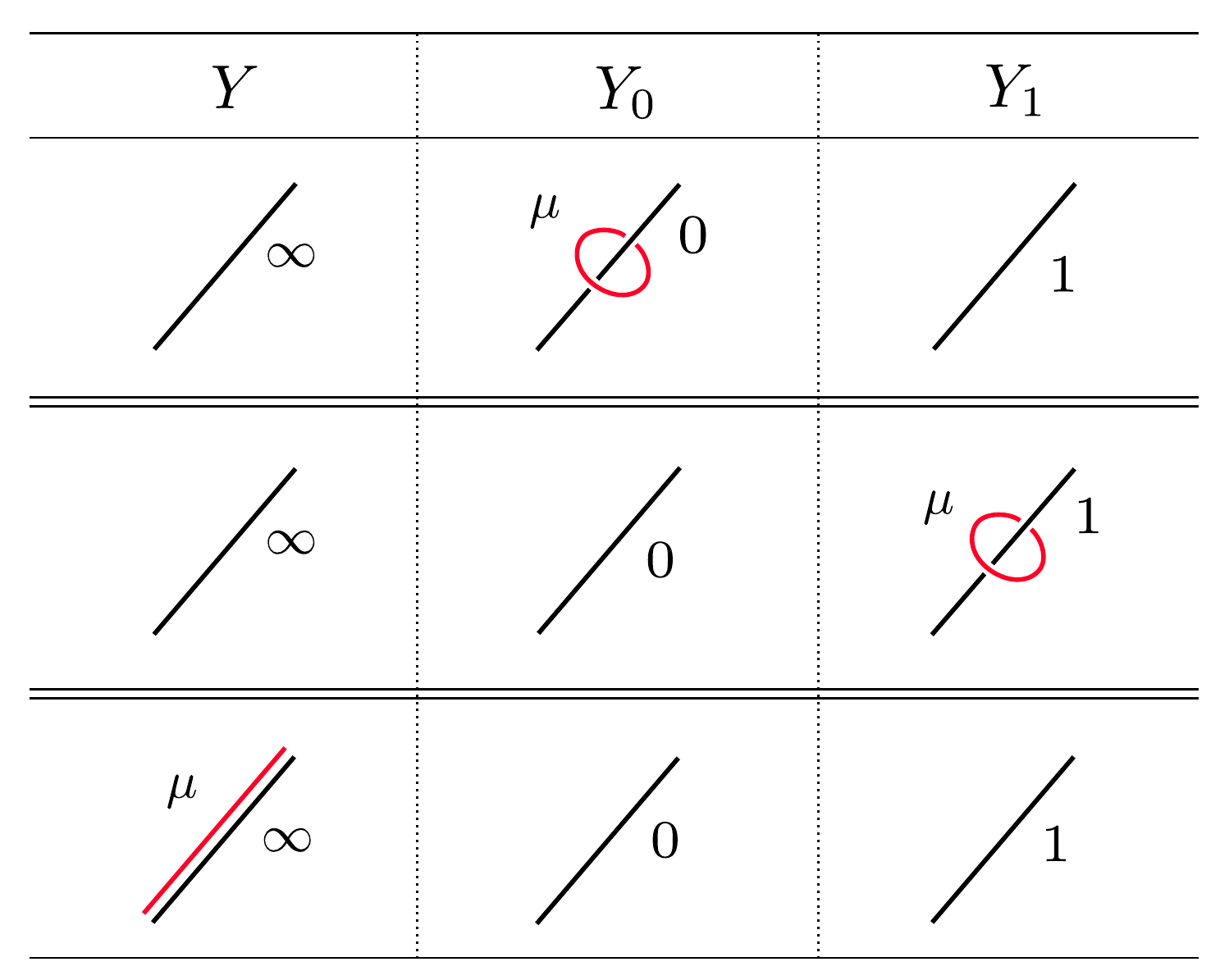}
\caption{Local surgery diagrams. The slanted line in each case is the knot $K$. 
Each row represents a possible construction for a surgery triad.}
\label{fig:ses}
\end{figure}

\noindent The loop $\mu$ in $\y_0$, pushed out of the surgery solid torus, becomes 
a small meridional loop around the surgered neighborhood of $K$ in $\y_0$. This is depicted 
in the top row of Figure \ref{fig:ses} in a local surgery diagram for $\y_0$.
One can view $\y_1$ (resp. $\y$) as obtained from $0$-surgery on the induced framed 
knot in $\y_0$ (resp. $\y_1$), see \S \ref{sec:topology}. 
Thus we obtain two more local surgery diagram depictions of where 
$\mu$ may be placed,
listed in the bottom two rows of Figure \ref{fig:ses}. See also \S \ref{sec:geomrep}.

Floer's exact triangle was studied by Braam and Donaldson in \cite{bd}, where a 
detailed proof following Floer's ideas can be found. In this paper we provide 
an alternative proof. 
The proof relies on an 
algebraic lemma which was first used by Ozsv\'ath and Szab\'o \cite{os}. 
The lemma requires the input of maps between the three 
relevant chain complexes satisfying certain properties. The maps we choose count 
instantons on families of metrics that are parameterized by convex polytopes. 
This approach was used by Kronheimer, Mrowka, Ozsv\'ath and Szab\'o to prove a 
surgery exact sequence in the monopole case \cite{kmos}. 
Our proof is largely an adaptation of Kronheimer and Mrowka's proof in \cite{kmu}
of an analogous exact triangle in singular instanton knot homology.
 
This method of proof leads to a generalization of Floer's theorem to a so-called 
link surgeries spectral sequence, as was first done by Ozsv\'ath and Szab\'o in Heegaard 
Floer homology \cite{os}. Let $\lt$ be a framed link in $\y$ with components 
$\lt_1,\ldots,\lt_m$. For each $v\in\{0,1,\infty\}^m$ let $\y_v$ be the result of
$v_i$ surgery on $\lt_i$ for $1\leq i \leq m$. Briefly, we say $\y_v$ is the result 
of $v$-surgery on $\lt$. Choose a geometric representative $\omega$ for $\ybb$ disjoint 
from $\lt$. Let $\omega_v\subset\y_v$ be $\omega$ together 
with a core for the knot in $\y_v$ induced by $L_i$ for each $i$ with $v_i=0$. Let $\ybb_v$ be 
bundles over $\y_v$ geometrically represented by the $\omega_v$. If the bundles $\ybb_v$ 
can be geometrically represented according to these rules we say that they form a 
\textit{surgery cube}.

\begin{theorem}\label{thm:2}
	Suppose the bundles {\em $\ybb_v$} for $v\in\{0,1,\infty\}^m$ are admissible and that they form a surgery cube. 
	Then there is a spectral sequence
\[
	\bigoplus_{v\in\{0,1\}^m}\ih(\ybb_v) \quad \leadsto \quad \ih(\ybb).
\]
That is, the left side is the {$\e^1$}-page and the sequence converges to the right side.
\end{theorem}

\noindent A more detailed statement is provided in Theorem \ref{ss1}. 
An analogous result in monopole Floer homology was proved by Bloom \cite{bloom} with $\mathbb{F}_2$-coefficients, 
and in singular instanton knot homology by Kronheimer and Mrowka \cite{kmu}.

From this we obtain a surgery spectral sequence for the groups $\ih^\#(\y)$, 
which generally must involve the twisted groups $\ih^\#(\y;\lambda)$.
The group $\ih^\#(\y;\lambda)$ is four consecutive gradings of $\ih(\ybb^\#)$, 
where $\ybb^\#$ is a bundle over $\y\# T^3$ geometrically 
represented by an $S^1$-factor of $T^3$ together with $\lambda\subset \y$. 
In this setting, the surgeries on the link $L$ are performed 
away from the 3-tori, and every bundle is automatically admissible.
To minimize the number of non-trivial bundles in the mix, 
we refer to the bottom row of Figure \ref{fig:ses}.
Using this, we can ensure that 
the bundles for $v$-surgeries with $v\in\{0,1\}^m$ are 
geometrically represented only by 
the $S^1$-factor of $T^3$. The trade-off is that the geometric representative 
for the bundle over $\y$ is the $S^1$-factor together with the link $L$. 
We obtain

\begin{theorem}\label{thm:4}
Let $\lt$ be a framed $m$-component link in $\y$. There is a spectral sequence
\[
	\bigoplus_{v\in\{0,1\}^m}\ih^\#(\y_v) \quad \leadsto \quad  \ih^{\#}(\y;L).
\]
That is, the left side is the {$\e^1$}-page and the sequence converges to the right side.
\end{theorem}

\noindent A more detailed statement is given in Theorem \ref{thm:framed}.

Now we introduce branched double covers, following Ozsv\'ath and Szab\'o \cite{os}.
Let $\lt$ be a link in $S^3$ and 
$\Sigma(\lt)$ the double cover of $S^3$ branched over $\lt$. Let $D$ 
be a planar diagram for $\lt$ with $m$ crossings. For each $v\in\{0,1\}^m$ there 
is a resolution diagram $D_v$ which is a disjoint union of circles, obtained by 
performing $0$- and $1$-resolutions according to Figure \ref{fig:crossing_res}.
Each branched cover $\Sigma(D_v)$ is diffeomorphic to $\#^k S^1\times S^2$ 
where $D_v$ has $k+1$ circles. Further, there is a link 
$\lt'\subset\overline{\Sigma(\lt)}$ and a framing on $\lt'$ such that 
$\Sigma(D_v)$ is the result of $v$-surgery on $\lt'$. If we draw a small arc
between each crossing in $D$, the preimages in 
the branched cover $\Sigma(\lt)$ are loops, and the link $\lt'$ is the union 
of these preimages.

With this setup, from Theorem \ref{thm:4} we have a spectral sequence 
\begin{equation}
	\bigoplus_{v\in\{0,1\}^m}\ih^\#(\Sigma(\text{D}_v))
	\quad \leadsto \quad \ih^\#(\overline{\Sigma(\lt)};L').\label{sps}
\end{equation}
We claim that 
$[\lt']\in H_1(\Sigma(\lt);\mathbb{F}_2)$ is zero, so that 
the target of this spectral sequence is in fact $\ih^\#(\overline{\Sigma(\lt)})$.
The diagram $D$ divides the plane into regions. To show $[\lt']=0$, it suffices 
to color the regions black and white in a way such that each crossing touches exactly 
one black region. See Figure \ref{fig:orientedres}. 
For then the black regions can be lifted to a surface in $\Sigma(\lt)$ 
whose boundary is $\lt'$, implying $[\lt']=0$.

To color the regions, we follow an argument communicated to the author by Jianfeng Lin. 
We proceed as if performing the algorithm to construct a Seifert surface, as in 
\cite[\S 5.4]{rolfsen}. First, we orient $\lt$. Then we resolve each crossing as in 
Figure \ref{fig:orientedres}. 
We assign to each circle $z$ in the resolved diagram two signs, $a_z$ and $b_z$. The first 
sign $a_z$ is $+1$ if $z$ is oriented counter-clockwise in the plane, and $-1$ otherwise. 
The second sign $b_z$ is given by $(-1)^N$ where $N$ is the number of circles that 
surround $z$. Now color, with black, 
the regions that are directly interior to each circle $z$ with $a_zb_z=+1$. Transferring the coloring 
back to the unresolved diagram, each crossing touches exactly one such region.

\begin{figure}[t]
\includegraphics[scale=.75]{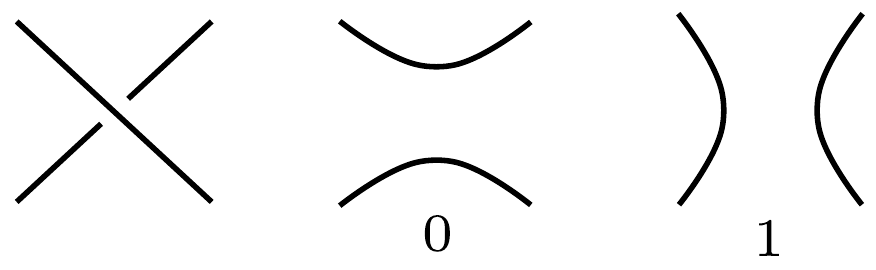}
\caption{From the diagram $\text{D}$ to a resolution diagram $\text{D}_v$.}
\label{fig:crossing_res}
\end{figure}

\begin{figure}[t]
\includegraphics[scale=.70]{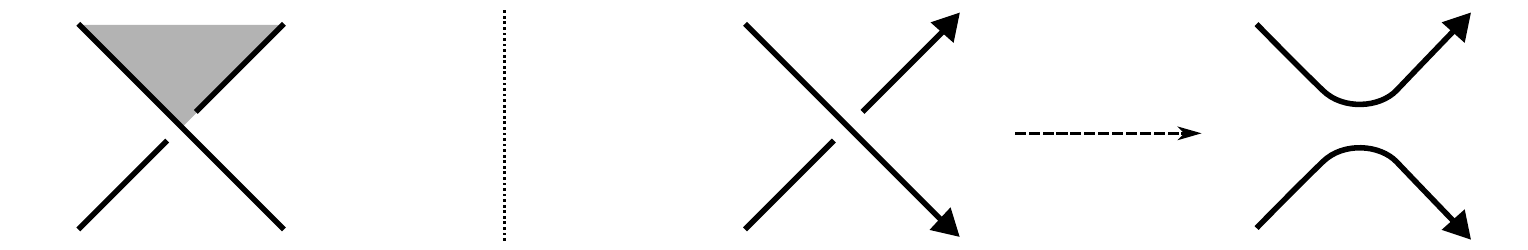}
\caption{On the left, we want to color the regions of the diagram so that at each crossing exactly one of the 
four regions is colored. On the right, we go from an oriented diagram to a disjoint union of oriented circles.}
\label{fig:orientedres}
\end{figure}

This reduces the proof of Theorem \ref{thm:1} to identifying the $E^1$-page of (\ref{sps})
and then understanding the gradings.
We can compute the groups $\ih^\#(\Sigma(D_v))$ because each $\Sigma(D_v)$ 
is of the form $\#^k S^1\times S^2$, and we can compute the $E^1$-differential 
because the cobordism maps involved are topologically simple. 
This is carried out in \S \ref{sec:branched}, where we identify the $E^1$-page 
as the chain complex used to compute $\kh(\lt)$ from the diagram $D$. 
We then check in \S \ref{sec:finalgr} that the relevant gradings are preserved, 
completing the proof of Theorem \ref{thm:1}.

%% file: topology.tex
\section{Bundles in the Exact Triangle}\label{sec:topology}

In this section we introduce the manifolds and bundles that feature 
in the proof of Theorem \ref{thm:floer}. We take a systematic approach to the 
bundles $\ybb_i$ that appear in Floer's exact triangle by 
extending Dehn surgery to $\so$-bundles. This viewpoint was Floer's \cite{f2}, 
and is expanded upon in \cite{bd}. The construction of surgery cobordism bundles 
$\xbb_{ij}$ in \S \ref{sec:x} is straightforward in this setting. 
These bundles induce the maps in the exact triangle. We then introduce some 
hypersurfaces in $\x_{ij}$ that yield useful metric families; 
these were used in \cite{kmos,bloom,kmu}. In \S \ref{sec:geomrep}
we relate our new setup to that of the statement of Theorem \ref{thm:floer} in \S \ref{sec:main}.

In this section, we write $A\cup_f B$ for the space obtained from the disjoint union 
of $A$ and $B$, with points identified using the map $f$. Our convention is that the
gluing map $f$ is always from a subset of $B$ to a subset of $A$. We freely
use isomorphisms of the form $A\cup_f B\simeq A\cup_{fg} C$, where
$g$ is an isomorphism from a subset of $C$ to a subset of $B$. All constructions
that are not smooth have a canonical smoothing, as mentioned in \cite[Rmk. 1.3.3]{gs}. 
All (principal) $\so$-bundles have right actions. Thus our bundle gluing maps, 
in order to be equivariant, always involve left multiplication on trivialized fibers.\\

\subsection{Dehn Surgery with Bundles}\label{sec:dehn}

Let $\ybb$ be an $\so$-bundle over a closed, oriented 3-manifold $\y$. Let 
$\knot:S^1\times D^2\to \y$ be an embedding. We refer to $\knot$ as a framed 
knot in $\y$. We consider equivariant embeddings $\kbb:S^1\times D^2\times \so\to\ybb$ 
that lie above $\knot$, i.e. $\kbb/\so=\knot$. We refer to $\kbb$ as a framed knot 
in $\ybb$. 
The space of bundle automorphisms of $S^1\times D^2\times\so$ fixing the base space 
has two connected components. 
An automorphism $\tau$ not isotopic to the identity is
\[
	\tau(w,z,a)=(w,z,c(w)a)
\]
where $(w,z)\in S^1\times D^2$, $a\in\so$, and 
$c$ is a standard inclusion $S^1\to \so$ of a maximal torus. In particular, $c$ 
is a homomorphism and generates $\pi_1(\so)\simeq \mathbb{Z}/2$. 
If $\kbb$ is one embedding, another embedding lying above $\knot$ is given by 
$\kbb\tau$.

We generalize Dehn surgery to surgery on the framed knots $\kbb$. For 
$\Omega = (A,b)\in \text{SL}(2,\mathbb{Z})\ltimes (\mathbb{Z}/2)^2$ 
we define an automorphism $\psi_\Omega$ of $S^1\times\partial D^2\times \so$ by
\[
	\psi_\Omega(w,z,a) = (w^{A_{11}} z^{A_{12}},w^{A_{21}} z^{A_{22}}, c(w)^{b_1}c(z)^{b_2}a).
\]
Let $\kbb'$ be the interior of the image of $\kbb$. The result of $\Omega$-surgery on $\kbb$ 
is then defined to be the identification space
\[
	\ybb_\Omega(\kbb) = (\ybb\setminus \kbb')\cup_{\kbb\psi_\Omega} (S^1\times D^2\times\so).
\]
There is an induced framed knot $\Omega(\kbb)$ in $\ybb_\Omega(\kbb)$ given by the inclusion 
of $S^1\times D^2\times\so$ into the above expression for $\ybb_\Omega(\kbb)$. 
The product of elements in $G=\text{SL}(2,\mathbb{Z})\ltimes (\mathbb{Z}/2)^2$ is given by 
$(A',b')(A,b) = (A'A, b' A + b )$. The assignment $\Omega\mapsto \psi_\Omega$ 
induces an isomorphism from $G$ to the group of isotopy classes of orientation preserving 
equivariant automorphisms of $S^1\times \partial D^2 \times \text{SO}(3)$. 
We have an associativity rule 
\[
	\ybb_{\Omega'\Omega}(\kbb)  \simeq (\ybb_{\Omega'}(\kbb))_{\Omega}(\Omega'(\kbb)).
\]
The space $\ybb_\Omega(\kbb)$ is naturally a bundle over $\y_{p/q}(K)$, the result of 
$p/q$ Dehn surgery on the framed knot $K$ in $\y$, where $\Omega=(A,b)$, $p=A_{22}$, 
$q=A_{12}$ and of course $K=\kbb/\so$. Note that 
the automorphism $\tau$ above restricts to $\psi_\Theta$ where 
$\Theta=(1_{2\times 2},(1,0))\in G$. We have the transformation rule 
$\ybb_\Omega(\kbb\tau)\simeq \ybb_{\Theta\Omega}(\kbb)$.\\

\subsection{The Surgery Bundle $\ybb_i$} There is a particular choice of surgery parameter 
$\Omega$ that Floer used in the setting of his exact triangle:
\begin{equation}\label{lambda}
	\Lambda =  \left(\left[\begin{array}{cc} -1 & 1 \\ -1& 0 \end{array}\right],(1,0)  \right).
\end{equation}
To understand this, write $\Lambda = \Psi\Lambda'$, where
\begin{equation}
	\Psi = (1_{2\times 2},(0,1)), \quad \Lambda'=\left(\left[\begin{array}{cc} -1 & 1 \\ -1& 0 \end{array}\right],(0,0)  \right).\label{eq:psi}
\end{equation}
First, $\Psi$ twists the trivialization around $\partial D^2$. Then, 
$\Lambda'$ performs $0$-surgery on $K$, leaving bundles alone.
Note that $\Lambda^3=1$. With $\ybb$ and $\kbb$ 
fixed, we define for $i\in\mathbb{Z}$ the surgery bundles $\ybb_i = \ybb_{\Lambda^{i+1}}(\kbb)$, 
the surgery base manifolds $\y_i =\ybb_i/\so$, and the induced embeddings 
$\kbb_{i} = \Lambda^{i+1}(\kbb)$. The index offset is here so that $\y_0$ and $\y_1$ are 
simply $0$- and $1$-surgery on $\knot\subset\y$, respectively. Because $\Lambda^3=1$, 
there are isomorphisms $\ybb_i \simeq \ybb_{i+3}$.\\

\subsection{The Surgery Cobordism $\xbb_{ij}$}\label{sec:x} Our goal is to construct 
cobordism bundles $\xbb_{ij}:\ybb_i\to \ybb_j$ for $i< j$. Each $\xbb_{ij}$ will be 
an $\so$-bundle over a standard surgery cobordism $\x_{ij}:\y_i\to \y_j$. We first 
construct $\xbb_{ij}$ when $j-i=1$ and use these as building blocks for the general 
construction. Write $\hbb=D^2\times D^2\times \so$. We view $\hbb$ as a 2-handle thickened 
by $\so$. Also write

\begin{figure}[t]
\includegraphics[scale=.90]{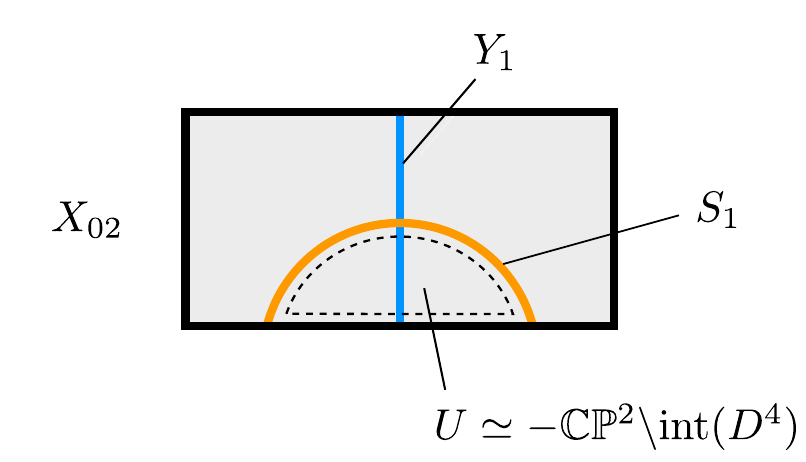}
\caption{The two hypersurfaces $\y_1$ and $S_1$ in the interior of $\x_{02}$. The 
3-sphere $S_1$ separates off a copy of $-\mathbb{C}\mathbb{P}^2$ minus a 4-ball.}
\label{fig:met13}
\end{figure}

\begin{center}
	$\partial \hbb = \hbb_1 \cup \hbb_2$,\\
	$\hbb_1 = \partial D^2\times D^2\times \so$,\\
	$\hbb_2 =  D^2\times \partial D^2\times \so$.\\
\end{center}
Viewing $\kbb_0$ as a map $\hbb_1 \to \{1\}\times \ybb_0$, we define $\xbb_{01}$ by setting
\[
	\xbb_{01} =([0,1] \times \ybb_0) \cup_{\kbb_0} \hbb.
\]
The definition of $\xbb_{ij}$ for general $j-i=1$ is similar. We want to define $\xbb_{02}$ as 
$\xbb_{01}\cup_{\ybb_1}\xbb_{12}$. To make sense of this expression we give an explicit 
identification of $\partial \xbb_{01}\setminus \ybb_0$ with $\ybb_1$. Let the interior of 
the image of $\kbb_0$ in $\ybb_0$ be denoted $\kbb'_0$. Note that
\[
	\partial\hbb_1=\hbb_1\cap\hbb_2=\partial\hbb_2
\]
is a trivial bundle over a 2-torus. Now we write
\[
	\partial \xbb_{01}\setminus \ybb_0 =  (\ybb_0 \setminus \kbb_0') \cup_{\kbb_0|{\hbb_1\cap\hbb_2}} \hbb_2.
\]
Let $\psi:\hbb_1\to \hbb_2$ be an isomorphism. Then
\[
	\partial \xbb_{01}\setminus \ybb_0 \simeq  (\ybb_0 \setminus \kbb'_0) 
	\cup_{\kbb_0\psi|{\hbb_1\cap\hbb_2}} \hbb_1 = (\ybb_0)_{\psi|{\hbb_1\cap\hbb_2}}(\kbb_0).
\]
To identify this bundle with $\ybb_1=(\ybb_0)_\Lambda(\kbb_0)$ we need $\psi$ such that 
$\psi|_{\hbb_1\cap\hbb_2}=\psi_\Lambda$; we choose
\[
	\psi:\hbb_1\to\hbb_2, \quad \psi(w,z,a) := (\overline{w}z,\overline{w},c(w)a).
\]
Making this choice, we have identified $\partial \xbb_{01}\setminus \ybb_0$ 
with $\ybb_1$. Finally, to construct $\xbb_{ij}$ for $j-i>1$,
we inductively define $\xbb_{ij}=\xbb_{i,j-1}\cup_{\ybb_{j-1}}\xbb_{j-1,j}$,
where the gluing is done according to the same identification process.\\

\subsection{The Bundle $\sbb_i$}\label{sec:decomp1} We construct a subset 
$\sbb_1\subset \xbb_{02}$ which is a bundle over a 3-sphere $\s_1\subset\x_{02}$. One 
gets $\sbb_i$ inside $\xbb_{i-1,i+1}$ for each $i$ in a similar fashion. Write
\begin{equation}
	\xbb_{02}= ([0,1]\times \ybb_0 \cup_{\kbb_0} \hbb) \cup_{\ybb_{1}} ([0,1]\times \ybb_1\cup_{\kbb_1} \hbb) \simeq ([0,1]\times \ybb_0) \cup_{\kbb_0} \hbb \cup_{\psi} \hbb\label{iso}
\end{equation}
with notation as in the construction of $\xbb_{01}$. Introduce the subset
\[
	\hbb(r,s) = D^2(r)\times D^2(s)\times\so \subset \hbb
\]
where $D^2(r)$ is the disk of radius $r$, $0<r\leq 1$, and consider the following restriction bundles of $\xbb_{02}$:
\[
	\ubb=\hbb(1/2,1) \cup \hbb(1,1/2)\subset\hbb\cup_\psi\hbb, \quad \sbb_1 = \partial\ubb.
\]
It is well-known that the base space $U$ of $\ubb$ 
is diffeomorphic to $-\mathbb{C}\mathbb{P}^2$ minus 
an embedded 4-ball, cf. \cite[Ex. 4.2.4]{gs}. It follows that $\sbb_1$
is a trivial bundle over a 3-sphere $\s_1$. We see that we can decompose $\x_{02}$ along $\s_1$ into a 
connected sum of $-\mathbb{C}\mathbb{P}^2$ with a manifold whose boundary is $\y_2 \sqcup \overline{\y}_0$. 
The intersection $\sone\cap\y_1$ is 2-torus. This decomposition is depicted in Figure \ref{fig:met13}.

We claim that $\ubb$ is a non-trivial bundle. 
We check that the restriction of $\ubb$ to an essential 
sphere is non-trivial. Define $\mathbb{D}_1 =  D^2\times \{0\} \times \so$ and 
$\mathbb{D}_2 =  \{0\}\times D^2\times \so$ as subsets of $\hbb$. Consider
\[
	\mathbb{D}_2 \cup_{\psi|{\partial \mathbb{D}_1}} \mathbb{D}_1\subset\ubb. 
\]
This is 
isomorphic to $D^2\times \so\cup_{f} D^2\times \so$ where $f$ is the automorphism 
of $\partial D^2\times\so$ given by $f(z,a)=(\overline{z},c(z)a)$. This 
is a nontrivial bundle over a 2-sphere.\\

\begin{figure}[t]
\includegraphics[scale=.75]{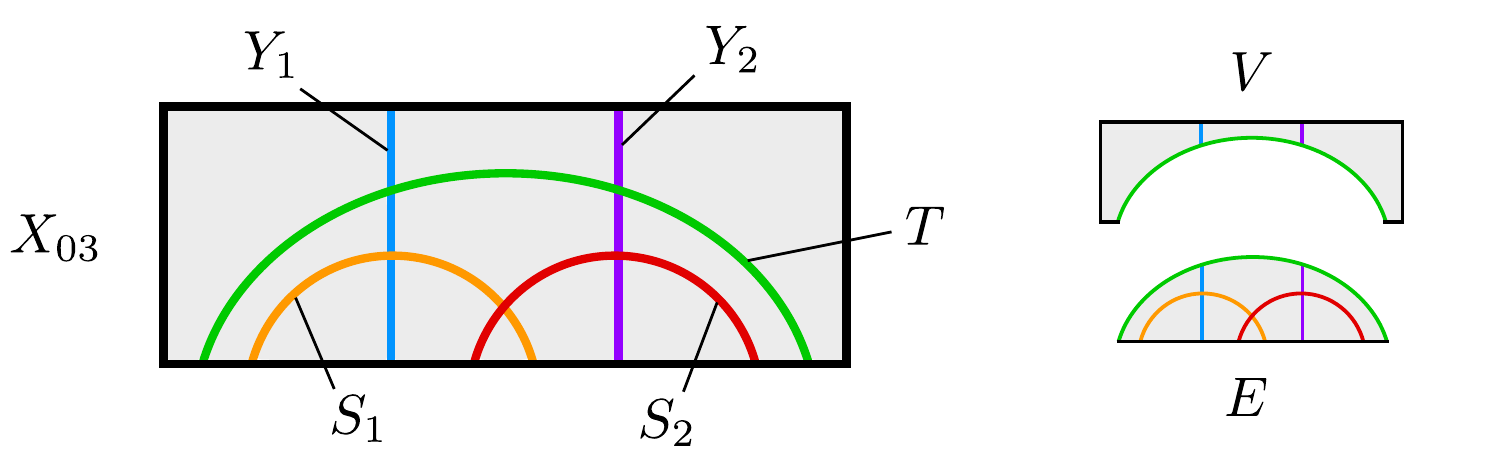}
\caption{The intersections of the five hypersurfaces in the interior of $\x_{03}$. 
The $S^1\times S^2$ hypersurface $T$ divides $\x_{03}$ into two pieces, $V$ and $E$. This picture first appeared in \protect\cite{kmos}.}
\label{fig:met3}
\end{figure}

\subsection{The Bundle $\tbb$}\label{sec:decomp2} We construct a subset 
$\tbb\subset\xbb_{03}$ which is a trivial bundle over $T\subset\x_{03}$ where 
$T$ is diffeomorphic to $S^1\times S^2$. By iterating (\ref{iso}) and 
stretching the ends we write $\xbb_{03}$ as
\[
	\xbb_{03} \simeq  
	([0,1]\times \ybb_0)\cup_{\kbb_0} \hbb \cup_{\psi} \hbb 
	\cup_{\psi} \hbb\cup_{\ybb_0} ([0,1] \times \ybb_0).
\]
Identifying $\xbb_{03}$ with the expression on the right, we define the restriction bundles
\[
	\ebb = \hbb \cup_{\psi} \hbb \cup_{\psi} \hbb, \quad \tbb = \partial \ebb,
\]
and their respective base spaces $E$ and $T$. We have an isomorphism
\begin{equation}
f:\sthreebb = \hbb_1\cup_{(\psi|_{\hbb_1\cap\hbb_2})^2}\hbb_2 \to S^1\times S^2\times \so\label{triv}
\end{equation}
where, viewing $S^2\subset S^1\times S^2$ as $\mathbb{C}\cup\infty$, we set
\begin{equation*}
f|_{\hbb_1}=\text{id}, \quad f|_{\hbb_2}(z,w,a) = (\overline{w},\overline{w}/\overline{z},c(w)a).
\end{equation*}
The triviality of the bundle $\sthreebb$ is also seen 
from the observation that it is the restriction of a 
bundle on a space in which $\sthree$ is contractible. We note that 
we could have also trivialized $\sthreebb$ 
by using a similar isomorphism in which $f|_{\hbb_2}=\text{id}$. 
These two isomorphisms determine trivializations 
that differ, in the terminology of \S \ref{sec:instantons}, 
by a non-even gauge transformation.

We remark that the 
intersections $\sthree\cap\y_1$ and $\sthree\cap\y_2$ are 2-tori. We illustrate the 
arrangement of intersections in Figure \ref{fig:met3}. We note that $T$ may 
be described as the boundary of a regular neighborhood of the union of the two 
essential spheres inside the copies of $-\mathbb{C}\mathbb{P}^2$ divided off by 
$\sone$ and $\stwo$. The hypersurface $T$ separates $\x_{03}$ into two 4-manifolds, 
$E$ and $V$, where $E$ is diffeomorphic to $-\mathbb{C}\mathbb{P}^2$ 
minus a neighborhood of an unknotted circle, and $V$ is diffeomorphic to $[0,1] \times \y_0$ minus a 
neighborhood of $\{1/2\}\times \knot$.\\

\subsection{An Involution of $\mathbb{E}$}\label{sec:involution}

We construct an involution $\sigma:\eothreebb\to\eothreebb$. 
We write
\[
 \ebb = \mathbb{H}^{-1}\cup_\psi \mathbb{H}^0\cup_\psi \mathbb{H}^{+1}
\]
where the superscripts have been added to distinguish the copies of $\mathbb{H}$. 
We write $[w,z,a,i]\in\ebb$ for the point represented by $(w,z,a)\in\mathbb{H}^i$. 
We define our involution by
\[
	\sigma[w,z,a,i] = [\overline{z},\overline{w},c(w)^{i(i+1)}c(z)^{i(i-1)}a,-i].
\]
Here we have extended $c^2:S^1\to\so$ to a map 
$c^2:D^2\to\so$ such that $c^2(\overline{w})=(c^2(w))^{-1}$. Note that $\sigma$ interchanges the outer copies of $\mathbb{H}$ and 
fixes the middle copy of $\mathbb{H}$. 
It is straightforward that $\sigma$ is well-defined: writing $\sigma$ 
as three maps $\sigma_i:\hbb^{+i}\to\hbb^{-i}$, one uses the relations
\[
	\sigma_0^2=\text{id}, \quad \sigma_{\pm 1} = \psi^{\pm 1}\sigma_0\psi^{\pm 1}, \quad \psi^3 = \text{id},
\]
whenever these compositions are defined.
The involution $\sigma$ is a bundle automorphism that restricts to an orientation-preserving 
diffeomorphism of $E$. It fixes $\tbb$ and swaps $\mathbb{S}_1$ with $\mathbb{S}_2$.

\begin{figure}[t]
\includegraphics[scale=.65]{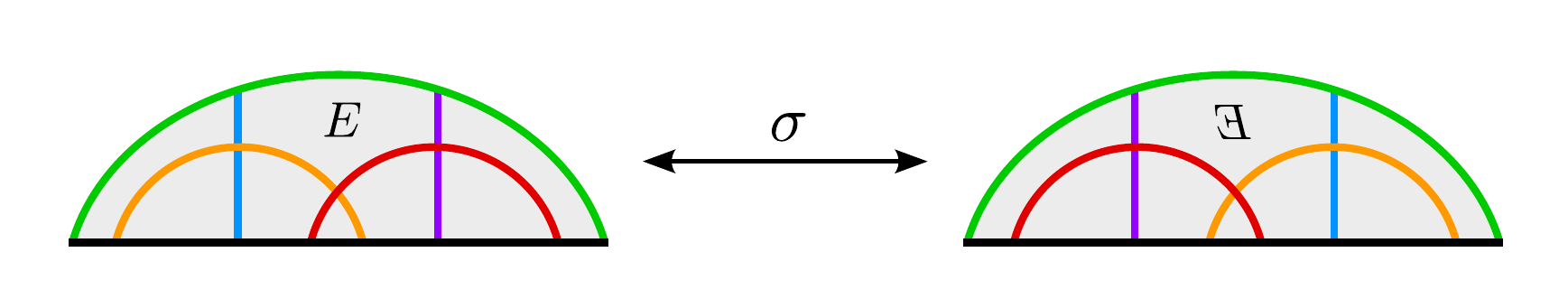}
\caption{The involution $\sigma$.}
\label{fig:involution}
\end{figure}

Let us look at how the involution affects $\mathbb{T}$. Recall the isomorphism (\ref{triv}). 
We have
\begin{equation}
	f\sigma f^{-1}(w,z,a) = (w,w/z,c(\overline{w})d(z)a)\label{eq:invdiff}
\end{equation}
where $w\in S^1, z\in\mathbb{C}\cup\infty$, $a\in\so$ 
and $d:S^2\to\so$ is the double of $c^2:D^2\to\so$. 
It is easily seen that $f\sigma f^{-1}$ is isotopic to a composition, $\theta\circ\upsilon$, where
\[
	\theta(w,z,a)= (w,w/z,a), \quad \upsilon(w,z,a) = (w,z,c(w)a).
\]
The map $\theta$ is a diffeomorphism of $S^1\times S^2$ that, with respect to our trivialization, is extended in trivial way to the overlying bundle.
In the terminology of \S \ref{sec:instantons}, $\upsilon$ is a non-even gauge transformation of the trivial bundle over $S^1\times S^2$. The involution $\sigma$ will be useful in the proof of the exact triangle.\\

\subsection{Geometric Representatives}\label{sec:geomrep}

Let $\omega$ be an embedded loop in $\y$. Extend this to an embedding 
$\kbb_\omega:S^1\times D^2\times\so\to\y\times\so$. Let $\Psi=(1_{2\times 2},(0,1))$ 
as in (\ref{eq:psi}). Then the result of 
$\Psi$-surgery on $\kbb_\omega$ as a framed knot in $\y\times\so$ is a bundle geometrically 
represented by $\omega$. More generally, $\omega$ can be a collection of embedded loops, 
and $\Psi$-surgery for each component gives a bundle geometrically represented by $\omega$.

This relates our current framework to the statement of Theorem \ref{thm:floer} in \S \ref{sec:main}. 
Let $\omega$ be a closed, unoriented 1-manifold in $\y$, and $\knot$ a framed knot in $\y$ 
disjoint from $\omega$. We set
\begin{equation}\label{geomrep}
	\ybb = (\y\times\so)_\Psi(\kbb_\omega),
\end{equation}
where it is understood that if $\omega$ has multiple components, we do $\Psi$-surgery for 
each component. This description of $\ybb$ gives a preferred trivialization away from a 
neighborhood of $\omega$. We let $\kbb$ be the $\so$-thickening of $\knot$ using this 
preferred data, precomposed with $\tau$. That is,
\[
	\kbb = (K\times\text{id}_{\so})\tau.
\]
Recall that $\tau$ restricts to $\psi_\Theta$ where 
$\Theta=(1_{2\times 2},(1,0))$, and that $\ybb_0$ is defined as 
$\ybb_\Lambda(\kbb)$. Using $\Theta\Lambda = \Lambda'\Psi$
with notation as in (\ref{eq:psi}), we have
\[
	\ybb_0 \simeq \ybb_{\Lambda'\Psi}(K\times\text{id}_{\so}).
\]
Because $\Lambda'$ is $0$-surgery without bundle-twisting, 
we see $\ybb_0$ is of the form (\ref{geomrep}),
where $\y$ is replaced by $\y_0$ and $\omega$ replaced by
$\omega\cup K_0$, where $K_0$ is the induced knot in $\y_0$.
Thus $\ybb_0$ is geometrically represented by $\omega\cup K_0$.
Pushing $K_0$ away from the surgered 
neighborhood makes it a small meridional loop $\mu$ as in Figure \ref{fig:ses}, by the 
nature of $0$-surgery.

We may deduce that $\ybb_1$ is geometrically represented by $\omega\subset\y_1$ by 
either of two ways. First, we may interpret $\y_1$ as $0$-surgery on the induced knot 
$\knot_0\subset\y_0$ and iterate the rule already established, forgetting about bundles 
altogether. Alternatively, we can 
repeat the above argument for $\Lambda^2$ in place of $\Lambda$. The difference in 
this case is that $\Theta\Lambda^2 = (\Lambda')^2$. This is $1$-surgery on $K$ without 
bundle-twisting.

%% file: instantons.tex
\section{Instanton Homology}\label{sec:instantons}

In this section we review the relevant aspects of instanton homology 
for admissible bundles. 
Our main technical references are \cite{d,kmu}. Other useful references 
include \cite{f1,f2,bd,froy,s}. In \S \ref{sec:index} 
we define the index $\mu$ that encodes 
the expected dimension of instanton moduli spaces. Section \S \ref{sec:met} 
is an adaptation of some results from \cite[\S 3.9]{kmu} regarding 
maps obtained from families of metrics on cobordisms.
In \S \ref{sec:indexbounds} we discuss how to use the index to put constraints 
on the existence of instantons, and in \S \ref{sec:instgrading} we discuss
the $\mathbb{Z}/2$-grading on $\ih(\ybb)$.\\

\subsection{Instanton Groups}\label{sec:instantongroups} 

Let $\mathbb{Y}$ be an $\so$-bundle over a closed, connected, oriented Riemannian 
3-manifold $\y$. The group $\ih(\ybb)$ is heuristically a Morse homology group 
computed using a suitably perturbed Chern-Simons functional $\textbf{cs}_\pi:
\mathscr{C}(\ybb)\to\mathbb{R}$ modulo a group of gauge transformations:
\[
	\textbf{cs}_\pi(a) = -\frac{1}{8\pi^2}\int_{[0,1]\times Y} \text{tr}(F_A^2) + f_\pi(a).
\]
Here $A$ is a connection on $[0,1]\times \ybb$ which restricts to a base connection $a_0$ on $\{0\}\times\ybb$ and the connection $a$ on $\{1\}\times \ybb$, and $f_\pi$ is a small perturbation, see \cite[\S 3.4]{kmu}.
We have written $\mathscr{C}(\ybb)$ for the space of smooth connections on $\ybb$, an 
affine space modelled on $\Omega^1(\ybb_\text{ad})$, where 
$\ybb_\text{ad}=\ybb\times_\text{ad}\mathfrak{so}(3)$ is the adjoint bundle 
of $\ybb$.

Let $\xbb$ be an $\so$-bundle over an $n$-dimensional manifold $\x$.
In our constructions we do not use the full automorphism group 
$\mathscr{G}(\xbb)$ of $\xbb$, but rather, 
following the terminology of \cite{froy}, we use the subgroup 
$\mathscr{G}_\text{ev}=\mathscr{G}_\text{ev}(\xbb)$ of
\emph{even} gauge transformations. Elements 
of $\mathscr{G}_\text{ev}$ are called determinant-1 gauge transformations 
in \cite{kmu} and restricted gauge transformations in \cite{bd}. Viewing 
gauge transformations as sections of the bundle $\xbb\times_{\text{Ad}}\so$, 
the even transformations are the ones that lift to sections of 
$\xbb\times_{\text{Ad}}\su$. There is an exact sequence 
\begin{equation}
	1 \longrightarrow \mathscr{G}_\text{ev}(\xbb) \longrightarrow \mathscr{G}(\xbb)
	\overset{\eta}{\longrightarrow} H^1(\x;\mathbb{F}_2)\longrightarrow 1, \label{eta}
\end{equation}
where $\eta$ measures the obstruction to deforming a gauge transformation over 
the 1-skeleton of $\x$. For a connection $\connA$ on $\xbb$ we write $h^0(\connA)$ 
for the dimension of its $\mathscr{G}_\text{ev}$-stabilizer. The possible values 
of $h^0(\connA)$ are $0,1,3$. For us, the only stabilizers that will appear will be 
$1,S^1,\so$.
We call $\connA$ {\em irreducible} if $h^0(\connA)=0$, and {\em reducible} otherwise.

We will write $\conna,\connb,\connc,\ldots$ for typical connections on bundles over 3-manifolds and 
$\connag,\connbg,\conncg,\ldots$ for their respective $\mathscr{G}_\text{ev}$-classes.
A typical connection on a bundle over a 4-manifold $\x$ is written as $\connA$, 
and simply $[\connA]$ for its $\mathscr{G}_\text{ev}$-class.

Let $\mathscr{B}(\ybb)$ denote the quotient $\mathscr{C}(\ybb)/\mathscr{G}_\text{ev}$. 
The functional $\textbf{cs}_\pi$ induces a map 
$\textbf{cs}'_\pi:\mathscr{B}(\ybb)\to \mathbb{R}/\mathbb{Z}$. 
The set of critical points of $\textbf{cs}'_\pi$ 
is denoted $\mathfrak{C}$ or $\mathfrak{C}(\ybb)$; 
when the perturbation $\pi$ is zero this is the set of flat connection classes on $\ybb$.
We write $h^1(\conna)$ for the dimension of the Zariski tangent space of $\connag$ in 
$\mathfrak{C}$. Following \cite{d}, when $h^0(\conna)=h^1(\conna)=0$, 
the connection $\conna$ is called \textit{acyclic}.
Let $\mathfrak{C}^{\text{irr}}$ denote the subset of irreducibles in $\mathfrak{C}$. 
When $\ybb$ is admissible and a suitable perturbation is chosen, $\mathfrak{C}^{\text{irr}}$ is a 
finite set of acyclic classes, 
and it is in fact all of $\mathfrak{C}$ or is missing only the trivial class, 
according to whether $b_1(\y)\neq 0$ or not. Assume such a perturbation is chosen. 

Fix a base connection $\conna_0$ on $\ybb$. We define the chain group 
\[
	\chain(\ybb) = \bigoplus_{\connag\in\mathfrak{C}^\text{irr}} \mathbb{Z}\Lambda(\connag)
\]
where $\Lambda(\connag)$ is the 2-element set of orientations of the real line $\text{det}(D_{\connA})$, 
where $\connA$ is a connection on $\mathbb{R}\times\ybb$ with $\connA|_{\ybb\times \{t\}}$ equivalent to 
$\conna_0$ for $t \ll 0$ and in the class $\connag$ for $t\gg 0$, and $D_{\connA}$ is the Fredholm 
operator $-d_A\oplus d_A^+$ defined on suitable Sobolev spaces in \S \ref{sec:index}; see also \cite[\S 3.6]{kmu}. Here $\mathbb{Z}\Lambda(\connag)$ means the 
infinite cyclic group with generators the elements of $\Lambda(\connag)$. We often think 
of $\chain(\ybb)$ as generated by $\mathfrak{C}^{\text{irr}}$; when doing this it is 
understood that we have chosen distinguished elements from each set $\Lambda(\connag)$.

A connection $\connA$ on an $\so$-bundle over a Riemannian 4-manifold is an \emph{instanton} or is 
\emph{anti-self-dual} (ASD) if its curvature $F_{\connA}$ satisfies 
\[
	\star F_{\connA}=-F_{\connA}
\]
where $\star$ is the Hodge star. 
The \textit{energy} of a connection $\connA$ is given by $\| F_{\connA} \|_{L^2}^2=-\int\text{tr}(F_{\connA}\wedge\star F_{\connA})$.
Instantons on $\xbb=\mathbb{R}\times\ybb$ may be interpreted as gradient flow-lines 
for the Chern-Simons functional. In actuality we consider a perturbed instanton equation\ 
involving $\pi$, and call the solutions instantons as well. 
Given acyclic $\conna,\connb\in\mathscr{C}(\ybb)$ we let 
$\moduli(\conna,\connb)$ be the space of $\mathscr{G}_\text{ev}$-classes of finite-energy 
instantons on $\xbb$ 
asymptotic at $-\infty$ to $\conna$ and at $+\infty$ to $\connb$. 
When the perturbation is zero, elements $[\connA] \in\moduli(\conna,\connb)$ 
are distinguished by the property $\frac{1}{8\pi^2}\|F_{\connA}\|_{L^2}^2 = \textbf{cs}(\connb)-\textbf{cs}(\conna)$.

For a small, generic perturbation $\moduli(\conna,\connb)$ is a smooth manifold,
and we write
\[
	\ind(\conna,\connb)=\dim\moduli(\conna,\connb).
\]
Passing to $\mathscr{G}_\text{ev}$-classes, the number $\ind(\connag,\connbg)$ is well-defined modulo $8$, 
and equips $\chain(\ybb)$ with a relative $\mathbb{Z}/8$-grading given by
$\text{gr}(\connag)-\text{gr}(\connbg)\equiv \ind(\connag,\connbg)$. The space 
$\moduli(\conna,\connb)$ has an $\mathbb{R}$-action by translation along the $\mathbb{R}$-factor 
of $\mathbb{R}\times\ybb$, and we write
\begin{align*}
	& \check{\moduli}(\conna,\connb) = \moduli(\conna,\connb)/\mathbb{R}.
\end{align*}
The data of $\connag,\connbg$ and 
the lift of $\ind(\connag,\connbg)\in\mathbb{Z}/8$ to $d\in\mathbb{Z}$ are sufficient 
to describe $\moduli(\conna,\connb)$; viewing $[A]\in M(a,b)$ as a path in $\mathscr{B}(\ybb)$, the index $d$ faithfully records the homotopy class of $[A]$ relative to the endpoints $\connag,\connbg$.
That said, if $d=\ind(\conna,\connb)$, we also write $\moduli(\connag,\connbg)_d$ for the space $\moduli(\conna,\connb)$, and similarly $\check{\moduli}(\connag,\connbg)_{d-1}$ for $\check{\moduli}(\conna,\connb)$.
Thus $\moduli(\connag,\connbg)_d$ is a $d$-dimensional component of instanton classes
whose limits are in the classes $\connag$ and $\connbg$. 

Suppose $\connag,\connbg \in\mathfrak{C}^\text{irr}$ with $\ind(\connag,\connbg)\equiv 1$.
With suitable perturbation, $\check{\moduli}(\connag,\connbg)_0$ is a 
finite set, and as explained in \cite[\S 3.6]{kmu}, each of its elements determines 
an isomorphism $\Lambda(\connag)\to\Lambda(\connbg)$. Denoting the induced isomorphism 
$\mathbb{Z}\Lambda(\connag)\to\mathbb{Z}\Lambda(\connbg)$ corresponding to 
$[A]\in \check{\moduli}(\connag,\connbg)_0$ by the symbol $\epsilon[A]$, the differential $\partial$ 
for $\chain(\ybb)$ is defined in pieces by
\[
	\partial|_{\mathbb{Z}\Lambda(\connag)\to\mathbb{Z}\Lambda(\connbg)} = 
	\sum_{[A]\in\check{\moduli}(\connag,\connbg)_0}\epsilon[A].
\]
If we choose an element from each $\Lambda(\connag)$, then we may view $\partial$ as a map 
on $\mathfrak{C}^\text{irr}$ and write 
$\langle \partial \connag,\connbg\rangle = \#\check{\moduli}(\connag,\connbg)_0$, where $\#$ indicates a 
signed count. The differential lowers the relative $\mathbb{Z}/8$-grading by $1$. 
The identity $\partial^2=0$ is obtained by interpreting the boundary of 
a 1-dimensional moduli space $\check{\moduli}(\connag,\connbg)_1$ as a disjoint union 
of broken trajectories $\check{\moduli}(\connag,\conncg)_0\times\check{\moduli}(\conncg,\connbg)_0$. 
The relatively $\mathbb{Z}/8$-graded abelian group $\ih(\ybb)$ is defined to 
be $H_\ast(\chain(\ybb),\partial)$.

In defining the complex $\chain(\ybb)$ we have chosen a Riemannian 3-manifold $\y$, an 
admissible $\so$-bundle $\ybb$ over $\y$, a perturbation $\pi$, and a base connection $a_0$ on $\ybb$. 
When working with the chain group we always assume such data is chosen. 
The isomorphism class of the relatively $\mathbb{Z}/8$-graded group $\ih(\ybb)$ depends only 
on the oriented homeomorphism type of $\y$ and $w_2(\ybb)$.\\

\subsection{Maps from Cobordisms}\label{sec:cob}

Let $\x:\y_1\to\y_2$ be a cobordism from $\y_1$ to $\y_2$. That is, $\x$ is a compact, 
connected, oriented 4-manifold 
with an orientation preserving diffeomorphism $\partial \x \simeq \y_2\sqcup \overline{\y}_1$. 
As before, each $\y_i$ is connected.
Assume $\x$ is equipped with a metric that is product-like near its boundary. Suppose further 
that $\xbb$ is an $\so$-bundle over $\x$ with $\xbb|_{\y_i} = \ybb_i$ where each 
$\ybb_i$ is admissible. We abbreviate this setup as $\xbb:\ybb_1\to\ybb_2$. To obtain a chain map
\[
	m(\xbb):\chain(\ybb_1)\to \chain(\ybb_2),
\]
first form the bundle $(\mathbb{R}_{\leq 0}\times \ybb_1) \cup \xbb \cup (\mathbb{R}_{\geq0}\times \ybb_2)$ 
over the Riemannian 4-manifold obtained from $\x$ by attaching cylindrical 
ends to the boundary. We define $\moduli(a,\xbb,b)$ to be the space of $\mathscr{G}_\text{ev}$-classes of 
finite-energy instantons on this bundle. With suitable perturbations chosen, $\moduli(a,\xbb,b)$ is a smooth 
manifold, and we write $\ind(\conna,\xbb,\connb)=\dim\moduli(\conna,\xbb,\connb)$. As before, $\ind(\connag,\xbb,\connbg)$ 
is well-defined modulo $8$, and we write
$\moduli(\connag,\xbb,\connbg)_d=\moduli(\conna,\xbb,\connb)$ for $d=\ind(\conna,\xbb,\connb)$.

Now suppose 
$\connag \in\mathfrak{C}^\text{irr}(\ybb_1)$ and $\connbg\in\mathfrak{C}^\text{irr}(\ybb_2)$ with 
$\ind(\connag,\xbb,\connbg) \equiv 0$. With suitable perturbations, $\moduli(\connag,\xbb,\connbg)_0$ is 
a finite set of points. In defining $\chain(\ybb_i)$, basepoint connections $\conna_{i,0}$ are chosen. 
Let $A$ be a connection on $\xbb$ (with cylindrical ends attached) with limits at the ends 
equivalent to the $\conna_{i,0}$. An orientation of the line $\text{det}(D_A)$ will be called 
an {\em I-orientation} of $\xbb$, following 
\cite[Def. 3.9]{kmu}. With an I-orientation of $\xbb$, an element $[A]\in\moduli(\connag,\xbb,\connbg)_0$ 
determines an isomorphism $\epsilon[A]:\mathbb{Z}\Lambda(\connag)\to\mathbb{Z}\Lambda(\connbg)$, 
and $m(\xbb)$ is defined in pieces by
\[
	m(\xbb)|_{\mathbb{Z}\Lambda(\connag)\to\mathbb{Z}\Lambda(\connbg)} = \sum_{[A]\in\moduli(\connag,\xbb,\connbg)_0}\epsilon[A].
\]
In shorthand, $\langle m(\xbb)\connag,\connbg \rangle  = \#\moduli(\connag,\xbb,\connbg)_0$. 
When $\ind(\connag,\xbb,\connbg)\not\equiv 0$ this 
part of the differential is zero. Different choices 
of I-orientations only affect the overall sign of the map $m(\xbb)$. The notation we use for 
composing bundle cobordisms is given by 
\[
	\xbb_2\circ\xbb_1 = \xbb_1\cup_{\ybb_2}\xbb_2:\ybb_1\to \ybb_3
\]
where $\xbb_i:\ybb_i\to\ybb_{i+1}$ for $i=1,2$. We write $\ih(\xbb_1):\ih(\ybb_1)\to\ih(\ybb_2)$ for the map on homology induced by $m(\xbb_1)$.
Having assumed $\y_i$ is connected for $i=1,2$, we have the composition law
\[	
	\ih(\xbb_2\circ\xbb_1) = \ih(\xbb_2)\circ\ih(\xbb_1).
\]
There is a well-defined 
notion of composing I-orientations using (\ref{detiso}) below, 
and this is needed to make sense of this expression. For a general discussion of the composition law involving disconnected
3-manifolds see \cite[\S 5.2]{kmu}. We mention that the composition law follows from the homotopy formula (\ref{met1}) 
below, using a 1-dimensional family of metrics that stretches along $\y_2$.\\

\subsection{Index Formulas}\label{sec:index}

The numbers $\ind(a,b)$ and $\ind(a,\xbb,b)$ above are more properly described as the indices of 
certain Fredholm operators. Let $\xbb:\ybb_1\to\ybb_2$ as above. The $\ybb_i$ 
are not assumed to be admissible. Let $a$ and $b$ be connections on $\ybb_1$ and $\ybb_2$, 
respectively. Attach cylindrical ends to $\xbb$ as above and call the result $\xbb$ as well. 
Choose a connection $A$ on $\xbb$ with $A|_{\ybb_1\times \{t\}}$ equal to $a$ for 
$t \ll0$ and $A|_{\ybb_2\times \{t\}}$ equal
to $b$ for $t \gg 0$, and consider the operator 
\[
	D_A=-d_A^\ast\oplus d_A^+: L^p_{s,\phi}(\Lambda^1\otimes\xbb_\text{ad})\to 
	L^p_{s-1,\phi}((\Lambda^0\oplus\Lambda^+)\otimes\xbb_\text{ad})
\]
where $L^p_{s,\phi}=\phi L^p_s$ are Sobolev spaces weighted by the real function $\phi$, equal 
to $e^{-\epsilon t}$ for some sufficiently small $\epsilon > 0$ on the ends 
$\mathbb{R}_{\leq 0}\times\y_1$ and $\mathbb{R}_{\geq 0}\times\y_2$, and equal to $1$ otherwise. 
This operator arises from linearizing the instanton equation and using a Coulomb gauge condition. 
If $\xbb':\ybb_2\to\ybb_3$ and $A'$ is a connection on $\xbb'$ with limit $b$ over $\ybb_2$, 
there is a natural isomorphism
\begin{equation}
	\text{det}(D_{A})\otimes\text{det}(D_{A'}) \simeq \text{det}(D_{A\cup A'})\label{detiso}
\end{equation}
and the index relation $\text{ind}(D_A)+\text{ind}(D_{A'})=\text{ind}(D_{A\cup A'})$ 
holds, see for example \cite[Prop. 5.11]{d}. In the definition of $\chain(\ybb)$ in \S \ref{sec:instantongroups} we take $\xbb=[0,1]\times\ybb$ to define $D_A$.

Note that the two ends $\ybb_1$ and $\ybb_2$ of the cobordism $\xbb$ have opposite Sobolev weights 
in the description of $D_A$. If we instead view $\xbb:\emptyset\to \ybb_2\sqcup \overline{\ybb}_1$ then the 
construction yields a different operator $D_A'$. That is, 
$D_A'$ differs from $D_A$ by using the weight function $\phi'$ in place of $\phi$, 
where $\phi'$ is obtained by altering $\phi$ over $\mathbb{R}_{\leq 0}\times\y_1$ 
from $e^{-\epsilon t}$ to $e^{+\epsilon t}$. We have the relation
\[
	\text{ind}(D_A')-\text{ind}(D_A) = h^0(a) + h^1(a),
\]
cf. \cite[Prop. 3.10]{d}. When there is one cylindrical end, 
the number $\text{ind}(D_A')$ is the same as 
$\text{ind}^-(A)$ in the notation of \cite{bd} and $\text{ind}^+(A)$ in the notation of \cite{d}.

The index $\text{ind}(D_A')$ is the expected dimension of the moduli space $\moduli(a,\xbb,b)^\text{irr}$ 
of irreducible instanton classes. It is this number that we refer to in computations, so we define
\[
	\ind(a,\xbb,b) =\ind(A) = \text{ind}(D_A'), \label{index}
\]
and this agrees with our earlier usage of $\ind(a,\xbb,b)$. Note that the order of the symbols 
$a,\xbb,b$ does not matter, and is only suggestive of the situation in mind.
If $\xbb_1$ and $\xbb_2$ are bundles over 
cobordisms and are composable, we have the gluing formula
\begin{equation}
	\ind(a,\xbb_2\circ\xbb_1,c) = \ind(a,\xbb_1,b) + \ind(b,\xbb_2,c) + h^0(b) + h^1(b).\label{glue}
\end{equation}
If $\xbb$ is over a closed 4-manifold $\x$ then we also have
\begin{equation}
	\ind(\xbb) = -2p_1(\xbb) -3(1-b_1(\x)+b_+(\x)).\label{closed}
\end{equation}
Here $b_+(\x)$ is the dimension of a maximal positive definite subspace for the intersection form on 
$H_2(\x;\mathbb{R})$. The term $1-b_1(\x)+b_+(\x)$ may also be written as $(\chi(\x)+\sigma(\x))/2$, 
where $\chi$ is the Euler characteristic and $\sigma$ the signature.\\

\subsection{Maps from Families of Metrics on Cobordisms}\label{sec:met}

This section extracts formulae due to Kronheimer and Mrowka from \cite[\S 3.9]{kmu}.
We first consider families of metrics in a general context. Let $\x$ be any smooth manifold and 
$\s$ a hypersurface in the interior of $\x$. We assume $\s$ has a neighborhood 
$N\subset\x$ diffeomorphic to $(-1,1)\times\s$. A \emph{metric on} 
$\x$ \emph{cut along} $\s$ is a Riemannian metric $g$ on $\x\setminus \s$ that on the 
neighborhood $N$ is of the form
\[
	dr^2/r^2 + g_0,
\]
where $g_0$ is a metric 
on $\s$ and $r$ is the parameter of $(-1,1)$. We also call $g$ simply a \emph{cut metric}. 
We may regard a Riemannian manifold with a cut metric as one with
two opposing cylindrical ends that along the cut hypersurface $S$ meet only at infinity. 

Given a collection of hypersurfaces $\mathcal{H}=\{\s_i\}$ in the interior of $\x$ with 
similar neighborhoods we construct a set of metrics $G=G(\mathcal{H})$ on $\x$ that are 
cut along various subsets of $\mathcal{H}$. The construction is intuitively 
simple: stretch an initial metric in all possible ways along each hypersurface.

First, suppose that $\mathcal{H}$ has no intersecting hypersurfaces. We will parameterize 
the family $G$ by $[0,1]^d$ where $d=|\mathcal{H}|$. Let $b_t$ be a family of positive 
smooth functions on $[-1,1]$ parameterized smoothly by $t\in[0,1)$ such that $b_t(r)$ 
approaches $1/r^2$ as $t$ goes to $1$. For some fixed $\varepsilon$, $0 < \varepsilon < 1$, we require that $b_t(r)=1$ for $|r|>\varepsilon$. We also require $b_t \neq b_{s}$ when $t\neq s$. We choose the 
initial metric $G(0)$ on $\x$ so that it is of the form $dr^2 + g_i$ in the neighborhood 
of $\s_i\subset\x$ diffeomorphic to $(-1,1)\times \s_i$. Here $\s_i\in\mathcal{H}$ and 
$g_i$ is a metric on $\s_i$. For $t\in[0,1]^{d}$ we define $G(t)$ on $\x$ by changing 
$G(0)$ in the neighborhood of $\s_i$ to $b_{t_i}(r)dr^2+g_i$.

Now consider an arbitrary set of hypersurfaces $\mathcal{H}$. Let $\mathcal{H}_0$ be 
a subset of $\mathcal{H}$ with no intersecting hypersurfaces. We have constructed a family 
$G(\mathcal{H}_0)$ for each such $\mathcal{H}_0$. We glue the hypercubes $[0,1]^{d_0}$ where 
$d_0=|\mathcal{H}_0|$ together to form a space in the obvious way: when two points
correspond to the same metric, identify them. This defines the family $G(\mathcal{H})$.

Now suppose $\xbb:\ybb_1\to \ybb_2$ as in \S \ref{sec:cob}. Let $G=G(\mathcal{H})$ be 
a family of metrics on $\x$ constructed as above. We extend $G$ to a family of metrics on 
$\x$ with cylindrical ends attached, 
product-like on the ends, which we also call $G$. Let $\moduli_G(a,\xbb,b)$ be the moduli 
space of pairs $([A],g)$ where $g\in G$ and $A$ is a finite-energy instanton with respect to $g$. 
The meaning of this is staightforward if $g$ is an uncut, smooth metric. An instanton with a 
metric cut along $S\subset \x$ is an instanton on the complement of $\text{S}$, with 
its limits on the two cylindrical ends $[0,\infty)\times\text{S}$ agreeing. 
More details can be found in \cite[\S 3.9]{kmu}.

Let $G=G(\mathcal{H})$ be a family of metrics on $\x$ as constructed above. 
In the cases in which we are interested, $G$ will 
have the structure of a convex polytope. The metrics parameterized by a face of $G$
consist of cut metrics, cut along a hypersurface in 
$\mathcal{H}$. 
The expected dimension of $\moduli_G(a,\xbb,b)$ is
$\ind(a,\xbb,b)  + \dim G$. A map
\[
	m_G(\xbb):\chain(\ybb_1)\to \chain(\ybb_2)
\]
is defined just as for cobordisms. To fix the sign of $m_G(\xbb)$, 
in addition to an I-orientation of $\xbb$, we must orient the metric family $G$. 
The following three formulae are due to Kronheimer and Mrowka, \cite[\S 3.9]{kmu}, and 
arise from understanding the compactification and gluing of certain moduli spaces. First,
\begin{equation}
	(-1)^{\dim G}m_G(\xbb)\partial - \partial m_G(\xbb) =m_{\partial G}(\xbb). \label{met1}
\end{equation}
In writing this we have inherited the orientation conventions of \cite{kmu}, 
with the exception that the quotients $\check{\moduli}(a,b)$ are oriented oppositely, 
changing the signs of the maps $\partial$.
For the polytopes $G$ that we will consider, $\partial G$ decomposes 
into a union of faces $G(\s)$, one for each hypersurface $\s\in\mathcal{H}$. In 
this case
\begin{equation}
	m_{\partial G}(\xbb)=\sum_{\s\in\mathcal{H}} m_{G(\s)}(\xbb). \label{met2}
\end{equation}
Finally, suppose $\xbb$ is the composite of two bundle cobordisms: $\xbb = \xbb_2\circ\xbb_1$. 
Also suppose that $G=G_1\times G_2$ where $G_1$ is a family of metrics that only varies on 
$\x_1$ and $G_2$ on $\x_2$, and all metrics are cut along $\x_1\cap\x_2$. Then
\begin{equation}
	m_G(\xbb) =(-1)^{\dim G_1\dim G_2}m_{G_2}(\xbb_2)m_{G_1}(\xbb_1) \label{met3}
\end{equation}
where we interpret $G_1$ as a family of metrics on $\x_1$ and $G_2$ as a family on $\x_2$. Here 
the metric families are oriented, and $G=G_1\times G_2$ is an orientation preserving identification.\\

\subsection{Index Bounds} \label{sec:indexbounds}
The following discussion is based on \cite[\S 3.4]{bd} and \cite[\S 4]{d}, with the 
material of \cite[\S 3.9]{kmu} in mind.
So far we have only mentioned moduli spaces for which the limiting connections are 
acyclic. This guarantees, in particular, that all instantons are irreducible. 

For simplicity, suppose $\xbb$ has one cylindrical end.
We consider moduli spaces $\moduli(\xbb,a)$ where $a$ is any almost flat connection 
(i.e., an element of $\mathfrak{C}$), 
where the finite-energy instantons exponentially approach $a$ over the cylindrical end. Then, 
with suitable perturbation,
the subset of irreducibles $\moduli(\xbb,a)^\text{irr}$ is a smooth manifold of 
dimension $\ind(\xbb,a)$.
In this case, the existence of $[A]\in\moduli(\xbb,a)^\text{irr}$ 
implies $\ind(\xbb,a)=\ind(A)\geq 0$. 
On the other hand, if all the instantons are reducible with common isotropy group $\Gamma$, 
the space $\moduli(\xbb,a)$ has dimension $\ind(\xbb,a)+ \dim \Gamma$. Recall 
$h^0(A)=\dim\Gamma$.
In this case, after perturbation, 
the existence of an instanton $[A]$ in the moduli space implies the bound 
\begin{equation}
	\ind(A) + h^0(A)\geq 0. \label{boundred}
\end{equation}
More generally, suppose 
$([A],g)\in\moduli_G(\xbb,a)$ for a family of metrics $G$. Then we obtain
\begin{equation}
  \ind(A) + h^0(A) + \dim G \geq 0.  \label{bound1}
\end{equation}
We also consider the case in which some of the limiting connections are allowed to vary.
Suppose $[0,\infty)\times \ybb$ is the cylindrical end of $\xbb$,  
and consider a smooth manifold $\mathfrak{F}\subset \mathfrak{C}(\ybb)$ of critical points
to which the Chern-Simons functional is non-degenerate transverse. 
We consider $\moduli(\xbb,\mathfrak{F})$,
the instanton classes that exponentially approach the set $\mathfrak{F}$.
The irreducibles within typically form a smooth manifold whose components have dimensions
mod $8$ congruent to $\ind(\xbb,\connag)+\dim\mathfrak{F}$, where $\connag\in \mathfrak{F}$. 
We write $\moduli(\xbb,\mathfrak{F})_d^{\text{irr}}$ for the $d$-dimensional component.

We can introduce metrics into all of these situations. 
The most general situation we consider is the following. 
Suppose $\mathfrak{F}$ is as above, 
and consider the moduli space $\moduli_G(\xbb,\mathfrak{F})$. 
If $([A],g)$ is a member, in the generic case we obtain a bound
\begin{equation}
  \ind(A) + h^0(A) + \dim G  + \dim\mathfrak{F} \geq 0. \label{bound2}
\end{equation}
We write $\moduli_G(\xbb,\mathfrak{F})^{\circ}_d$ 
for the $d$-dimensional moduli space of instantons $([A],g)$ 
with $d$ equal to the left side of (\ref{bound2})
and where $\circ=\text{irr},\text{red},\text{flat}$ describes the respective stabilizer-types 
$h^0(A)=0,1,3$. One can drop the assumption that $\mathfrak{F}$ is smooth and obtain moduli spaces that 
are stratified according to the structure of $\mathfrak{F}$. Such spaces have been studied 
in \cite{t,mmr}.\\

\subsection{Gradings}\label{sec:instgrading}

In addition to the relative $\mathbb{Z}/8$-grading on $\ih(\ybb)$,
we can define an absolute $\mathbb{Z}/2$-grading 
following \cite[\S 2.1]{froy} and \cite[\S 5.6]{d}. It is more 
generally defined on the critical sets $\mathfrak{C}$.
If $\connag\in\mathfrak{C}$, its grading is given by
\[
	\text{gr}(\connag) = b_1(E)+b_+(E)+\ind(\ebb,\connag) \mod 2,
\]
where $\ebb:\emptyset\to\ybb$ is an $\so$-bundle 
over a connected 4-manifold $E$ with $\partial E=\y$ that restricts to $\ybb$ over $\y$. 
The differential of $\chain(\ybb)$ shifts this grading by $1$.
A map $m(\xbb):\chain(\ybb_1)\to\chain(\ybb_2)$ shifts the grading by the parity of
\begin{equation}
	\text{deg}(\x) = -\frac{3}{2}(\chi(\x)+\sigma(\x))+\frac{1}{2}(b_1(\y_2)-b_1(\y_1)) ,\label{degofmap}
\end{equation}
cf. \cite[\S 4.5]{kmu}. More generally, a map $m_G(\xbb)$ shifts the grading by 
$\text{deg}(\x)+\dim G$. As an example, suppose $\tbb^3$ is the bundle over 
$T^3$ with $w_2(\tbb^3)$ Poincar\'{e} dual to an $S^1$-factor. Then 
$\ih(\tbb^3)$ is two copies of $\mathbb{Z}$ supported in the even grading. 
Note that the trivial connection 
$\theta$ on $S^3$ has $\text{gr}(\theta)\equiv 1$.
We note that $\ih(\overline{\ybb})_i$ is the same as 
the cohomology group $\ih(\ybb)^{b_1(\y)+1+i}$, where $\overline{\ybb}$ means the orientation of 
the base space $\y$ is reversed. 
For our conventions regarding the absolute $\mathbb{Z}/8$-grading in the 
case that $\y$ is a homology 3-sphere, see \S \ref{sec:connect}.

%% file: triangle.tex
\section{Proving the Exact Triangle}

In this section we prove Theorem \ref{thm:floer}, Floer's exact triangle.
We use an algebraic lemma first 
used in \cite{os} by Ozsv\'ath and Szab\'o 
to prove an exact sequence in 
Heegaard Floer homology. 
The use of metric stretching maps in this context was 
applied in \cite{kmos} by Kronheimer, Mrowka, and the previous two authors 
to prove an exact triangle in monopole Floer homology. 
Bloom \cite{bloom} also treats 
the monopole case. 
Our proof is largely an adaptation of Kronheimer and Mrowka's proof \cite{kmu}
in the singular instanton knot homology setting. 
In particular, while \S \ref{sec:firstmaps} is 
essentially part of Floer's original proof, see \cite[\S 4]{bd}, the contents of 
\S \ref{sec:secondmaps}, notably the idea for Lemma \ref{lem:fiber} and its proof, 
are based on ideas from \cite[\S 7.1]{kmu}.\\

\subsection{The Triangle Detection Lemma}

The following statement is adapted from \cite[\S 7.1]{kmu} 
and first appeared in \cite{os}.

\begin{lemma}
	Let $(\emph{\text{C}}_i,\partial_i)$ be a sequence of complexes, 
	$i\in\mathbb{Z}$. Suppose that there are chain maps $f_i:\emph{\text{C}}_i\to
\emph{\text{C}}_{i+1}$ and maps $h_i:\emph{\text{C}}_i\to \emph{\text{C}}_{i+2}$ satisfying
\[
        f_{i+1}f_i + \partial_{i+2}h_i + h_{i}\partial_i=0.
\]
Suppose further that each sum
\[
        f_{i+2}h_{i} + h_{i+1}f_{i}
\]
induces an isomorphism $H(\emph{\text{C}}_i)\to H(\emph{\text{C}}_{i+3})$. Then
\[
       \cdots \to H(\emph{\text{C}}_i) \xrightarrow[]{{H}(f_i)} 
       H(\emph{\text{C}}_{i+1}) \xrightarrow[]{{H}(f_{i+1})} 
       H(\emph{\text{C}}_{i+2}) \to \cdots
\]
is an exact sequence. Furthermore, the anti-chain map 
$f_i\oplus h_i:\emph{\text{C}}_i\to \text{\emph{Cone}}(f_{i+1})$ is a
quasi-isomorphism for each $i\in\mathbb{Z}$.\label{alg}
\end{lemma}

\noindent To apply this lemma, we use the notation of \S \ref{sec:topology}, 
so that we have a 3-periodic sequence of surgery bundles $\ybb_i$, $i\in\mathbb{Z}$,
and surgery cobordism bundles $\xbb_{ij}:\ybb_i\to\ybb_j$ whenever $j> i$. 
We let $(\chain_i,\partial_i)$ be the instanton 
chain complex $\chain(\ybb_i)$ with its differential. We take $f_i$ to be $m(\xbb_{i,i+1}):\chain(\ybb_i)\to \chain(\ybb_{i+1})$. 
The map $h_i$ is defined in \S \ref{sec:firstmaps}, and in \S \ref{sec:secondmaps} we define a chain homotopy $k_i$ from
$f_{i+2}h_{i} + h_{i+1}f_{i}$ to an intermediate map, and then show that this 
intermediate map is chain homotopic to the identity map of $\chain_i$ up to sign. 
All maps are of the form $m_G(\xbb)$.\\

\subsection{The $h_i$ maps}
\label{sec:firstmaps}

We define $h_0:\chain_0\to \chain_{2}$ in this section. Recall from \S \ref{sec:decomp1} 
that we can write
\[
	\x_{02}=W\cup_{\s_1} U
\]
where $U$ is diffeomorphic to $\cp$ minus a 4-ball, and $W$
has boundary $\y_2 \sqcup \overline{\y}_0\sqcup S_1$. The map $h_0$ 
is taken to be $m_G(\xbb_{02})$ where $G$ is a family of metrics on $\x_{02}$ induced 
by the set of two intersecting hypersurfaces $\mathcal{H}=\{\s_1,\y_{1}\}$. Thus $G$ is 
parameterized by an interval, with endpoint metrics $G(\s_1)$ and $G(\y_1)$, cut
along $\s_1$ and $\y_1$, respectively, as depicted in Figure \ref{fig:met2}. Equations 
(\ref{met1}) and (\ref{met2}) yield
\[
	-h_0\partial_0 -\partial_{2}h_0 = m_{G(\s_1)}(\xbb_{02}) + m_{G(\y_{1})}(\xbb_{02}).
\] 
By equation (\ref{met3}), we also have $m_{G(\y_{1})}(\xbb_{02}) = m(\xbb_{12})m(\xbb_{01}) = f_{1}f_0$. 
It remains to show that $m_{G(\s_1)}(\xbb_{02})=0$.

\begin{figure}[t]
\includegraphics[scale=.85]{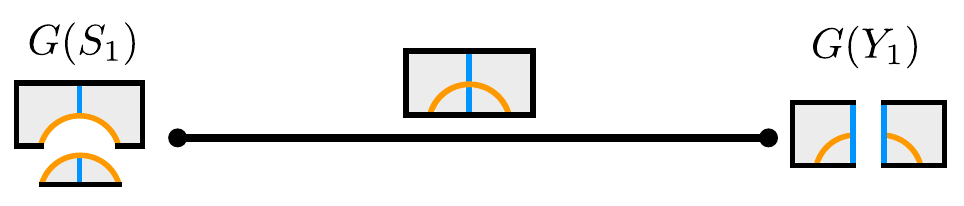}
\caption{The family of metrics on $\x_{02}$ used to define the $h_0$ map.}
\label{fig:met2}
\end{figure}

\begin{figure}[t]
\includegraphics[scale=.85]{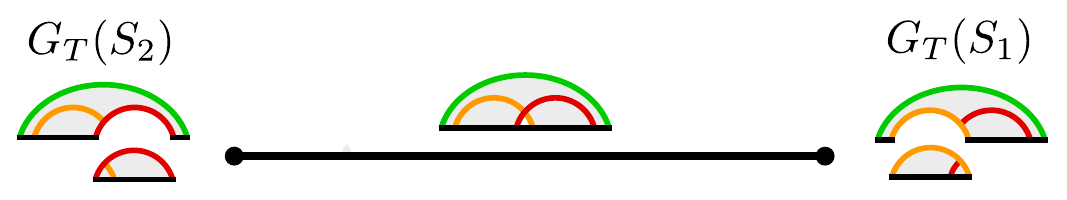}
\caption{The family of metrics $G_T$ on $E\subset\x_{03}$.}
\label{fig:met7}
\end{figure}

Let $a$ and $b$ be given with $\ind(a,\xbb_{02},b)=0$. To show that $m_{G(\s_1)}(\xbb_{02})=0$, 
it suffices to show that $\moduli_{G(\s_1)}(a,\xbb_{02},b)$ is empty for any such $a,b$. 
We prove this by contradiction. Suppose $[A]\in\moduli_{G(\s_1)}(a,\xbb_{02},b)$. 
Write 
$\ubb$ and $\wbb$ for the restriction of $\xbb_{02}$ to $U$ and 
$W$, respectively. Because 
$G(\s_1)$ is cut along $\s_1$, $[A]$ is a pair $[A_{W}],[A_{U}]$ in
$\moduli(a,\wbb,b,c)\times \moduli(c,\ubb)$ for 
some flat connection $c$ on $\sbb_1$. We arrange that the perturbation data near $\s_1$ is $0$. 
The gluing formula (\ref{glue}) says
\[
	\ind(A) =  \ind(A_{W}) + \ind(A_{U}) + h^0(c) + h^1(c).
\]
The flat connection $c$ is on a 3-sphere, so $h^1(c)=0$ and $h^0(c)=3$. Since $a$ and $b$ are 
irreducible, so is $A_{W}$. It follows that $\ind(A_W)\geq 0$, see inequality 
(\ref{bound1}). The connection $A_{U}$ may be reducible to $S^1$, 
but no further, because $\ubb$ is non-trivial, so $h^0(A_{U})\leq 1$. 
It follows from (\ref{boundred}) that $\ind(A_{U})\geq -1$, implying $\ind(A)=\ind(a,\xbb_{02},b)\geq 2$, a contradiction. \\

\subsection{The $k_i$ maps}\label{sec:secondmaps}

We define $k_0:\chain_0\to \chain_0$ in this section. Recall from \S \ref{sec:decomp2} that we 
have five hypersurfaces $\y_1,\y_2,\s_1,\s_2,T$ in $\x_{03}$ 
that intersect one another as in Figure \ref{fig:met3}. We define $k_0$ to be 
$m_G(\xbb_{03})$ where $G$ is the family of metrics on $\x_{03}$ induced by the set of 
hypersurfaces $\mathcal{H}=\{\y_1,\y_2,\s_1,\s_2,T\}$. The family $G$ 
is parameterized by a pentagon and has faces $G(\y_{1}), G(\y_{2}), G(\s_1), G(\s_2), G(T)$, 
each of which is an interval of metrics broken along the indicated 
hypersurface. See Figure \ref{fig:met4}. Equations (\ref{met1}) and (\ref{met2}) yield
\[
	k_0\partial_0-\partial_{0}k_0 = \sum_{\s\in\mathcal{H}}m_{G(\s)}(\xbb_{03})
\]
and the argument from \S \ref{sec:firstmaps} shows that 
$m_{G(\s_1)}(\xbb_{03}) = m_{G(\s_2)}(\xbb_{03}) = 0$. 
We also have $m_{G(\y_{1})}(\xbb_{03}) = h_{1}f_0$ and $m_{G(\y_{2})}(\xbb_{03}) = f_{2}h_0$ by 
(\ref{met3}). Thus
\[
	k_0\partial_0 -\partial_0k_0= m_{G( T )}(\xbb_{03}) + f_2h_0 + h_1f_0,
\]
or in other words, $k_0$ is a chain homotopy from $-m_{G(T)}(\xbb_{03})$ to $f_{2}h_{0} + h_{1}f_{0}$. 
The proof is thus complete if we establish

\begin{figure}[t]
\includegraphics[scale=.8]{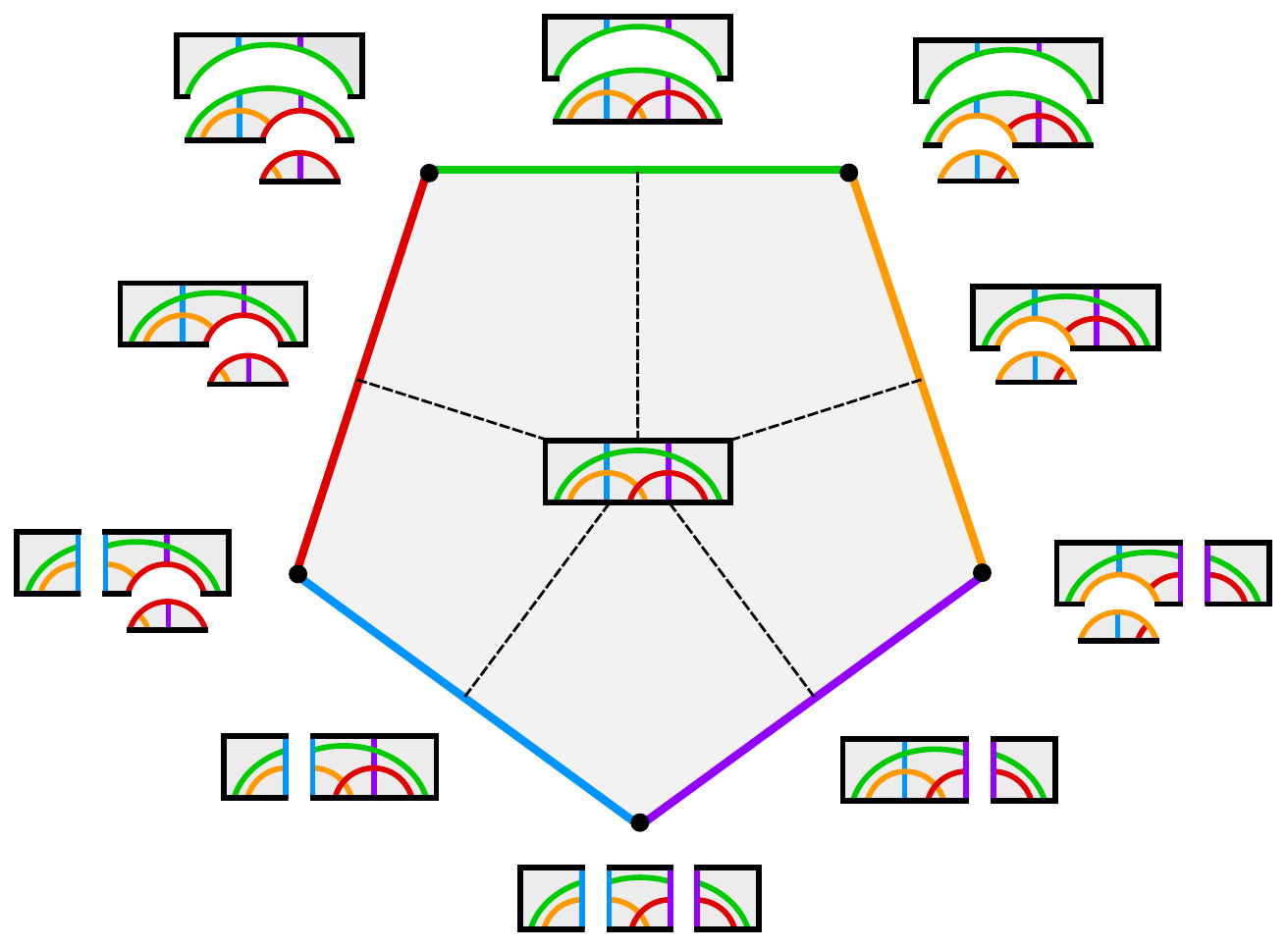}
\caption{The family of metrics on $\x_{03}$ used to define the $k_0$ map. This picture 
is modelled on Bloom's from \protect\cite{bloom}.}
\label{fig:met4}
\end{figure}

\begin{lemma}
	\emph{$m_{G(T)}(\xbb_{03})$} is chain homotopic to \emph{$\pm\text{id}:\chain_0\to \chain_0$}. \label{chainhom}
\end{lemma}

\noindent The remainder of this section goes towards proving this lemma. 
From \S \ref{sec:decomp2}, we know the hypersurface $\T$ induces a decomposition 
\[
	\x_{03}=V \cup_{T} E
\]
where $E$ is diffeomorphic to 
$\cp$ minus a regular neighborhood of an unknotted $S^1$. Let 
$\vbb, \ebb$ be the restrictions of $\xbb_{03}$ to $V, E$, respectively. 
The restriction of $G(T)$ to $V$ is a single metric. On the other hand, the 
restriction of $G(T)$ to $E$ is an interval of metrics, and we denote this 
family by $G_T$, see Figure \ref{fig:met7}. We arrange that the perturbations 
used near $T$ are zero, so that the relevant limiting connections are flat. 

The map $m_{G(\sthree)}(\xbb_{03})$ is defined by counting isolated points 
$[\connA]\in \moduli_{G(\sthree)}(\connag,\xbb_{03},\connbg)_0$. That is,
\[
	\langle m_{G( T )}(\xbb_{03}) \connag,\connbg\rangle = \#\moduli_{G_T}(\connag,\xbb_{03},\connbg)_0
\]
where $\#$ means a signed count determined by orienting moduli spaces.
Note that $\ind(\connA)=-1$ since $\dim G(T) = 1$. 
Let $a$ and $b$ be the limiting connections of $A$ on $\ybb_0$ and $\ybb_3$, respectively, 
so $[\conna]=\connag$ and $[b]=\connbg$. 
Each such $A$ can be written as a pair 
\begin{equation}
	A_V,(A_E,g)\label{pair}
\end{equation}
where $A_V$ is an instanton on $\vothreebb$
with limit $a$ over $\ybb_0$, $b$ over $\ybb_3$, 
and some flat limit $c$ over $\tbb$; and $A_E$ is a $g$-instanton on $\eothreebb$ 
where $g\in G_T$, and $A_E$ has the same flat limit $c$ over $\sthreebb$.

First, let us understand $\mathfrak{T}=\mathfrak{C}(\tbb)$, the 
space of $\mathscr{G}_\text{ev}$-classes of flat connections on $\tbb$. 
Recall that $\tbb$ is a trivial $\so$-bundle over an $S^1\times S^2$. 
Choose a spin structure for $\tbb$, i.e. a lift to an $\su$-bundle. Lifting connections 
sets up a bijection between flat $\so$-connections modulo $\mathcal{G}_\text{ev}$ on $\tbb$
with flat $\su$-connections modulo $\su$ gauge transformations. It is well-known that this 
latter set is in correspondence with $\text{Hom}(\pi_1(T),\su)$ modulo conjugation, which is 
essentially the set of conjugacy classes of $\su$.
The space of conjugacy classes of $\su$ is $[-1,1]$, given 
by the trace map divided by $2$.

The isomorphism $\mathfrak{T} \simeq [-1,1]$ depends on the spin structure of $\tbb$ 
chosen. There are two such choices, and they
are related by any \textit{non}-even gauge transformation of $\tbb$; 
using such a transformation the 
isomorphisms $\mathfrak{T}\simeq [-1,1]$ are related by reflecting $[-1,1]$ about $0$. The choice of 
isomorphism can also be determined by choosing a trivial holonomy flat connection 
on $\tbb$; this choice corresponds to $1\in[-1,1]$.
We record the following.

\begin{lemma}
	A choice of spin structure for $\tbb$ determines an isomorphism $\mathfrak{T}\simeq [-1,1]$. 
	The action on $\mathfrak{T}$ by $\mathscr{G}/\mathscr{G}_\text{ev}\simeq\mathbb{Z}/2$ under 
	this isomorphism is reflection about $0$. \label{lem:flat}
\end{lemma}

\noindent We can now understand the structure of the relevant moduli space 
following basic index computations.
Write $\mathfrak{T}^0$ for the interior of 
$\mathfrak{T}$, and $G_T^0$ for the interior of $G_T$.

\begin{lemma}\label{lem:fiber}
The moduli space $\moduli(\connag,\xbb_{03},\connbg)_0$ can be identified 
with the fiber product 
\[	
	\moduli(\connag,\vbb,\connbg,\mathfrak{T}^0)_0\times_{\mathfrak{T}^0}\moduli_{G_T^0}(\mathfrak{T}^0,\ebb)^\text{red}_1
\]
after a suitable perturbation.\label{lemma:fiber}
\end{lemma}

\noindent The moduli space on the right is the space of pairs $([A_E],g)$ where 
$g\in G_T^0$ and $A_E$ is a $g$-instanton on $\ebb$ (exponentially decaying over the ends),
such that the flat limit class of $A_E$ over $T$ lies in the interior of $\mathfrak{T}$;
$h^0(A_E)=1$, i.e. $A_E$ has gauge-stabilizer $S^1$; and 
$\mu(A_E)=1-h^0(A_E)-\dim G_T^0-\dim\mathfrak{T}^0=-2$.
In other words, the lemma says that in the pair (\ref{pair}) representing 
$[A]\in\moduli(\connag,\xbb_{03},\connbg)_0$,
we have the constraints 
\begin{align}
	\conncg = [\connc] \in\mathfrak{T}^0, \quad g\in G_T^0,\quad \ind(A_V)=-1,\quad \ind(A_E)=-2.\label{eq:constraints}
\end{align}
The fiber product is taken with respect to limit maps $\lambda:\moduli\to\mathfrak{T}^0$ 
that send an instanton class to its flat limit class over $\tbb$, where 
$\moduli$ is one of the two moduli spaces appearing in the lemma.
This fiber product description is an application of the Morse-Bott gluing 
theory as discussed in \cite[\S 4.5.2]{d} and \cite{mrowka,mmr,t}. Our situation, 
that of instantons broken along $S^1\times S^2$ with flat limits in 
$\mathfrak{T}\simeq [-1,1]$, is similar to that of Fintushel and Stern's in \cite{fs2tor}, 
where results of Mrowka's thesis \cite{mrowka} are used, and we will refer the reader 
to these sources for more details. 
We mention that for the above fiber product it is important that 
the stabilizers of $\connc$ and $A_E$, each a circle, can be identified. 
In general, one must record a gluing parameter in $\Gamma_\connc/\Gamma_{A_V}\times\Gamma_{A_E}$ where 
$\Gamma_A$ is the stabilizer of $A$. 
For instance, if both $[A_V]$ and $[A_E]$ were irreducible, there would be more than one choice of such a parameter.
We proceed to prove that the constraints (\ref{eq:constraints}) characterize
the possible gluing data.

\begin{proof}[Proof of Lemma \ref{lemma:fiber}]
We first show $\conncg\in\mathfrak{T}^0$. 
For convenience we set
\[
	h(\connc)=(h^0(\connc)+h^1(\connc))/2.
\]
We note that $h(\connc)=1$ or $3$, depending on whether $\conncg$ is in the 
interior or boundary of $\mathfrak{T}$, respectively, cf. \cite[\S 3]{fs2tor}.
By assumption $\ind(\connA)=-1$, so (\ref{glue}) yields
\[
	-1 = \ind(\connA) = \ind(\connA_V) +\ind(\connA_E) + 2h(\connc).
\]
Let $A_{S^1\times D^3}$ be a connection on the trivial bundle over $S^1\times D^3$
with one cylindrical end attached.
We identify the bundle over cross-sections of the end with $\sthreebb$, 
with the base having the opposite orientation of $T$. Suppose 
$A_{S^1\times D^3}$ has flat limit $\connc$. We glue $\connA_{S^1\times D^3}$ to $\connA_E$ to obtain a 
connection $\connA_{\cp}$ on a non-trivial bundle $\ebb'$ over 
$\cp$. The isomorphism class of $\ebb'$ depends on $\connc$, but we know 
$p_1(\ebb')=4k-1$ for some $k\in\mathbb{Z}$, cf. \cite[\S 4.1.4]{dk}. We have
\[
	\ind(\connA_E) + \ind(\connA_{S^1\times D^3}) + 2h(c)=\ind(\connA_{\cp}).
\]
We compute $\ind(A_{S^1\times D^3})$.
Two copies of $S^1\times D^3\times\so$, each with a cylindrical end, 
glue, overlapping the ends, to give $S^1\times S^3\times\so$. Index additivity yields 
\[
	2\ind(A_{S^1\times D^3})+ 2h(c) = \ind(S^1\times S^3\times\so).
\]
On the other hand, (\ref{closed}) says the right hand side is
\[
	-3(1-b_1+b_2^+)(S^1\times S^3) = 0.
\]
Thus $\ind(A_{S^1\times D^3}) = -h(c)$. This can also 
be deduced from the Atiyah-Patodi-Singer index theorem, 
cf. \cite[Thm. 3.10]{aps}. From (\ref{closed}) we obtain 
$\ind(A_{\cp})=-8k-1$, and then
\[
	\ind(A_V) = 8k - h(c), \qquad \ind(A_E) = -8k - 1 - h(c).
\]
Suppose for contradiction that $\conncg$ is on the boundary of $\mathfrak{T}$, so that $h(c)=3$. 
Since $A_V$ is irreducible and the boundary of $\mathfrak{T}$ has dimension $0$, we have
\[
	8k-3=\ind(A_V) \geq 0
\]
in the generic case, so $k > 0$. Since $\ebb'$ is nontrivial, $h^0(A_E)\in\{0,1\}$. 
Using (\ref{bound1}), we find 
\[
	-8k-4 = \ind(A_E) \geq -\dim G_T -\dim \partial \mathfrak{T} - h^0(A_E) \geq -2.
\]
Then $k < 0$, a contradiction. Thus $h(c)=1$ and $\conncg\in\mathfrak{T}^0$. 
It follows that $\ind(A_V)=8k-1$ and $\ind(A_E)=-8k-2$. 
Applying (\ref{bound1}) in this case, 
\[
	\ind(A_E) \geq -\dim\mathfrak{T}^0 -\dim G_T -h^0(A_E) \geq -3,
\]
so $k \leq 0$. Similarly, $\ind(A_V) \geq -\dim\mathfrak{T}^0 = -1$ 
gives $k\geq 0$. Thus $k=0$, yielding $\ind(A_V) = -1$ and $\ind(A_E)=-2$, 
as claimed.

Next, we rule out the possibility that $h^0(A_E)=0$, 
that is, that $[A]\in M_{G(T)}(\connag,\xbb_{03},\connbg)_0$ 
can be written as a gluing of $[A_V]$ and $([A_E],g)$ where 
$A_E$ is {\em irreducible}, i.e.
\[
	([A_E],g)\in\moduli_{G_T}(\mathfrak{T}^0,\ebb)^\text{irr}_0.
\]
Note that if there were such a gluing, we would have to keep track of a gluing 
parameter, as mentioned earlier. 
However, this moduli space of irreducibles and $\moduli(\connag,\vbb,\connbg,\mathfrak{T}^0)_0$ 
are both finite sets after perturbation, by standard compactness 
results, cf. \cite[\S 5]{fs2tor}. Further, the intersection of their flat limits in $\mathfrak{T}^0$ 
can be made transverse, in which case they have empty intersection. 
Thus, after a suitable perturbation, $h^0(A_E)=1$.

Finally, we show $g\in G_T^0$. Suppose for contradiction that $g\in\partial G_T$. Then 
$g$ is one of two metrics on $E$, $G_T(\s_1)$ or $G_T(\s_2)$, cut along $\s_1$ or 
$\s_2$, respectively. See Figure \ref{fig:met7}. Suppose $g=G_T(\s_1)$; the other case is similar.
Write 
\[
	E  =  X\cup_{\s_1} U
\]
where $U\simeq \cp\setminus\text{int}(D^4)$
and $X\simeq D^2\times S^2\setminus \text{int}(D^4)$. Note that the restriction of $\ebb$ 
over $X$ is trivial, while the restriction over $U$, as in \S \ref{sec:firstmaps},
is non-trivial; write $A_X$ and $A_{U}$ 
for the restriction of $A_E$
over these respective bundles. They have a common 
flat limit $d$ on $\sbb_{1}$. In particular, $h^0(d)=3$ and $h^1(d)=0$. 
The connection $A_X$ has the limit $c$ over $\tbb$ from before.

We compute $\ind(A_X)$ and $\ind(A_{U})$. There
is only one instanton class on $X$: 
the trivial class, cf. \cite[\S 7.4.1]{d}. Thus $A_X$ is trivial, 
so $h(c)=3$. Let $A_{S^1\times D^3}$ be a connection on the trivial bundle over 
$S^1\times D^3$ with one cylindrical end attached 
whose flat limit is $c$. Then $A_X$ and $A_{S^1\times D^3}$ glue, 
overlapping ends, to give a connection $A_{D^4}$ over $D^4$ with one cylindrical end attached. 
Then (\ref{glue}) and (\ref{closed}) yield
\[
	\ind(A_X) + \ind(A_{S^1\times D^3}) + 2h(c) = \ind(A_{D^4}) = -3.
\]
From above, $\ind(A_{S^1\times D^3})=-h(c)=-3$. Thus $\ind(A_X) = -6$. 
With $\ind(A_{U})=8k-1$ for some $k\in\mathbb{Z}$, we apply (\ref{glue}) once more to get
\[
	\ind(A_E) = \ind(A_X) + \ind(A_{U}) + 2h(d) = 8k-4.
\]
It follows that $\ind(A_E)\neq-2$, a contradiction.
\end{proof}

\begin{lemma}
The projection $\moduli_{G_T^0}(\mathfrak{T}^0,\ebb)_1^\text{red}\to G_T^0$ is a smooth homeomorphism.\label{lem:met}
\end{lemma}

\begin{proof}
The moduli space here is topologized as a subset of $\mathscr{B}\times G_T^0$, 
so the projection map is a continuous, open map. It is also smooth, 
in the transverse case, by general theory. It suffices to show 
bijectivity. The argument is a standard account of counting reducible instantons. 

Let $([A_E],g)$ be such that $\ind(A_E)=-2,h^0(A_E)=1$ and $g\in G_T^0$.
Because $H^1(E;\mathbb{R})=0$, $E$ admits no non-trivial real line bundles. 
Thus $h^0(A_E)=1$ implies $A_E$ is compatible 
with a splitting $\lbb \oplus \underline{\mathbb{R}}$ of the associated vector bundle of $\ebb$, 
where $\mathbb{L}$ is a complex line bundle and $\underline{\mathbb{R}}$ 
is a trivial real line bundle.
Gluing $A_E$ to a connection $A_{S^1\times D^3}$ on a trivial 
bundle over $S^1\times D^3$ with one cylindrical end attached 
gives an instanton $A_{\cp}$ on a bundle 
$\underline{\mathbb{R}}'\oplus \mathbb{L}'$ over $\cp$ where $\underline{\mathbb{R}}'$ and 
$\mathbb{L}'$ are extensions of $\underline{\mathbb{R}}$ and $\mathbb{L}$. 
The gluing formula says
\begin{equation}\label{p1}
	\ind(A_E) + \ind(A_{S^1\times D^3}) + h^0(c) + h^1(c) = \ind(A_{\cp}) =
	-2p_1(\underline{\mathbb{R}}'\oplus \mathbb{L}')-3.
\end{equation}
Using that $p_1(\underline{\mathbb{R}}'\oplus \mathbb{L}')=c_1(\mathbb{L}')^2$ 
we have $\ind(A_E)= -2c_1(\mathbb{L}')^2-4$. Since $\ind(A_E)=-2$, we conclude that $c_1(\mathbb{L}')^2=-1$. 
Let $P(E)$ denote the image of the map $H^2(E,\partial E;\mathbb{Z})\to H^2(E;\mathbb{Z})$.
Note that inclusion $E\to\cp$ induces an isomorphism of intersection forms from $H^2(\cp;\mathbb{Z})$ to $P(E)$, 
both negative definite of rank 1, under which $c_1(\mathbb{L}')$ is sent to $c_1(\mathbb{L})$. It follows that $c_1(\mathbb{L})$ is a generator of $H^2(E;\mathbb{Z})$.

There are thus two choices of $\mathbb{L}$ corresponding 
to the choices of generator for $H^2(E;\mathbb{Z})$. To get one from the other take the 
conjugate $\lbb^*$. The choice we make does not matter in the end, as we can relate 
the two by an even gauge transformation, by combining the conjugation map $\mathbb{L}\to\mathbb{L}^\ast$ 
with the involution of $\underline{\mathbb{R}}$ that 
reflects each fiber. Note 
that $\mathscr{G} = \mathscr{G}_\text{ev}$ for $\ebb$.

We are left with the problem of finding $g$-instantons on $\lbb$.
According to \cite[Prop. 4.9]{aps}, the space of $L^2$ 
harmonic 2-forms on $\e$ 
is isomorphic to the image of $H^2(E,\partial E;\mathbb{R})\to H^2(E;\mathbb{R})$, 
and under this isomorphism a harmonic form $x$ corresponds to its de Rham 
class $[x]$. In our case this map is an isomorphism $\mathbb{R}\to\mathbb{R}$. Further, 
any such harmonic $x$ satisfies $\star x = -x$, as follows: 
$\star x$ is $L^2$ harmonic, so $\star x =cx$ for some $c\in\mathbb{R}$; then 
$\star^2=1$, $\int x\wedge x<0$, and 
$0\leq \left\| x \right\|_{L^2}^2 = \int x\wedge \star x = c\int x\wedge x$ imply that $c=-1$.
Conversely, a closed $L^2$ 2-form $x$ satisfying $\star x = -x$ is easily seen to be 
$L^2$ harmonic.

The arguments from \cite[\S 2.2.1]{dk} easily 
adapt here, since $H^1(E;\mathbb{R})=0$, to show 
that given a closed $L^2$ 2-form $x$ on $E$, there is a 
connection $A$ on $\lbb$ with curvature $ix$ which is unique  
up to gauge equivalence. In this way, the unique $L^2$ harmonic 
2-form representing $-2\pi c_1(\lbb)$ specifies 
a unique $g$-instanton class on $\lbb$.
\end{proof}

\begin{lemma}
The moduli space $\moduli_{\partial G_T}(\partial \mathfrak{T},\ebb)_0^\text{red}$ consists of two points, 
and is the natural boundary of the open interval $\moduli_{G_T^0}(\mathfrak{T}^0,\ebb)_1^\text{red}$.
\end{lemma}

\begin{proof}
The previous lemma tells us that the ends of the latter moduli space 
are essentially the ends of $G_T$.
There are two endpoint metrics of $G_T$, labelled 
$G_T(\s_1)$ and $G_T(\s_2)$, each broken along the indicated 3-sphere. 
Any instanton $A$ on $\ebb$
compatible with $G_T(\s_1)$ is a gluing of 
the trivial instanton on the trivial bundle over $X\simeq D^2\times S^2\setminus \text{int}(D^4)$ 
with two cylindrical ends attached 
and an instanton $A_{U}$ on 
$U\simeq-\mathbb{C}\mathbb{P}^2\setminus \text{int}(D^4)$
with one cylindrical end attached. 
By the removable singularities theorem of Uhlenbeck, 
cf. \cite[Thm. 4.4.12]{dk}, the instanton $A_{U}$ 
uniquely extends to an instanton $A$ on a bundle $\mathbb{W}$ 
over $-\mathbb{C}\mathbb{P}^2$. If $A$ is to be a 
limit of elements in $\moduli_E$, then $p_1(\mathbb{W})=-1$. 
There is only one such instanton class
on $\mathbb{W}$, cf. \cite[\S 2.7]{kmu}.
Thus $[A]$ is uniquely determined. Similarly, there is one instanton 
class to add for $G_T(\s_2)$. That $A$ is trivial over $X$ 
implies the flat limits over $\tbb$ of these two instanton classes lie in $\partial\mathfrak{T}$. 
\end{proof}

Note that the map in Lemma \ref{lem:met} extends 
to a homeomorphism of closed intervals. We write
$\moduli_{G_T}(\mathfrak{T},\ebb)_1^\text{red}$
for the completed closed interval moduli space. We 
call a map between closed intervals \textit{proper} 
if it sends boundary to boundary. A proper map between
oriented, closed intervals has 
a well-defined degree, which is $0$ or $\pm 1$. Indeed, 
one can define the degree by looking at the induced map 
$S^1\to S^1$ obtained by identifying boundary points.

\begin{lemma}
	The map $\lambda:\moduli_{G_T}(\mathfrak{T},\ebb)_1^\text{red}\to\mathfrak{T}$ defined 
	by sending an instanton class to its flat limit class over $\tbb$ has degree $\pm 1$.
\end{lemma}

\begin{proof}
	We use the involution $\sigma:\eothreebb\to\eothreebb$ of \S \ref{sec:involution}.
	Write $\moduli$ for the moduli space in the lemma. 
	We see that $\sigma$ induces an action on $\moduli$, and
	because $\sigma(\tbb)=\tbb$, an action on $\mathfrak{T}$. 
	We can arrange the family of metrics $G_T$ so that 
	$\sigma$ restricts to an isometry of the base space and reflects $G_T$,
	in turn swapping the endpoints of the interval $\moduli$. 
	If we establish that $\sigma$ also swaps the endpoints of the interval $\mathfrak{T}$, 
	we are done, because the limit map $\lambda$ respects the 
	action of $\sigma$. From \S \ref{sec:involution} we know that with respect to a fixed trivialization $\tbb\simeq S^1\times S^2\times\so$, $\sigma$ 
	is isotopic to a composition $\theta\circ\upsilon$, where $\theta$ is a diffeomorphism of $S^1\times S^2$ lifted in a trivial way to $S^1\times S^2\times\so$. The diffeomorphism under consideration acts trivially on $\pi_1(T)$, and hence $\theta$ acts trivially on $\mathfrak{T}$. The map $\upsilon$ is a non-even gauge transformation, so by Lemma \ref{lem:flat}, it reflects the interval $\mathfrak{T}$. It follows that $\sigma$ reflects $\mathfrak{T}$.
\end{proof}

\begin{proof}[Proof of Lemma \ref{chainhom}]
By our fiber product description of $\modx$ we can write
\[
	\#\moduli_{G_T}(\connag,\xbb_{03},\connbg)_0 =\pm\sum\varepsilon(x)\varepsilon(y)
\]
where the sum is over pairs
\[
	(x,y)\in\moduli(\connag,\vbb,\connbg,\mathfrak{T}^0)_0\times\moduli_{G_T^0}(\mathfrak{T}^0,\ebb)_1^\text{red}
\]
having equal flat limit class $\lambda(x)=\lambda(y)\in\mathfrak{T}^0$. 
Each $x$ and $y$ has a sign, $\varepsilon(x)$ and $\varepsilon(y)$ respectively, prescribed by orienting 
moduli spaces. In the generic case, the sum of the 
$\varepsilon(y)$ for a fixed value $\lambda(y)$ equals $\pm\text{deg}(\lime)=\pm 1$. In this way 
we obtain
\[
	\#\moduli_{G_T}(\connag,\xbb_{03},\connbg)_0=\pm\#\moduli(\connag,\vbb,\connbg,\mathfrak{T}^0)_0
\]
where the sign does not depend on the pair $(\connag,\connbg)$.
Thinking of cobordisms as morphisms, we abbreviate $[0,1]\times \ybb_0$ to $1_{\ybb_0}$. Write 
$1_{\ybb_0} = \vothreebb \cup_{\sthreebb} \mathbb{W}$ where 
$\mathbb{W}$ is a trivial bundle over $W=S^1\times D^3$. We choose 
the perturbation data for $\wbb$ to be $0$. Let $Q$ be the family of metrics 
on $[0,1]\times\y_0$ induced by $\mathcal{H}=\{\sthree\}$. The boundary of $Q$ consists of an initial 
product metric on $[0,1]\times\y_0$ and a metric $Q(\sthree)$ cut along $\sthree$. Thus (\ref{met1}) and 
(\ref{met2}) yield 
\[
	-m_{Q}(1_{\ybb_0}) \partial_0-\partial_0 m_{Q}(1_{\ybb_0}) = m_{Q(\sthree)}(1_{\ybb_0}) + m(1_{\ybb_0}).
\] 
Of course, $m(1_{\ybb_0})$ is the identity. It remains to show 
$m_{Q(\sthree)}(1_{\ybb_0})=\pm m_{G(\sthree)}(\xbb_{03})$, or
\begin{equation}
	\#\moduli_{Q(\sthree)}(\connag,\connbg)_0=\pm\#\moduli(\connag,\vbb,\connbg,\mathfrak{T}^0)_0 \label{eq:fib2}
\end{equation}
where again the sign does not depend on the pair $(\connag,\connbg)$.
In the spirit of our previous arguments, we establish this by arguing that 
$\moduli(\connag,\vbb,\connbg,\mathfrak{T}^0)_0$
can be written as a fiber product
\[
	\moduli(\connag,\vbb,\connbg,\mathfrak{T}^0)_0\times_{\mathfrak{T}^0}\moduli(\mathfrak{T}^0,\mathbb{W})_1^\text{flat}.
\]
Here $\moduli(\mathfrak{T}^0,\mathbb{W})_1^\text{flat}$ is the 1-dimensional family of flat connection
classes on $\mathbb{W}$ with arbitrary flat limit class in $\mathfrak{T}^0$. 
Indeed, any flat connection class on $\tbb$ uniquely extends to a flat connection class
on $\mathbb{W}$ over $S^1\times D^3$.
We conclude that all instantons on $\mathbb{W}$ are flat, cf. \cite[\S 7.4]{d}.
In particular, the limit map $\lambda:\moduli(\mathfrak{T}^0,\mathbb{W})_1^\text{flat}\to\mathfrak{T}^0$ 
is a smooth homeomorphism. Now suppose $[A]\in\moduli_{Q(\sthree)}(\connag,\connbg)_0$ 
restricts to a pair $[A_V],[A_\text{W}]$ of instantons on $\vbb$ and $\wbb$, respectively, 
with equal limit $\conncg$ over $\tbb$. Then
\[
 0 = \ind(A) = \ind(A_V)+\ind(A_\text{W}) + 2h(c).
\]
We saw in Lemma \ref{lemma:fiber} that $\ind(A_W)=-h(c)$, so $\ind(A_V)= -h(c)$. The space of $[A_V]$ with 
$\ind(A_V)=-2$ is generically empty, so we conclude that $\ind(A_V)=-1$. It follows that 
$\conncg\in\mathfrak{T}^0$. Because the stabilizer of each $A_W$ is $\su$, 
the gluing parameter space is trivial, and our fiber product description is verified, cf. \cite[\S 4]{fs2tor}.
Because the limit map $\lambda:\moduli(\mathfrak{T}^0,\mathbb{W})_1^\text{flat}\to\mathfrak{T}^0$ 
is a homeomorphism, our fiber product yields (\ref{eq:fib2}). This completes the proof of Lemma \ref{chainhom}, 
and consequently the proof of Theorem \ref{thm:floer}.
\end{proof}

%% file: links.tex
\section{A Link Surgeries Spectral Sequence}

In this section we prove Theorem \ref{thm:2}. We follow
\cite{kmu} and \cite{bloom}. In \cite{kmu}, Kronheimer and Mrowka work over 
$\mathbb{Z}$, taking care with signs, and we adapt many of the details from their setup. 
Bloom's paper \cite{bloom} is especially 
descriptive of the combinatorics involved here, and provides many illustrations. 
As mentioned in the introduction, the idea for this spectral sequence 
originates from Ozsv\'ath and Szab\'o's paper \cite{os}.\\

\subsection{The Cobordisms \& Metric Families}

Let $\ybb$ be an admissible bundle over $\y$ and $\lt\subset \y$ a framed link 
with $m$ components $\lt_1,\ldots,\lt_m$. Suppose we have admissible bundles 
$\ybb_v$ for $v\in\{\infty,0,1 \}^m$ that form a surgery cube as in 
\S \ref{sec:main}. We conflate the 
subscript $\infty$ with $-1$ and write $\ybb_v$ for $v\in\{-1,0,1\}^m$. Further, 
we write $\ybb_v$ for $v\in\mathbb{Z}^m$ by taking the modulo $3$ reduction of $v$. 
Define the norms
\[
	|v|_1=\sum_{i=1}^m |v_i|, \qquad |v|_\infty = \max_{1\leq i \leq m}\{|v_i|\}.
\]
We use the partial order on $\mathbb{Z}^m$ that says $v \leq w$ whenever $v_i \leq w_i$ 
for $i=1,\ldots,m$.

Since the $\ybb_v$ form a surgery cube, they can be generated by the data of
$\ybb$ and a framed link $\lbb=\lbb_1\cup\cdots\cup\lbb_m$ in $\ybb$ as in 
\S \ref{sec:dehn}, where each $\lbb_i$ is an equivariant embedding of
$S^1\times D^2\times\so$ into $\ybb$. For $v < w$ we have surgery bundle cobordisms 
$\xbb_{vw}:\ybb_v\to\ybb_w$ constructed 
by iterating the construction for $\xbb_{ij}$ from \S \ref{sec:x} for each $\lbb_i$. 
To give a definition, first set $k=|w-v|_1$.
We choose a maximal chain $v=v(0)<v(1)<\cdots < v(k) = w$. Each $\xbb_{v(i)v(i+1)}$ 
may be viewed as a surgery bundle as defined in \S \ref{sec:x}, and we may set
\[
	\xbb_{vw} = \xbb_{v(k-1)v(k)}\circ\cdots\circ\xbb_{v(0)v(1)}.
\]
The choice of maximal chain does not affect the isomorphism type of $\xbb_{vw}$. 
In fact, the identification of (\ref{iso}) lends a more invariant 
interpretation: we may view $\xbb_{vw}$ as $\ybb_v\times [0,1]$ 
with, for each $i=1,\ldots,m$, a copy of $\hbb \cup_\psi \cdots \cup_\psi \hbb$ 
($w_i-v_i$ copies of $\hbb$) 
attached to $\ybb_v\times \{1\}$ via the framed knot $\Lambda^{v_i+1}(\lbb_i)$. 
We have the isomorphism
\[
	\xbb_{vw}\simeq\xbb_{uw}\circ\xbb_{vu}
\]
whenever $v<u<w$. We write $\bf{0}$ for the element of $\mathbb{Z}^m$ with all zeros, and similarly 
$\bf{n}$ for the element with all elements equal to $n\in\mathbb{Z}$. Note 
that $\xbb_{\bf{0}\bf{3}}$ is not $\xbb_{\bf{0}\bf{0}}=\ybb_{\bf{0}}\times [0,1]$, 
but for instance $\xbb_{\bf{0}\bf{1}}\simeq \xbb_{\bf{3}\bf{4}}$. The base space of 
$\xbb_{vw}$ is written $\x_{vw}$. In the sequel we will only consider $\xbb_{vw}$ with 
$|w-v|_\infty \leq 3$.

As in the case when $\lt$ had one component, we have distinguished hypersurfaces in 
the interior of $\x_{vw}$. Of course, the 3-manifolds $\y_{u}\subset \x_{vw}$ for 
$v < u < w$ are the first examples. Note that $\y_{u}$ and $\y_{u'}$ are disjoint if 
and only if $u < u'$ or $u' < u$. For each $i\in\{1,\ldots,m\}$ and $k$ with 
$v_i<k<w_i$ we have a 3-sphere $\s_k^i$ in $\x_{vw}$ which generalizes $\s_{1}\subset\x_{02}$ 
from \S \ref{sec:decomp1}. The spheres $\s_k^i$ and $\s_l^j$ intersect if and only if $i=j$ and 
$|k-j|\leq 1$, and $\s_k^i$ intersects $\y_u$ if and only if $u_i=k$. For $v,w\in\mathbb{Z}^m$ 
with $v< w$ and $|w-v|_\infty \leq 2$ we define a set of hypersurfaces in $\x_{vw}$:
\[
	\mathcal{H}_{vw} = \{ \y_u : v < u < w \} \cup \{ \s^i_{k} : 1\leq i\leq m,\; v_i < k < w_i\}.
\]
Note that the second set is empty if $|w-v|_\infty<2$.

We obtain a family of metrics $G_{vw}=G(\mathcal{H}_{vw})$ on $\x_{vw}$ as constructed in \S \ref{sec:met}.
The space of metrics $G_{vw}$ is a convex polytope called a graph-associahedron, and 
\[
	\dim G_{vw} = |w-v|_1-1,
\]
as Bloom explains in \cite[Thm. 5.3]{bloom}. In fact, 
when $|w-v|_\infty<2$, $G_{vw}$ is the permutahedron $P_N$, the convex polytope 
defined as the convex hull in $\mathbb{R}^N$ of all permutations of 
$(1,2,\ldots,N)\in\mathbb{R}^N$ where $N= |w-v|_1$. For example, $P_3$ is a hexagon, 
and the polytope $P_4$ is shown (hollowed out) in Figure \ref{fig:perm}. 
Write $m_{vw}=m_{G_{vw}}(\xbb_{vw})$ and $\partial_v$ for the differential of 
$\chain(\ybb_v)$. From the formulae in \S \ref{sec:met} we obtain
\begin{equation}
	 (-1)^{|w-v|_1-1}m_{vw}\partial_{v}-\partial_w m_{vw} = \sum_{v < u < w} m_{G(\y_u)}(\xbb_{vw})
	 + \sum_{\substack{1\leq i \leq m\\w_i <k < v_i}} m_{G(\s_k^i)}(\xbb_{vw}).\label{mets}
\end{equation}
As in \S \ref{sec:firstmaps}, each $m_{G(\s_k^i)}(\xbb_{vw})=0$. Also, the family $G(\y_u)$ 
can be identified with the product $G_{vu}\times G_{uw}$. Before we apply equation (\ref{met3}), 
we discuss the arrangement of signs.

\begin{figure}[t]
\includegraphics[scale=.5]{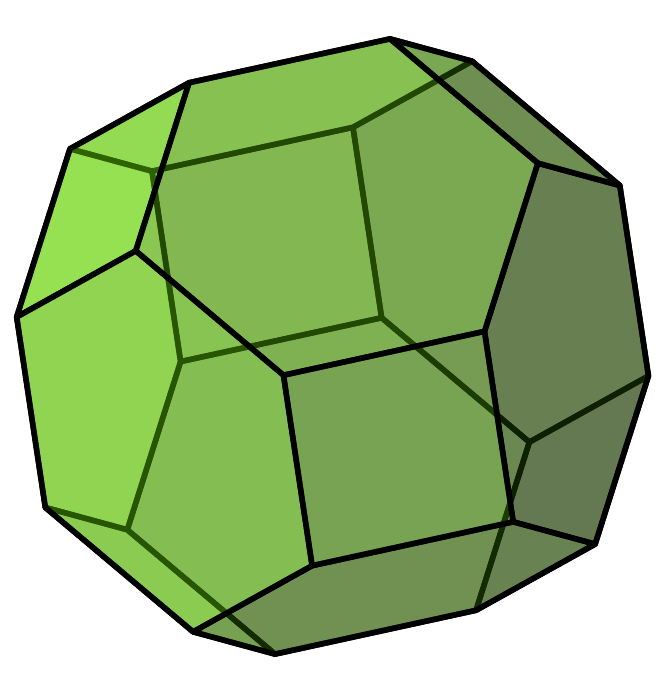}
\caption{The permutahedron $P_4$.}
\label{fig:perm}
\end{figure}

It is possible to choose I-orientations $\mu_{vw}$ for $\xbb_{vw}$ such that 
$\mu_{vw}=\mu_{uw}\circ\mu_{vu}$ whenever $v < u < w$, and we do so. 
For a proof, see \cite[Lemma 6.1]{kmu}. We can orient each $G_{vw}$ such that 
the identification of $G_{vu}\times G_{uw}$ with $G(\y_u)\subset \partial G_{vw}$ 
has orientation deficiency $(-1)^{\dim G_{vu}}$. That is, the product orientation for 
$G_{vu}\times G_{uw}$ using our chosen orientations differs from the
boundary orientation as induced from $G_{vw}$ by the sign $(-1)^{\dim G_{vu}}$. This essentially
follows from the discussion in \cite{kmu} following Prop. 6.4. With this understood, equation 
(\ref{met3}) yields
\[
	m_{G(\y_u)}(\xbb_{vw}) = (-1)^{(\dim G_{uw}+1)\dim G_{vu}}m_{uw}m_{vu}.
\]
Writing $m_{vv}=\partial_v$, equation (\ref{mets}) becomes
\[
	\sum_{v\leq u \leq w} (-1)^{|w-u|_1(|u-v|_1-1)}m_{uw}m_{vu}=0.
\]
We remind the reader that this holds under the assumptions 
that $v < w$ and $|w-v|_\infty \leq 2$. The case $v=w$ also holds, encoding the relation $\partial_v^2=0$.\\

\subsection{Constructing the Spectral Sequence}

We now construct the spectral sequence of Theorem \ref{thm:2}. We define a chain complex 
$(\chainbf,\partialbf)$ with a filtration $F^i \chainbf$. 
The filtration will induce the spectral sequence we desire. To begin, set
\begin{equation}
	\chainbf  = \bigoplus_{v\in\{0,1\}^m} \chain(\ybb_v), \quad \partialbf = 
	\sum_{v \leq w} \partial_{vw}\label{complex}
\end{equation}
where $\partial_{vw}= (-1)^{s(v,w)}m_{vw}$. The sign here is given by 
\[	
	s(v,w) = (|w-v|_1^2- |w-v|_1)/2 + |v|_1,
\]
as lifted from \cite[eq. 38]{kmu}. We compute the $\chain(\ybb_v)\to\chain(\ybb_w)$ component of $\partialbf^2$ to be
\[
	(-1)^{s(v,w)+|w|_1}
	\sum_{v\leq u\leq w} (-1)^{|w-u|_1(|u-v|_1-1)}m_{uw}m_{vu}=0.
\]
We call $(\chainbf,\partialbf)$ the \textit{link surgeries complex associated to} 
$(\ybb,\lbb)$, with the understanding that
the necessary auxiliary choices we've made have been fixed.

We define the filtration on $(\chainbf,\partialbf)$ by setting
\begin{equation}
	F^i{\chainbf} = \bigoplus_{|v|\geq i} \chain(\ybb_v) \subseteq \chainbf.\label{eq:filt}
\end{equation}
Since $\partialbf$ involves only terms with $v\leq w$, it is immediate that 
$\partialbf F^i{\chainbf} \subseteq F^i{\chainbf}$. This filtered complex 
induces a spectral sequence whose $E^1$-page and $E^1$-differential $d^1$ are given by
\[
	\e^1 = \bigoplus_{v\in\{0,1\}^m}\ih(\ybb_v), \quad d^1 = 
	\sum_{\substack{v < w\\|w-v|_1=1}} (-1)^{\delta(v,w)}m(\xbb_{vw}),
\]
where $\delta(v,w)\equiv \sum_{1\leq i \leq j}v_i$, in which
$j$ is the unique index where $v$ and $w$ differ. This carries over from the discussion 
following \cite[Cor. 6.9]{kmu}. To prove Theorem \ref{thm:2} it remains 
to identify the $\e^\infty$-page: we must show that the homology of 
$(\chainbf,\partialbf)$ is the instanton homology $\ih(\ybb)$.\\

\subsection{Convergence}

Let $(\chainbf,\partialbf)$ be the link surgeries complex associated to $(\ybb,\lbb)$. 
For $i\in\mathbb{Z}$ define the chain complex $(\chainbf_i,\partialbf_i)$
to be the link surgeries complex associated 
to $(\ybb_{\Lambda^{i+1}}(\lbb_1),\lbb\setminus \lbb_1)$.
Recall that the notation $\ybb_{\Lambda^{i+1}}(\lbb_1)$ is from \S 3.2, and stands for $\Lambda^{i+1}$-surgery on $\lbb_1$ in $\ybb$.
We conflate $\infty$ and $-1$ in the following.
Note that for $i=\infty,0,1$ and $a,b\in\mathbb{Z}^{m-1}$ 
we have $(\partialbf_i)_{ab} = \partial_{vw}$ 
where $v=(i,a)$ and $w=(i,b)$. Thus we can work exclusively with the maps $\partial_{vw}$ 
with $v,w\in\mathbb{Z}^m$. Consider the map $\textbf{f}_0:\chainbf_0\to \chainbf_1$ given by
\[
	\textbf{f}_0= \sum_{\substack{v,w\in\{0,1\}^{m}\\ v_1=0, w_1 = 1}} \partial_{vw}.
\]
It should be understood that $\partial_{vw}=0$ if $v\not\leq w$. In words, $\textbf{f}_0$ 
is the sum of the components in the differential $\partialbf$ that correspond to 
surgery-cobordisms that include surgery on $\lt_1$. This is an anti-chain map, 
and the larger complex $(\chainbf,\partialbf)$ is the cone-complex of $\textbf{f}_0$. That is,
\[
	\chainbf = \chainbf_0\oplus\chainbf_1, \quad \partialbf  =  
	\left( \begin{array}{cc} \boldsymbol{\partial}_0 & 0 \\ \mathbf{f}_0 &
\boldsymbol{\partial}_1 \end{array} \right).
\]
Define a map $\textbf{F}:\chainbf_{\boldsymbol{\infty}}\to\chainbf$ by
\[
	\textbf{F} = \sum_{\substack{v_1=-1\\w_1\in\{0,1\}}} \partial_{vw}.
\]
This is an anti-chain map: the relation $\textbf{F}\partialbf_{\boldsymbol{\infty}}+ \partialbf\textbf{F}=0$ 
is an encoding of (\ref{mets}) via
\[
	\sum_{\substack{v_1=u_1=-1\\w_1\in\{0,1\}}} \partial_{uw}\partial_{vu} + 
	\sum_{\substack{v_1=-1\\u_1,w_1\in\{0,1\}}} \partial_{uw}\partial_{vu}=0.
\]
Equip $\chainbf$ and $\chainbf_\infty$ with filtrations as in (\ref{eq:filt}) but using 
the sum $\sum_{i=2}^m v_i$ instead of $|v|_1$. Then $\textbf{F}$ respects these 
filtrations, and on the $E^0_p$-components of the induced spectral sequences, 
the map induced by $\textbf{F}$ takes the form
\[
	\textbf{F}^0_p:\bigoplus_{\substack{v_1=-1\\ \sum_{i \geq 2} v_i = p}}\chain(\ybb_v) \to 
	\bigoplus_{\substack{v_1\in\{0,1\} \\ \sum_{i \geq 2} v_i = p}}\chain(\ybb_v)
\]
and for $v$ with $v_1=-1$ is given by
\[
	\textbf{F}^0_p|_{\chain(\ybb_v)} = \partial_{vv'}\oplus \partial_{vv''}
\]
where $v'$, $v''$ have $v'_1=0$ and $v''_1=1$, and otherwise agree with $v$. 
But $\partial_{vv'}$ is the map $f_{-1}$ in \S \ref{sec:firstmaps} for the surgery triangle 
involving $\ybb_v$ and $\lt_1$; and likewise $\partial_{vv''}$ is the map $h_{-1}$. 
It follows from Lemma \ref{alg} that $\textbf{F}^0$ is a quasi-isomorphism, and 
hence so is $\mathbf{F}$. By
removing each link component as we have just done for $L_1$, and composing the $m$ 
maps $\textbf{F}$ associated to each removal, we get a 
quasi-isomorphism $\mathbf{Q}$ from $(\chain(\ybb),\partial)$ to $(\chainbf,\partialbf)$, 
completing the proof of Theorem \ref{thm:2}.\\

\subsection{Gradings}\label{sec:linksgr}

We follow Bloom's \cite{bloom} treatment of gradings for the spectral sequence. 
We refer to the mod 2 grading on the complex $\chain(\ybb)$ defined 
in \S \ref{sec:instgrading} as $\text{gr}[\ybb]$.
We define a grading $\text{gr}[{\chainbf}]$ on the 
complex $\chainbf$ in (\ref{complex}). For $x\in \chain(\ybb_v)\subset\mathbf{C}$ 
with homogeneous $\text{gr}[{\ybb_v}]$ grading, we define
\begin{equation}
	\text{gr}[\chainbf](x) \equiv \text{gr}[{\ybb_v}](x) + \text{deg}(\x_{\boldsymbol{\infty}v}) + |v|_1 \mod 2.\label{eq:gr1}
\end{equation}
Recall that we conflate $\boldsymbol{\infty}$ with $-\mathbf{1}\in\mathbb{Z}^m$. Let $\pi_w:\chainbf\to\chain(\ybb_w)$ be the projection. Note that
\begin{equation}
	\text{gr}[\chainbf](\pi_w(\partialbf(x))) \equiv \text{gr}[\ybb_w](m_{vw}(x)) + \text{deg}(\x_{\boldsymbol{\infty}w})+|w|_1 \mod 2.\label{eq:gr2}
\end{equation}
We have the additivity relation $\text{deg}(\x_{\boldsymbol{\infty}w}) \equiv \text{deg}(\x_{\boldsymbol{\infty}v}) + \text{deg}(\x_{vw})$, and also
\[
	\text{gr}[\ybb_w](m_{vw}(x)) = \text{gr}[\ybb_v](x)+ \dim(G_{vw}) + \text{deg}(\x_{vw}).
\]
Knowing $\dim(G_{vw})=|w-v|_1-1$ shows that the expressions (\ref{eq:gr1}) and 
(\ref{eq:gr2}) differ by $1$ mod $2$, and 
thus the differential $\partialbf$ alters $\text{gr}[\chainbf]$ by $1$.

The quasi-isomorphism 
$\mathbf{Q}:\chain(\ybb)\to \chainbf$ is a composition of $m$ maps $\mathbf{F}$ as in the
previous section. Thus it is a sum of maps of the form $m_G(\xbb_{\boldsymbol{\infty}v})$, 
where $v\in \{0,1\}^m$ and $G = G_1\times\cdots\times G_m$. Here $G_i=G_{v(i)v(i+1)}$ 
only varies on $\x_{v(i)v(i+1)}\subset\x_{\boldsymbol{\infty}v}$ and 
$\boldsymbol{\infty}=-\mathbf{1} = v(1) < v(2) < \cdots < v(m+1)=v$.
Using $\text{dim}(G_{vw})=|w-v|_1-1$ 
for $v<w$, we find $\dim(G)=|v|_1$. Since the $\text{gr}[\ybb]$ to $\text{gr}[\ybb_v]$ degree of $m_G(\xbb_{\boldsymbol{\infty}v})$ is $\dim(G)+\text{deg}(\x_{\boldsymbol{\infty}v})$, it follows that $\mathbf{Q}$ preserves the $\mathbb{Z}/2$-gradings.

There is also a $\mathbb{Z}$-grading on $\chainbf$ given by the vertex weight $|v|_1$ for a homogeneous 
element in $\chain(\ybb_v)\subset\chainbf$, and by construction $\partialbf$ 
increases this by $1$. We summarize a more detailed statement of Theorem \ref{thm:2}; compare 
\cite[Cors. 6.9, 6.10]{kmu}  and \cite[Thm. 1.1]{bloom}.

\begin{theorem}
Let $\lt$ be an oriented, framed link with m components in $\y$. For each 
$v\in \{\infty,0,1\}^m$ denote by $\y_v$ the result of $v$-surgery on $\lt$ 
and let $\ybb_v$ be an admissible bundle over $\y_v$ such that the total 
collection of $\ybb_v$ forms a surgery cube. 
For $v<w$ there are surgery cobordism bundles $\xbb_{vw}$ from $\ybb_v$ to $\ybb_w$ with 
I-orientations $\mu_{vw}$ satisfying $\mu_{uw}\circ\mu_{vu}=\mu_{vw}$ whenever $v<u<w$,
such that there is a spectral sequence $(E^r,d^r)$ with
\[
	\text{{\em $E^1 = \bigoplus_{v\in\{0,1\}^m} \ih(\ybb_v), \qquad d^1 = \sum_{\substack{v < w\\ |w-v|_1 = 1}} (-1)^{\delta(v,w)}\ih(\xbb_{vw})$}}
\]
where $\delta(v,w)=\sum_{1\leq i \leq j} v_j$, in which $j$ is the unique index where 
$v$ and $w$ differ. 
The spectral sequence is graded by $\mathbb{Z}/2\times\mathbb{Z}$, where
$d^r$ has bi-degree $(1,r)$.
The $\mathbb{Z}/2$-grading is given by (\ref{eq:gr1}) while the $\mathbb{Z}$-grading 
is by vertex weight.
The spectral sequence converges by the $E^{m+1}$-page to $\ih(\ybb)$, and it induces the 
usual $\mathbb{Z}/2$-grading on $\ih(\ybb)$. \label{ss1}
\end{theorem}

%% file: framed.tex
\section{Framed Instanton Homology}

In this section we discuss the basic constructions and properties of the 
groups $\ih^\#(\y)$. These are a special case of the groups 
$\ih^\#(\y,\knot)$ introduced by Kronheimer and Mrowka in \cite{kmu}. 
Here $\y$ is a 3-manifold and $\knot$ is a knot or link in $\y$, 
and we have $\ih^\#(\y)=\ih^\#(\y,\emptyset)$. The name 
{\em framed instanton homology} comes from \cite{kmki}. The group $\ih^\#(\y)$ 
is isomorphic to the sutured instanton group $\text{SHI}(M,\gamma)$ 
from \cite{kms}, where $M$ 
is the complement of an open 3-ball in $\y$ and 
$\gamma$ is a suture on the 2-sphere boundary.\\

\subsection{Framed Instanton Groups}\label{sec:framed}
Let $\y$ be a connected, oriented, closed 3-manifold. 
Consider an $\so$-bundle $\ybb^\#$ over $\y\# T^3$ 
with $\ybb^\#$ trivial over $\y$ and non-trivial over $T^3$.
To make the construction of $\ybb^\#$ from $\y$ more precise, we can 
once and for all pick a point $x\in T^3$, a bundle $\tbb^3$ over $ T^3$ 
geometrically represented by an $S^1$-factor, and an isomorphism $\tbb^3_x\simeq\so$. 
Then, up to inessential choices, $\ybb^\#$ can be constructed from $\y$ and a basepoint 
$y\in\y$. Indeed, we can perform the connected sum $\y\# T^3$ between 
3-balls around $y$ and $x$, and glue the bundles $\y\times\so$ and $\tbb^3$ 
by expanding the isomorphism $\so\simeq \tbb_x^3$ near $x$.

We describe a useful operation for cobordisms in this context. 
Let $\x:\y_1\to\y_2$ be a cobordism and let $\gamma:[0,1]\to \x$ be a properly 
embedded path with $\gamma(0)$ and $\gamma(1)$ being the chosen basepoints in 
$\y_1$ and $\y_2$, respectively. Given another such pair $\x',\gamma'$ where 
$\x':\y'_1\to\y'_2$, we form a cobordism
\[
	\x\Join\x':\y_1\#\y_1'\to\y_2\#\y_2'
\]
as follows: let $\Gamma$ be a neighborhood of $\gamma$ diffeomorphic to 
$\text{int}(D^3)\times [0,1]$, and write
\[
	\partial(\x\setminus \Gamma) \setminus (\y_1\cup \y_2\setminus \partial\Gamma) = S^2\times [0,1];
\]
do the same for $\x'$, and identify the copies of $S^2\times [0,1]$ by
an orientation reversing homeomorphism. See 
Figure \ref{fig:join}. We omit the paths from the notation $\x\Join\x'$ 
because for all of our cobordisms there will be a natural choice of path 
up isotopy relative to the boundaries. The operation $\Join$ extends to 
glue together cobordisms of bundles $\xbb$ and $\xbb'$ if a path of 
isomorphisms $\xbb_{\gamma(t)}\simeq \xbb'_{\gamma'(t)}$ is chosen.

Let $g$ be a gauge transformation of $\ybb^\#$ with $\eta(g)\in H^1(\y\# T^3;\mathbb{F}_2)$ 
Poincar\'{e} dual to a 2-torus $\Sigma\subset T^3$ over which $\ybb^\#$ is non-trivial.
Here $\eta:\mathscr{G}(\xbb)\to H^1(\x;\mathbb{F}_2)$ is from the exact sequence (\ref{eta}).
Such a transformation may be constructed explicitly as in 
\cite[Lemma A.2]{ds}.
Define the framed gauge transformations $\mathscr{G}^\#$ 
to be the subgroup of $\mathscr{G}(\ybb^\#)$ generated by $\mathscr{G}_\text{ev}(\ybb^\#)$ 
and $g$. We let $\mathfrak{C}^\#$ denote the critical set of a perturbed 
Chern-Simons functional $\textbf{cs}_\pi$ on $\mathscr{C}/\mathscr{G}^\#$.
Note that $\mathfrak{C}^\#$ is obtained from $\mathfrak{C}(\ybb^\#)$ 
by modding out by the $\mathbb{Z}/2$-action of degree 4 
induced by the gauge transformation $g$. 

We define the chain complex $\chain^\#(\y)$ for $\ih^\#(\y)$
following ideas from \cite[\S 4.4]{kmu}. 
This definition transparently replaces the notion of an I-orientation 
with that of a homology orientation.
Fix once and for all a bundle 
$\mathbb{W}:S^3\times\so\to \tbb^3$ over $T^2\times D^2\setminus \text{int}(D^4):S^3\to T^3$ 
extending $\tbb^3$. Fix a path $\gamma$ in $W$ beginning in $S^3$, 
ending at $x\in T^3$, and 
a path of isomorphisms $\mathbb{W}_{\gamma(t)}\simeq\so$, the isomorphisms at the ends being the 
natural choices. 
We define
\[
	\chain^\#(\y) = \bigoplus_{\connag\in\mathfrak{C}^\#} \mathbb{Z}\Lambda^\#(\connag)
\]
where $\Lambda^\#(\connag)$ is the 2-element set of orientations of the 
line $\text{det}(D_A)$; here $A$ is a connection on $[0,1]\times\ybb\Join\mathbb{W}$ 
(with cylindrical ends attached) where the 
limit of $A$ over the $\mathbb{R}\times \ybb$ cylindrical end is equivalent to 
the trivial connection, and the limit of $A$ over the 
$\mathbb{R}\times \ybb^\#$ end is in the class $\connag$. The operator $D_A$ 
is as in \S \ref{sec:index}.

The differential for $\chain^\#(\y)$ is straight-forward to define,
following the construction of the differential for $\ih(\ybb)$ in \S \ref{sec:instantongroups}, which followed 
\cite[\S 3.6]{kmu}. Note that a base connection as in the definition for $\chain(\ybb)$ is no longer needed. 
In summary, given $\y$ with a basepoint, with suitable metric and 
perturbation, the complex $\chain^\#(\y)$ and hence the group $\ih^\#(\y)$ are determined. 
The isomorphism class of $\ih^\#(\y)$ depends only on $\y$.\\

\begin{figure}[t]
\includegraphics[scale=.38]{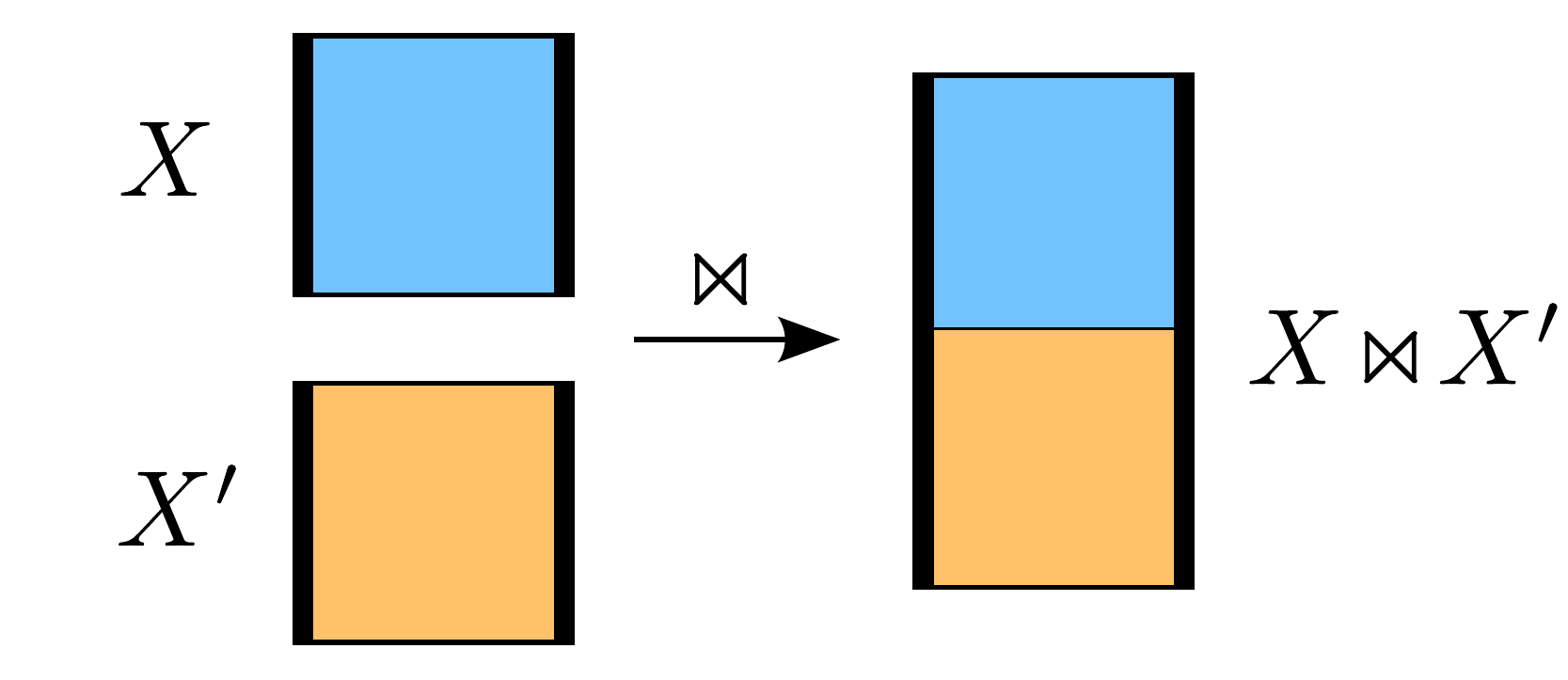}
\caption{A schematic depiction of the $\Join$ operation. The thicker
lines represent actual boundary components.}
\label{fig:join}
\end{figure}

\subsection{Maps from Cobordisms}\label{sec:cobs}
We describe how a cobordism $\x:\y_1\to\y_2$ with a path $\gamma$ as above 
gives rise to a map $\ih^\#(\x):\ih^\#(\y_1)\to\ih^\#(\y_2)$. Again, 
we omit $\gamma$ from the notation because there will always be a natural choice for us. 
We always assume $\x$ and $\y_1,\y_2$ are connected.
Take the path in $ T^3\times [0,1]$ given by $t\mapsto (x,t)$. 
Using this we form a cobordism
\[
	\x^\#=\x\Join( T^3\times [0,1]):\y_1\# T^3\to \y_2\# T^3.
\]
Further, there is a natural choice for bundle $\xbb^\#$ over $\x^\#$ by 
performing the $\Join$ operation between $\x\times\so$ and $\tbb^3\times [0,1]$ 
using the constant path of isomorphisms $\so\simeq \tbb_x^3$.
We enlarge the even gauge transformation 
group used for $\xbb^\#$ to include gauge transformations 
whose restriction to each $T^3$ is of the form $g$ from \S \ref{sec:framed} above. 
See \cite[\S 5.1]{kmu} for a general discussion. Then, in the usual way,
we obtain a chain map $m^\#(\x):\chain^\#(\y_1)\to\chain^\#(\y_2)$ and an induced map 
$\ih^\#(\x)$ on homology.

The data of an I-orientation 
may be replaced by an orientation of the line $\text{det}(D_A)$ where $A$ is the 
trivial connection on $\x\times\so$. Following \cite[\S 3.8]{kmki}, but using 
homology instead of cohomology, this amounts to an orientation of the vector space
\[
	\mathcal{L}(\x):=H_1(\y_1;\mathbb{R})\oplus H_1(\x;\mathbb{R})\oplus H_2^+(\x;\mathbb{R}),
\]
where $H_2^+(\x;\mathbb{R})$ is a maximal positive definite subspace for the 
intersection form on $H_2(\x;\mathbb{R})$. 
A choice of such an orientation is called a \textit{homology orientation} 
for the cobordism $\x$, and is typically denoted $\mu_\x$. In summary, given 
$\x:\y_1\to\y_2$, a path $\gamma$ from the basepoint of $\y_1$ to the basepoint of $\y_2$, 
a suitable perturbation and metric, and a homology orientation 
$\mu_\x$, the chain map $m^\#(\x)$ is determined. The induced map $\ih^\#(\x)$ depends on $\x$, $\mu_\x$, and presumably $\gamma$.

We define $\ih^\#(\emptyset)=\ih^\#(S^3)$, and when $\x:\emptyset\to \partial \x$, 
we define the map $\ih^\#(\x)$ by deleting a 4-ball in $\x$. In particular, when 
$\x$ is a compact, connected, oriented 4-manifold with connected boundary, and an orientation of 
$H_1(\x;\mathbb{R})\oplus H_2^+(\x;\mathbb{R})$ is chosen, we obtain an element
\[
	[\x]^\# \in \ih^\#(\partial \x).
\]
We also obtain a map $[\x]_\#:\ih^\#(\overline{\partial \x})\to\mathbb{Z}$
by viewing $\x:\overline{\partial \x}\to\emptyset$ and orienting 
$H_1(\partial\x;\mathbb{R})\oplus H_1(\x;\mathbb{R})\oplus H_2^+(\x;\mathbb{R})$. 
If $\x$ is a closed, connected, oriented 4-manifold and $H_1(\x;\mathbb{R})\oplus H_2^+(\x;\mathbb{R})$ 
is oriented, then we have a number $[\x]^\#\in\mathbb{Z}$.

Finally, we mention another topological operation that arises naturally in 
this setting. This is the boundary sum $W\natural W'$ of two 4-manifolds 
with boundary, as used in \cite{gs}; one deletes a model half-4-ball along the boundaries of 
$W$ and $W'$ and glues them together with an 
orientation-reversing homeomorphism, so that $\partial(W\natural W')=\partial W\#\partial W'$. 
We have
\[
	\left(\x\Join\x'\right)\circ \left(W\natural W'\right) \simeq 
	\left(\x\circ W\right)\natural\left(\x'\circ W'\right)
\]
where compositions involved are of course assumed to make sense, and the same relation holds 
with the compositions reversed. See Figure \ref{fig:natural}.\\

\begin{figure}[t]
\includegraphics[scale=.33]{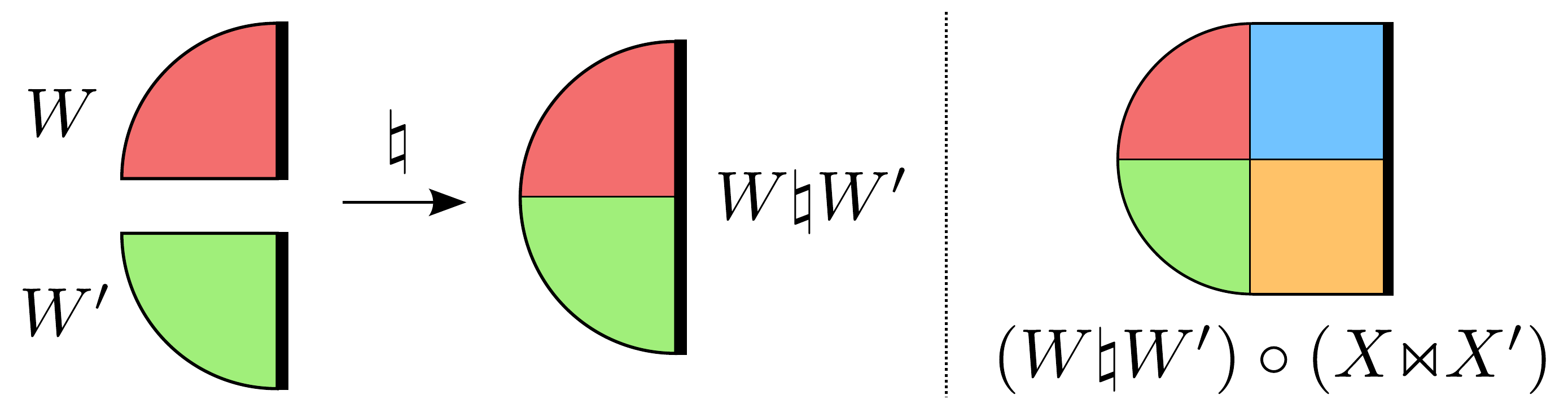}
\caption{On the left, a schematic depiction of the boundary sum $\natural$ operation. On the right, 
we compose the $\Join$ operation against $\natural$, and the result may be interpreted as involving only $\natural$. The thick lines represent actual boundary components.}
\label{fig:natural}
\end{figure}

\subsection{Grading} \label{sec:framedgr}

We now define the absolute $\mathbb{Z}/4$-grading on $\ih^\#(\y)$. 
Let $\mathbb{W}'$ be a completion of $\mathbb{W}$ from \S \ref{sec:framed} 
with the 4-ball filled in, so that it is 
a non-trivial bundle over $T^2\times D^2$, and we may write $\mathbb{W}':\emptyset \to \mathbb{T}^3$. Fix an integer $k$. 
For $\connag\in\mathfrak{C}^\#(\y)$ we define 
\[
	\text{gr}(\connag) := -\mu(\ebb \;\natural\; \mathbb{W}',\connag) -b_1(E)+b_+(E)-b_1(Y) + k \mod 4
\]
where $E:\emptyset\to\y$ is a 4-manifold with boundary $\y$ and $\ebb=E\times\so$. 
We choose $k$ such that $\ih^\#(S^3)$ 
is supported in grading $0$. The proof that this grading is 
well-defined is the same as the case of the absolute mod 2 grading for $\ih(\ybb)$ as for example in 
\cite{d}; we get $\mathbb{Z}/4$ instead of $\mathbb{Z}/2$ because the characteristic classes
of the bundles are uniformly controlled in this case. We
give the argument for completeness, and compute the degrees of cobordism maps. 
We have chosen our conventions so that the degree formula 
aligns with that of \cite[Prop. 4.4]{kmu}.

\begin{prop}
The assignment \emph{$\connag\mapsto \text{gr}(\connag)$} gives a well-defined 
$\mathbb{Z}/4$-grading on $\chain^\#(\y)$ for which the differential has 
degree $-1$ and thus descends to a $\mathbb{Z}/4$-grading on $\ih^\#(Y)$. 
Given a cobordism $X:Y_1\to Y_2$ equipped with the data to form $\x^\#$ as in \S \ref{sec:cobs}, the 
degree of the induced map 
$\ih^\#(\x):\ih^\#(Y_1)\to\ih^\#(Y_2)$ is given by the expression for deg$(\x)$ in (\ref{degofmap}) taken
modulo 4. More generally, if $\ybb_i=\y_i\times\so$ and $\xbb:\ybb_1\to\ybb_2$ is possibly non-trivial and comes equipped with the 
data to form $\xbb^\#$, then the degree of the induced map $\ih^\#(\xbb):\ih^\#(Y_1)\to\ih^\#(Y_2)$ is given by
\begin{equation}
	-\frac{3}{2}(\chi(\x)+\sigma(\x)) + \frac{1}{2}(b_1(\y_2)-b_1(\y_1)) + 2\relp(\xbb) \mod 4\label{eq:degmaster}
\end{equation}
where the invariant $\relp(\xbb)\in\mathbb{Z}/2$ is defined by 
\[
	\relp(\xbb) \equiv [S]\cdot [S] \mod 2.
\]
Here $S\subset\x$ is a surface in the interior of $\x$, $[S]\in H_2(X;\mathbb{F}_2)$, and 
the image of $[S]$ in $H_2(X,\partial X;\mathbb{F}_2)$ is Poincar\'{e} dual to $w_2(\xbb)$.\label{gradingprop}
\end{prop}

\begin{proof}
Let $E':Y\to\emptyset$ and $\ebb'=E'\times\so$, 
and let $\wbb''$ be the reverse of $\wbb'$. In particular, we may write $\mathbb{W}'':\mathbb{T}^3\to\emptyset$. 
Then by (\ref{glue}) we have
\begin{equation}
	\mu(\ebb\;\natural\;\wbb',\connag) + \mu(\connag,\ebb'\;\natural\;\wbb'') = \mu((\ebb'\circ\ebb)\#(\wbb''\circ\wbb')).\label{eq:gradeglue}
\end{equation}
By (\ref{glue}) we may write the right hand side as
\[
	\mu(\ebb'\circ\ebb) + 3 + \mu(\wbb''\circ\wbb').
\]
Note $\wbb''\circ\wbb'$ is a bundle over $T^2\times S^2$, which necessarily has $p_1$ congruent to 
$0$ mod $4$. Also, $(1-b_1+b_+)(T^2\times S^2)=0$. By (\ref{closed}) we conclude that $\mu(\wbb''\circ\wbb')$ is 
congruent to $0$ mod $4$. Noting that $\ebb'\circ\ebb$ is a trivial bundle, 
(\ref{eq:gradeglue}) is mod 4 congruent to 
\[
 	\mu(\ebb'\circ\ebb) + 3 = 3(b_1-b_+)(E'\circ E)
\]
which by a Mayer-Vietoris argument (see \S \ref{sec:homor}) is mod 4 congruent to
\[
	-b_1(E)-b_1(E') + b_+(E) + b_+(E')+ b_1(Y).
\]
It follows that the expression
\[
	\text{gr}(\connag) - \mu(\connag,\ebb'\;\natural\;\wbb'') = b_1(E') - b_+(E') - 2b_1(Y)+k  \mod 4
\]
is independent of $\ebb$, and thus so is $\text{gr}(\connag)$.
In other words, $\text{gr}(\connag)$ is a well-defined $\mathbb{Z}/4$-grading on $\chain^\#(Y)$. 
Suppose $\connag,\connbg\in\mathfrak{C}^\#(Y)$ with $\mu(\connag,\mathbb{R}\times\ybb^\#,\connbg)=1$. 
Then
\[
	\mu(\ebb\;\natural\;\wbb',\connag)+\mu(\connag,\mathbb{R}\times\ybb^\#,\connbg) = \mu(\ebb\;\natural\;\wbb',\connbg)
\]
yields $\text{gr}(\connbg)-\text{gr}(\connag) = -1$. It follows that the differential 
lowers the grading by 1. Now we compute the degree of a map $\ih^\#(\x)$ induced by a cobordism 
$\x:\y_1\to\y_2$.
Let $\xbb=\x\times\so$ and form $\vbb = \xbb\Join(\tbb^3\times [0,1])$. 
Let $\connag\in \mathfrak{C}^\#(\y_1)$ and $\connbg\in \mathfrak{C}^\#(\y_2)$ 
with $\mu(\connag,\vbb,\connbg)=0$. 
Let $E:\emptyset\to\y_1$ and $\ebb=E\times\so$. Then (\ref{glue}) and $\mu(\connag,\vbb,\connbg)=0$ yield 
$\mu(\vbb\circ(\ebb\;\natural\;\wbb'),\connbg) = \mu(\ebb\;\natural\;\wbb',\connag)$. 
Thus $\text{deg}(\x) \equiv \text{gr}(\connbg)-\text{gr}(\connag)$ is given by
\[
	-b_1(\x\circ E) + b_+(\x\circ E) - b_1(Y_2) + b_1(E) - b_+(E) + b_1(Y_1).
\]
From the discussion in \S \ref{sec:homor}, $-b_1(\x\circ E) + b_+(\x\circ E)$ is equal to
\[
	 -b_1(E)-b_1(X) + b_+(E) + b_+(X) + b_1(Y_1).
\]
We obtain the simplified expression
\begin{equation}
	\text{deg}(X) \equiv -b_1(X) + b_+(X) + 2b_1(Y_1) - b_1(Y_2) \mod 4.\label{eq:degprop}
\end{equation}
Using the assumption that $\x$, $\y_1$ and $\y_2$ 
are connected and non-empty, we have $\chi(X) = 1-b_1(X)+b_2(X)-b_3(X)$.
Poincar\'{e}-Lefschetz duality tells us $b_3(X)=b_1(X,\partial X)$, and by the long 
exact sequence for the pair $(X,\partial X)$ with real coefficients we obtain
\[
	d-b_2(X)+b_1(\partial X)-b_1(X)+b_1(X,\partial X) -b_0(\partial X) + b_0(X) = 0,
\]
where $d$ is the dimension of the image of the map $H_2(X)\to H_2(X,\partial X)$.
Note $b_0(\partial X)=2$ and $b_0(X)=1$. On the other hand, $d=b_+(X)+b_-(X)$ 
and $\sigma(X) = b_+(X)-b_-(X)$. We obtain
\[
	\chi(X) = -2b_1(X) + b_1(Y_1) + b_1(Y_2) + d,\quad \sigma(X) = 2b_+(X) - d.
\]
Plugging this data into expression (\ref{degofmap}), rewritten here as
\[
	-\frac{3}{2}(\chi(\x)+\sigma(\x)) + \frac{1}{2}(b_1(\y_2)-b_1(\y_1)),
\]
yields, modulo 4, the expression for deg$(\x)$ in (\ref{eq:degprop}). Now we approach the more general statement, supposing that $\xbb:\ybb_1\to\ybb_2$ is possibly non-trivial.
We write
\begin{equation*}
\text{deg}(\xbb) \equiv \text{deg}(\x) + 2\relp(\xbb) \mod 4,\label{eq:grp1}
\end{equation*}
where $\relp(\xbb)$ is to be determined.
Let $E_1:\emptyset\to\y_1$ and $\ebb_1=E_1\times\so$. Given $\connag\in\mathfrak{C}^\#(\y_1)$, choose $\connbg\in\mathfrak{C}^\#(\y_2)$ such that 
$\mu(\connag,\xbb^\#,\connbg)\equiv 0$. Write $\xbb_\text{tr}=\x\times\so$. Then
\begin{align*}
	\text{deg}(\xbb)-\text{deg}(\x) &\equiv \mu(\xbb\circ\ebb_1\;\natural\;\wbb',\connbg)- \mu(\xbb_\text{tr}\circ\ebb_1\;\natural\;\wbb',\connbg).
\end{align*}
After closing up bundles using some $E_2:\y_2\to\emptyset$ with $\ebb_2=E_2\times\so$ 
and cancelling out the contribution from the bundle over $T^2\times S^2$ as above, this difference is seen from (\ref{closed}) to be
\[
	-2p_1(\ebb_2\circ\xbb\circ\ebb_1) = \frac{1}{4\pi^2}\int_{E_2\circ X\circ E_1}\text{tr}(F_A^2),
\]
where $A$ is any connection. We can choose $A$ to be trivial away from the interior of $X$, thus
\[
	\relp(\xbb) \equiv \frac{1}{8\pi^2}\int_X \text{tr}(F_A^2) \mod 2
\]
where $A$ is any connection on $\xbb$ that restricts to trivial connections on each $\ybb_i$.
In other words, $\mathscr{P}(\xbb)\equiv p_1(\xbb')$ mod $2$, where $\xbb'$ 
is any trivial extension of $\xbb$ over a closed 4-manifold. 
Thus
\begin{equation*}
	\relp(\xbb)\equiv \widetilde{w}_2(\xbb)^2 \mod 2,
\end{equation*}
where $\widetilde{w}_2(\xbb)$ is a lift of $w_2(\xbb)$ to $H^2(\x,\partial\x;\mathbb{F}_2)$. 
The result follows.
\end{proof}
\vspace{10px}

\subsection{Duality}\label{sec:dual}

The chain group $\chain^\#(\overline{\y})$ is the same as $\chain^\#(\y)$ but 
with the differential maps transposed. It follows that $\ih^\#(\y)$ and $\ih^\#(\overline{Y})$ 
are isomorphic over $\mathbb{Q}$. More precisely, given a homology orientation of $\y$, 
i.e. an orientation of $H_1(\y;\mathbb{R})$, we get an isomorphism
\begin{equation}
	\ih^\#(\overline{\y};\mathbb{Q})_i \simeq \ih^\#(\y;\mathbb{Q})_{b_1(\y)-i}^\ast.\label{eq:dual}
\end{equation}
The homology orientation is required to identify the chain groups. The 
grading shift in (\ref{eq:dual}) is explained as follows. Let $E_1:\emptyset \to \y$ and 
$E_2:\y\to\emptyset$, and $\connag\in\mathfrak{C}^\#(\y)$. Write $\overline{\connag}$ 
for the corresponding class in $\mathfrak{C}^\#(\overline{\y})$. 
From (\ref{glue}) and (\ref{closed}) we obtain 
\[
	\mu(\ebb_1\;
	\natural\; \mathbb{W}',\connag) + \mu(\overline{\connag},\ebb_2\;
	\natural\;\mathbb{W}'') = 3(b_1(E)-b_+(E))
\] 
where $\ebb_i=E_i\times\so$, $E=E_2\circ E_1$, and $\mathbb{W}''$ is the reverse of $\mathbb{W}'$. The bundle $\mathbb{W}''\circ\mathbb{W}'$ over $T^2\times S^2$ has been removed from 
the expression just as in \S \ref{sec:framedgr}. Using that 
$b_1(E)-b_+(E)$ is equal to
\[
	b_1(E_1)+b_1(E_2) - b_+(E_1) - b_+(E_2)-b_1(Y),
\]
see \S \ref{sec:homor}, we obtain $\text{gr}(\connag) + \text{gr}(\overline{\connag})\equiv b_1(Y) + 2k$. We claim $k$ is even. Let $\connag$ be the generator of $I^\#(S^3)$, represented by a flat connection on $T^3\simeq S^3\# T^3$. Recall that $k$ is chosen so that $I^\#(S^3)$ is supported in grading $0$, so we have $\text{gr}(\connag) =0$ (also see \S \ref{sec:examples}). In the definition of $\text{gr}(\connag)$, choose $\mathbb{E}:\emptyset\to S^3$ to be a trivial bundle over a 4-ball. Then
\[
	0 \equiv \text{gr}(\connag) \equiv -\mu(\mathbb{W}',\connag) + k \mod 4.
\]
Recall from the proof of Prop. \ref{gradingprop} that $\mu(\mathbb{W}''\circ\mathbb{W}')\equiv 0$ mod $4$, where $\mathbb{W}'':\mathbb{T}^3\to\emptyset$ is the reverse bundle-cobordism of $\mathbb{W}'$. By the index gluing formula (\ref{glue}) we then have $-\mu(\mathbb{W}',\connag) \equiv \mu(\connag,\mathbb{W}'')$ mod $4$. Since $W'$ is diffeomorphic to its orientation reversal, which is $W''$, we also have  $\mu(\mathbb{W}',\connag) \equiv \mu(\connag,\mathbb{W}'')$ mod $4$, as follows from the Atiyah-Patodi-Singer index formula \cite[Thm. 3.10]{aps}. Thus $k \equiv \mu(\mathbb{W}',\connag) \equiv 0$ mod $2$. It follows that
\[
	\text{gr}(\connag) + \text{gr}(\overline{\connag})\equiv b_1(Y) \mod 4,
\]
establishing the grading shift in (\ref{eq:dual}).\\

\subsection{Exact Triangles}

In this section we state a few exact triangles for framed instanton homology. 
For these it is necessary to allow non-trivial bundles. 
In the above constructions, take $\ybb^\#$ to be geometrically 
represented by $\lambda\cup\omega$ where $\lambda\subset\y$ and $\omega$ is an 
$S^1$-factor of $T^3$. We obtain a group $\ih^\#(\y;\lambda)$ that is 
now only relatively $\mathbb{Z}/4$-graded. It is isomorphic to four consecutive
gradings of the relatively $\mathbb{Z}/8$-graded group $\ih(\ybb^\#)$. 
The isomorphism class of $\ih^\#(\y;\lambda)$ depends only on the oriented homeomorphism 
type of $\y$ and the class $[\lambda]\in H_1(\y;\mathbb{F}_2)$.

Let $\y$ be a closed, oriented 3-manifold and $\lambda\subset\y$ a closed, unoriented 
1-manifold as above. Let $K$ be a framed knot in $\y$ disjoint from $\lambda$. 
Denote by $\y_i$ the result of $i$-surgery on $K$. 
Let $\mu$ be the core of the knot $K$ as viewed in $\y_0$. Then we have an 
exact triangle
\begin{equation*}
\cdots \ih^\#(\y;\lambda)\to\ih^\#(\y_0;\lambda\cup\mu)\to \ih^\#(\y_{1};\lambda) \to \ih^\#(\y;\lambda)\cdots
\end{equation*}
There are two other exact triangles corresponding to the two other rows in Figure \ref{fig:ses}. 
For example, if we view $\mu$ as the core of the knot inside $\y_i$ where $i=\infty$ or $i=1$, 
the exact sequence has $\mu$ appearing in the twisting for the group of $\y_i$, and not the other two.
Each of these is an application of Floer's original exact triangle, Theorem \ref{thm:floer},
obtained by connected summing each 3-manifold with $T^3$
and performing the surgeries away from $T^3$, with the appropriate overlying bundles.

By changing the framing of $K$, we obtain variants of the above triangles 
that are computationally handy. Let $l$ and $m$ be the longitude and meridian 
of $K$, respectively. Suppose the meridian is unchanged but the longitude 
is changed to $-pm+l$. Then we have
\begin{equation*}
\cdots \ih^\#(\y;\lambda)\to\ih^\#(\y_{p};\lambda\cup\mu)\to \ih^\#(\y_{p+1};\lambda)\to\ih^\#(\y;\lambda) \cdots
\end{equation*}
where again the core $\mu$ can be arranged in two other ways. 
Alternatively, keep the longitude the same but change the meridian to $m-ql$. Then
we have
\begin{equation*}
\cdots \ih^\#(\y_0;\lambda)\to\ih^\#(\y_{1/(q+1)};\lambda\cup\mu)\to \ih^\#(\y_{1/q};\lambda)\to\ih^\#(\y_0;\lambda) \cdots
\end{equation*}
where the same freedom with the placement of $\mu$ is understood. For other 
variants, we refer the reader to \cite[\S 42.1]{kmm}.

For an alternative perspective, one can begin with a 3-manifold $Z$ with torus boundary and consider 
the possible ordered triplets of Dehn fillings of $Z$ that are compatible with a surgery 
triangle description. 
This is the viewpoint taken in \cite[\S 42.1]{kmm} and \cite{os}.
 
We mention that the mod 2 degrees of the the cobordism maps 
in these exact triangles is the same as 
the monopole case, and is explained in \cite[\S 42.3]{kmm}. 
There are always non-trivial bundles amongst the 3 cobordism maps, 
even if the three framed groups are untwisted.
For in this case the composite of three 
consecutive cobordism bundles, call it $\xbb_{03}$ as in \S \ref{sec:x},
has $\relp(\xbb_{03})\equiv 1\mod 2$. This is because $\xbb_{03}$ is 
trivial away from a copy of $\cp$ minus a thickened $S^1$; over this area it restricts 
to a non-trivial bundle $\ebb$ which is easily seen to have $\relp(\ebb)\equiv 1$.
Then, by the additivity of $\relp(\xbb)$, at least one of $\xbb_{i,i+1}$ has 
$\relp(\xbb_{i,i+1})\equiv 1$. 
Note that, after computing $\text{deg}(\x_{03})=1$, we see $\text{deg}(\xbb_{03}) \equiv -1\mod 4$.\\

\subsection{Examples}\label{sec:examples} In this section we consider the 
framed instanton homology of $S^3$ and $S^1\times S^2$. 
To compute $\ih^\#(S^3)$ it 
suffices to compute $\ih(\tbb^3)$. This is well-known and elementary, see \cite{bd}. 
Let $N$ be a regular neighborhood of the geometric representative for $\tbb^3$. 
The flat connections modulo even gauge on $\tbb^3$ are in 
correspondence with the set
\[
	\{\rho\in\text{Hom}(\pi_1( T^3\setminus N),\su)\; | \; \rho(\nu)=-1\}/\su,
\]
where $\nu$ is a small meridian around $N$, and the $\su$-action is by conjugation. 
A computation shows that this set consists of two elements; 
these two elements are non-degenerate and irreducible. The two classes as 
generators for $\chain(\tbb^3)$ differ by degee $4$.
It follows that $\chain^\#(S^3)$ has one generator, and we obtain
\[
	\ih^\#(S^3)\simeq\mathbb{Z}_{0},
\]
where, as usual, the subscript indicates the grading.
We usually assume a 
distinguished generator for $\ih^\#(S^3)$ has been fixed.

Next, we compute $\ih^\#(S^1\times S^2)$. For this we adapt \cite[Lemma 8.3]{kmu}. 
By placing the twisting $\mu$ at an $S^3$, we have an exact sequence
\[
	\cdots \ih^\#(S^3) \xrightarrow{\alpha} \ih^\#(S^1\times S^2)
	\xrightarrow{\beta}\ih^\#(S^3)\xrightarrow{\gamma} \ih^\#(S^3)\cdots
\]
We apply the grading formula (\ref{degofmap}). 
The map $\alpha$ comes from the cobordism $D^2\times S^2\setminus \text{int}(D^4)$ 
from $S^3$ to $S^1\times S^2$. The overlying bundle is necessarily trivial. 
We compute $\text{deg}(\alpha) \equiv -1$. 
The map $\beta$ is the same cobordism, but reversed, and $\text{deg}(\beta)\equiv -2$. 
By the previous section, we know the sum of the degrees of the three maps is $-1$ mod 4,
so $\text{deg}(\gamma)\equiv 2$. This can be computed directly by observing that
$\gamma$ comes from the cobordism $\cp$ minus two 4-balls, from 
$S^3$ to $S^3$, with a non-trivial bundle. 
Because $\gamma:\mathbb{Z}_0\to\mathbb{Z}_0$ has degree $2$, 
it must be $0$. By exactness, we conclude
\[	
	\ih^\#(S^1\times S^2)\simeq \mathbb{Z}_{2}\oplus\mathbb{Z}_{3},
\]
where, as usual, the subscripts indicate gradings.
As is evident by the above computation, 
a canonical generator in grading $3$ for $\ih^\#(S^1\times S^2)$ is 
given by $[D^2\times S^2]^\#$. Recall from \S \ref{sec:cobs} that $[D^2\times S^2]^\#$ is the notation for the relative invariant induced by the cobordism $D^2\times S^2:\emptyset\to S^1\times S^2$. A canonical homology orientation is used here. 
The element $[S^1\times D^3]^\#$ generates the summand in grading 2. This 
is seen by identifying $S^1\times S^2$ with its orientation-opposite in 
a standard way, and viewing 
$[D^2\times S^2]_\#$ as a map $\ih^\#(S^1\times S^2)\to \mathbb{Z}$. 
For then we have
\[
	[D^2\times S^2]_\#[S^1\times D^3]^\#  =  \pm 1,
\]
since $S^1\times D^3$ and $D^2\times S^3$ glue along $S^1\times S^2$ to give $S^4$. 
However, the element $[S^1\times D^3]^\#$ is not canonically 
homology oriented; it requires an orientation of 
$H_1(S^1\times S^2;\mathbb{R})$. Thus a generator for 
$\mathbb{Z}_2\subset \ih^\#(S^1\times S^2)$ is distinguished by 
orienting $H_1(S^1\times S^2;\mathbb{R})$.\\

\subsection{The K{\"u}nneth Formula}\label{sec:prod}
Let $\y$ and $\y'$ be closed, oriented and connected 3-manifolds. If either 
one of $\ih^\#(\y)$ or $\ih^\#(\y')$ is torsion-free, there is a graded isomorphism
\[
	\ih^\#(\y\#\y')\simeq \ih^\#(\y)\otimes \ih^\#(\y').
\]
This is a special case of \cite[Cor. 5.9]{kmu}, and follows from Floer's original 
excision theorem. Further, this isomorphism is natural for split cobordisms, 
in the following sense. Let $\x:\y_1\to\y_2$ and $\x':\y_1'\to\y_2'$ be cobordisms 
with paths chosen so that the composite $\x\Join\x'$ is defined. Suppose 
the above product isomorphism holds for $\y_1\#\y_1'$ and $\y_2\#\y_2'$; then we 
have a commutative diagram
\[
\begin{CD}
	\ih^\#(\y_1\#\y_1')@>>{\simeq}> \ih^\#(\y_1)\otimes \ih^\#(\y_1')\\
	@V \ih^\#(\x\Join\x') VV	   @VV \ih^\#(\x)\otimes \ih^\#(\x') V      \\
	\ih^\#(\y_2\#\y_2')  @>{\simeq}>> \ih^\#(\y_2)\otimes \ih^\#(\y_2') \\
\end{CD}
\]
We do not address the arrangement of homology orientations here, as we will not require it.\\

\subsection{A Connected Sum of $S^1\times S^2$'s}\label{sec:example} Let $\y$ 
be a 3-manifold with $\y\simeq \#^k S^1\times S^2$. From the K{\"u}nneth 
formula it is clear that $\ih^\#(\y) \simeq \otimes^k (\mathbb{Z}_2\oplus\mathbb{Z}_3)$. 
The subscripts here indicate gradings. Let $\mu_\y$ be an orientation of $H_1(Y;\mathbb{R})$.
In this section we construct an isomorphism 
\[
	\phi:\ext^\ast( H_1(\y;\mathbb{Z})) \to\ih^\#(\y)
\]
which only depends on $\y$ and $\mu_\y$, not the decomposition $\y\simeq \#^k S^1\times S^2$. The choice of $\mu_\y$ only affects the overall sign of $\phi$. The exterior power here, and for most of the paper, is over the ring $\mathbb{Z}$.

Choose oriented, closed, embedded curves $c_1,\ldots,c_k$ in $\y$ such that 
there exists a diffeomorphism $\y\simeq\#^k_{i=1} S^1\times S^2$ sending 
$c_i$ to $S^1\times\text{pt}$ in the $i^\text{th}$ copy of $S^1\times S^2$. 
Given $J=\{i_1,\ldots,i_l\}\subset\{1,\ldots,k\}$ we define a 
cobordism $\x_J:\emptyset\to\y$ by starting with $\y\times[0,1]$ 
and attaching a 2-handle to each $c_i\times\{0\}$ if $i\in J$, and on top of this, attaching 3-handles and a 4-handle in a way such that $\partial \x_J = \y\times \{1\}$ and
\begin{equation}\label{handleattach}
	\x_J \simeq \x_1 \natural \cdots \natural \x_k, \qquad \x_i = 
	\begin{cases} S^1\times D^3 &\mbox{if $i\not\in J$}\\
	D^2\times S^2 & \mbox{if $i\in J$} \end{cases} 
\end{equation}
with $\partial\x_i$ the $i^\text{th}$ copy of $S^1\times S^2$ in the decomposition 
$\y\simeq\#^k_{i=1} S^1\times S^2$. Let $\{1,\ldots k\}\setminus J = \{i_{l+1},\ldots,i_{k}\}$
be such that $\mu_\y = [c_{i_1}\wedge \cdots \wedge c_{i_k}]$. To homology 
orient $\x_J$ we orient $\mathcal{L}(\x_J)= H_1(\x_J;\mathbb{R})$ 
by $[c_{i_{l+1}}\wedge\cdots\wedge c_{i_k}]$. Define
$\psi:\ext^\ast(c_1,\ldots,c_k)\to\ih^\#(\y)$ by
\[
	\psi(c_{i_1}\wedge\cdots\wedge c_{i_l})=[\x_J]^\#.
\]
This map is an isomorphism by the case $k=1$ and the K{\"u}nneth formula.

With the help of the orientation $\mu_\y$ of $H_1(\y;\mathbb{R})$, we can define a 
bilinear form
\[
	\langle \cdot,\cdot \rangle: I^\#(Y)\otimes I^\#(Y)\to \mathbb{Z},
\]
see also (\ref{eq:dual}). 
The elements $[\x_J]^\#$ as $J$ runs over subsets of $\{1,\ldots,k\}$ 
form a basis for $\ih^\#(\y)$, so it suffices to define the form on these.
Given $J,K\subset \{1,\ldots,k\}$, let $\x_J$ and $\x_K$ be as above with homology orientations 
$\mu_J$ and $\mu_K$, respectively. Then we have elements $[\x_J]^\#,[\x_K]^\#\in\ih^\#(\y)$.
Consider $\overline{\x}_K:\y\to\emptyset$ and homology orient it by $\mu_{\y}\wedge\mu_{K}$. 
This yields $[\overline{\x}_K]_\#:\ih^\#(Y)\to\mathbb{Z}$.
Then the bilinear form $\langle\cdot,\cdot\rangle$ is given by
\[
	\langle [\x_K]^\#,[\x_J]^\#\rangle = [\overline{\x}_K]_\#[\x_J]^\#=[\overline{\x}_K\circ\x_J]_\# = A_{JK}\in\mathbb{Z}.
\]
Now observe that
\begin{equation}\label{splits}
	\x_{K} \circ \x_{J} \simeq \x_1 \# \cdots \# \x_k, \qquad \x_i \simeq 
	\begin{cases} S^1\times S^3 &\mbox{if $i\not\in J\cup K$}\\
	S^2\times S^2 &\mbox{if $i\in J\cap K$}\\
	S^4 & \mbox{otherwise} \end{cases} 
\end{equation}
Note $[S^2\times S^2]^\#=0$, because the degree of the 
cobordism $S^3\to S^3$ given by $S^2\times S^2$ minus two $4$-balls is odd, 
and similarly for $[S^1\times S^3]^\#$. 
Using the naturality with respect to split cobordisms of the K{\"u}nneth formula,
we conclude that $A_{JK}\neq 0$ if and only 
if $J$ and $K$ are complementary, and in this case 
$A_{JK}=\pm 1$.
This sign may be determined by using Definition \ref{def:homcom}, but we will not need it.
It is clear that this bilinear form is non-degenerate. Note that $\langle \cdot,\cdot \rangle$ depends on $c_1,\ldots,c_k$ (which determine an identification of $\y$ with $\overline{\y}$).

We argue that $\psi$ is independent of the 2-handle framings 
chosen to construct the $\x_J$. First construct 
cobordisms $\x_J$ for each subset $J\subset\{1,\ldots,k\}$ as above. 
Choose some $J$, and construct a cobordism $\x'_J$ by attaching 
the 2-handles using possibly different framings as was done for $\x_J$, 
subject to the constraint that $\x'_J$ 
is of the form (\ref{handleattach}). Then
\[
	\overline{\x}_{K} \circ \x'_{J} \simeq \x_1 \# \cdots \# \x_k
\]
just as in (\ref{splits}), except now if $i\in J\cap K$ then 
$\x_i$ is a possibly non-trivial $S^2$-bundle over $S^2$, in which case $[\x_i]^\#=0$.
We homology orient $\x_J'$ in the 
same way as $\x_J$. 
It is easily seen that $[\x_J']^\#$ has 
all the same values $A_{JK}$ as $[\x_J]^\#$ under the bilinear pairing, and thus 
$[\x_J']^\# = [\x_J]^\#$.

Now we see how $\psi$ changes when we change the loops $c_i$. 
Consider replacing the oriented loop $c_1$ by an oriented connected sum 
$c_1\#c_2$. There are many ways of forming this connected sum. 
Let $\x_{c_1\# c_2}$ be the cobordism $\emptyset\to\y$ obtained 
by attaching to $\y\times [0,1]$ a 2-handle along $c_1\# c_2\times\{0\}$ and 3-handles and a 4-handle as above. Supposing $\mu_Y = [c_1\wedge \cdots \wedge c_k]$, we homology orient $\x_{c_1\# c_2}$ by $[c_2\wedge\cdots \wedge c_k]$, just as we homology orient $\x_{\{1\}}$ and $\x_{\{2\}}$. Then
\[
	[\x_{c_1\#c_2}]^\# = [\x_{\{1\}}]^\# + [\x_{\{2\}}]^\#.
\]
Viewing $\x_{c_1\# c_2}:\emptyset\to\y$ and $\overline{\x}_J:\y\to\emptyset$, this follows from computing
\[
	\overline{\x}_{J}\circ\x_{c_1\# c_2} \simeq \x \# \x_3 \# \cdots \# \x_k,  
	\qquad \begin{cases} \x\simeq S^4 &\mbox{if $|\{1,2\}\cap J| =1$}\\
	\text{deg}(\x)\text{ is odd} & \mbox{otherwise} \end{cases} 
\]
where each $\x_i\simeq S^4$ if $i\in J$ and $\text{deg}(\x_i)$ is odd otherwise, and then appealing to the non-degeneracy of our bilinear form. 
A similar argument shows $[\x_{c_1\#c_2\cup J}]^\# = 
[\x_{\{1\}\cup J}]^\# + [\x_{\{2\}\cup J}]^\#$
where $J$ is any subset of $\{3,\ldots,k\}$. 

As a consequence, $\psi$ induces a well-defined isomorphism
\[
	\phi:\ext^\ast (H_1(\y;\mathbb{Z})) \to \ih^\#(\y).
\]
This is because any two sets of loops $c_1,\ldots,c_k$ in $\y$ as above (having the 
property that there exists a diffeomorphism $\y\simeq \#^k S^1\times S^2$ sending 
each $c_i$ to a factor $S^1\times \text{pt}$) are related by sequences of connected sums (and the reverse operation) as in the previous paragraph. Indeed, these are just handle-slides, and a result of Laudenbach and 
Po\'{e}naru \cite{lp}, as cited in \cite[Rmk. 4.4.1]{gs}, says that any self-diffeomorphism 
of $\#^k S^1\times S^2$ extends to a diffeomorphism of $\natural^k S^1\times D^3$, a 
bounding 1-handlebody, which can be written as a composite of 1-handle slides. In fact, this result also says that the way in which the 3-handles and 4-handle are attached to construct $\x_J$ above is essentially unique.

In summary, $\phi$ is defined by choosing an orientation $\mu_Y$ of $H_1(Y;\mathbb{R})$, a diffeomorphism 
$\y\simeq \#^k S^1\times S^2$, oriented loops $c_1,\ldots,c_k$ corresponding to the 
$S^1\times\text{pt}$ factors, and setting
\[
	\phi([c_{i_1}]\wedge\cdots\wedge[c_{i_l}])=[\x_J]^\#
\]
where the element $[\x_J]^\#$ is 
defined as above. The content of the above discussion is that this map is 
well-defined and is an isomorphism. We mention that for 
$x\in\ext^i (H_1(\y;\mathbb{Z}))$ with $b_1(\y)=k$, the grading 
of $\phi(x)$ in $\ih^\#(\y)$ is given by $2k + i$ mod $4$.\\

\subsection{A Spectral Sequence}

The spectral sequence of Theorem \ref{ss1} leads to one for the groups 
$\ih^\#(\y)$. The setup is as follows. Again we have an $m$-component framed link 
$\lt$ in $\y$. We view $\lt$ as a link in $\y\# T^3$, and we choose a 
family of bundles over the surgered manifolds $\y_v\# T^3$ which for 
$v\in \{0,1\}^m$ are of the form $(\y_v\times\so)\#\tbb^3$, at the expense of 
having possibly non-trivial bundles (so twisted framed groups) 
for the indices $v\in\{0,1,\infty\}^m\setminus\{0,1\}^m$. We are using 
the third row of Figure \ref{fig:ses} to achieve this setup. 
This forces the bundle over $\y\# T^3$ to be geometrically represented 
by the link $L$ together with an $S^1$-factor of $T^3$.
More general spectral sequences may be obtained by allowing twisting 
in the $E^1$-page.

Before stating the resulting theorem, 
we discuss how to lift the previous $\mathbb{Z}/2$-grading $\text{gr}[\chainbf]$ 
for the $E^1$-page of the link surgeries spectral sequence to a $\mathbb{Z}/4$-grading, 
in the special case
where $[L]=0\in H_1(\y;\mathbb{F}_2)$. 
Write $\text{gr}[\y]$ for the $\mathbb{Z}/4$-grading on $\ih^\#(\y)$ 
and $\ybb_v^\#=\ybb_v\#\tbb^3$, where for $v\in\{0,1\}^m\cup\{\boldsymbol{\infty}\}$ 
we have $\ybb_v=\y_v\times\so$. Recall that we 
conflate $\boldsymbol{\infty}$ and $-\mathbf{1}$.
Also write $\xbb_{vw}^\#=\xbb_{vw}\Join(\tbb^3\times [0,1])$ for 
the surgery cobordism bundles.
For $v\in\{0,1\}^m\cup\{\boldsymbol{\infty}\}$ 
we may view each $\chain(\ybb^\#)$ as two copies of $\chain^\#(\y_v)$, 
$\mathbb{Z}/4$-graded by $\text{gr}[\y_v]$. For $v\in\{0,1\}^m$ and 
$x\in\chain(\ybb_v^\#)\subset\chainbf$
of homogeneous $\text{gr}[\y_v]$ grading, we define
\begin{equation}
	\text{gr}[\chainbf](x) = \text{gr}[\y_v](x)  - \text{deg}(\xbb_{\boldsymbol{\infty}v}) - |v|_1 \mod 4.\label{eq:grt}
\end{equation}
The verification that $\partialbf$ lowers this grading by 1, 
and that the quasi-isomorphism $\mathbf{Q}:\chain(\ybb^\#)\to \chainbf$ 
preserves the relevant $\mathbb{Z}/4$-gradings, is the same as in \S \ref{sec:linksgr}.

\begin{theorem}\label{thm:framed}
Let $\lt$ be an oriented framed link with $m$ components in $\y$ and for 
each $v\in \{\infty,0,1\}^m$ denote by $\y_v$ the result of $v$-surgery 
on $\lt$ in $\y$. There are surgery cobordisms $\x_{vw}$ for $v<w$ 
from $\y_v$ to $\y_w$ with homology orientations $\mu_{vw}$ satisfying 
$\mu_{uw}\circ\mu_{vu}=\mu_{vw}$ whenever $v<u<w$, and an appropriate 
bundle $\xbb_{vw}$ over each $\x_{vw}$, such that there is a spectral sequence $(E^r,d^r)$ 
with
\[
	E^1 = \bigoplus_{v\in\{0,1\}^m} \ih^\#(\y_v), \qquad d^1 = 
	\sum_{\substack{v < w\\ |w-v|_1 =1}} (-1)^{\delta(v,w)}\ih^\#(\xbb_{vw})
\]
where $\delta(v,w)$ is as in Theorem \ref{ss1}. 
The spectral sequence is graded 
by $\mathbb{Z}/2\times\mathbb{Z}$, where $d^r$ has bi-degree $(1,r)$,
and it converges 
by the $E^{m+1}$-page to the possibly twisted group 
\[
	\ih^\#(\y;L).
\]
The $\mathbb{Z}/2$-grading induced by the spectral 
sequence agrees with the 
$\mathbb{Z}/2$-grading of $\ih^\#(\y;L)$. If $[L]=0\in H_1(\y;\mathbb{F}_2)$, 
then we can lift the $\mathbb{Z}/2$-grading 
of the $E^1$-page to a $\mathbb{Z}/4$-grading by (\ref{eq:grt}), such that the induced 
$\mathbb{Z}/4$-grading agrees with the one on $\ih^\#(\y)$. The differential 
for the $\mathbb{Z}/4\times\mathbb{Z}$-grading has bi-degree $(-1,r)$.
\end{theorem}

%% file: oddkhov.tex
\section{Branched Double Covers}\label{sec:oddkh}

In this section we complete the proof of Theorem \ref{thm:1}. 
First, we define reduced odd Khovanov homology in \S \ref{sec:oddkh},
following Bloom's description from \cite{bloom2}. We give an alternative 
description of the differential that will suit our goals.
We then discuss how to compose homology orientations in \S \ref{sec:homor}. 
This is the framework we use to understand the signs 
in our spectral sequence.
In \S \ref{sec:branched}, we identify the $E^1$-page of 
the spectral sequence (\ref{sps}) with the reduced odd Khovanov chain complex. 
Finally, in \S \ref{sec:finalgr}, we discuss the $\mathbb{Z}/4$-grading of the spectral sequence 
and the conclusion of Corollary \ref{cor:3}.\\

\begin{figure}[t]
\includegraphics[scale=.8]{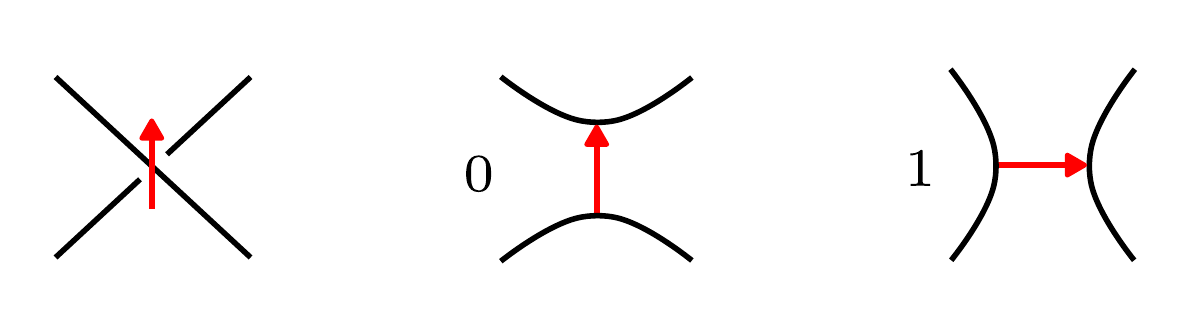}
\caption{Resolution conventions for the arc-decorated
diagrams in reduced odd Khovanov homology. There two choices for the placement of an arc at a given crossing; in the left-most picture, the arc can be pointing up (as depicted) or down. In the latter case, the arcs in the resolution pictures are correspondingly reversed.}
\label{fig:oddkh}
\end{figure}

\subsection{Odd Khovanov Homology}{\label{sec:oddkh}}
Let $\lt$ be an oriented link and $\dt$ a planar diagram for $\lt$. Suppose $\dt$ 
has $m$ crossings. We assume that each crossing has an arrow drawn over it, 
as in Figure \ref{fig:oddkh}.
Then for each $v\in\{0,1\}^m$ we can define a resolution diagram $\dt_v$ 
according to the rules of Figure \ref{fig:oddkh}. 
Each $\dt_v$ is a disjoint union of planar-embedded unoriented circles 
together with a disjoint union of planar-embedded oriented arcs, 
each arc beginning and ending at a circle.
Suppose $\dt_v$ has $k+1$ circles. Then we have a rank $k$ 
abelian group $V_v$ defined by 
\[
	V_v=\mathbb{Z}\{\text{arcs}\}/\text{ker}\left(\mathbb{Z}\{\text{arcs}\}\to\mathbb{Z}\{\text{circles}\}\right)
\]
where the map involved sends an arc to the circle at which it begins minus 
the circle at which it ends. A basis for $\text{V}_v$ is given by any $k$ 
arcs that touch all $k+1$ circles in $\dt_v$. Otherwise said, a basis is given by the edges of any spanning tree of the graph whose vertices are the circles of $\dt_v$ and edges are the arcs. We define
\[
	\chain_v = \ext^\ast (V_v), \qquad \chain = \bigoplus_{v\in\{0,1\}^m}\chain_v.
\]
For each $v,w\in\{0,1\}^m$ with $v<w$ and $|w-v|_1=1$ 
we introduce a map $\partial_{vw}':\chain_v\to \chain_w$. There is a single 
arc $x_{vw}$ in each of $\dt_v$ and $\dt_w$ that 
changes from a $0$-resolution position to a $1$-resolution position. There 
are two cases to consider, corresponding to two circles merging or splitting:
\[
	\partial'_{vw}(x) := \begin{cases} x_{vw} \wedge x &\mbox{if $0=x_{vw}\in \chain_v$ (split)}\\
	x & \mbox{if $0\neq x_{vw}\in \chain_v$ (merge)} \end{cases} 
\]
In these expressions we use the symbol $x_{vw}$ to stand both for 
an arc and its equivalence class in $V_v$. 
We call the collection of $\partial'_{vw}$ 
the \textit{pre-differential}.
The differential for $\chain$ is defined by
\[
	\partial = \sum \partial_{vw} = \sum \varepsilon_{vw}\partial'_{vw}
\]
where each $\varepsilon_{vw}$ is $+1$ or $-1$, and the sums are over $v,w$ 
with $v<w$ and $|w-v|_1=1$. The signs $\varepsilon_{vw}$ are chosen to satisfy two conditions. 
The first condition is that $\partial^2=0$. The second condition is as follows. 
Let $v<t,u$ with $|t-v|_1=|u-v|_1=1$ be three vertices where the 
arcs $x_{vu}$ and $x_{vt}$ are arranged 
in $\dt_v$ as in the left of Figure \ref{fig:typexy}. Let $w$ be the vertex with $w >t,u$ and 
$|w-t|_1=|w-u|_1=1$. Any four such vertices $v,u,t,w$ 
will be called a {\em type X} face. A {\em type Y} face 
is obtained by reversing one of either $x_{vu}$ or $x_{vt}$. The second condition
is that for a type X face, the sign
\[
	\varepsilon_{vu}\varepsilon_{vt}\varepsilon_{tw}\varepsilon_{uw}
\]
is always $+1$ or always $-1$; and the same product for a type Y face is also
always $+1$ or always $-1$, and is minus the type X sign. 
We call the collection of $\varepsilon_{vw}$ a {\em valid edge assignment} 
if it satisfies these two conditions. The reduced odd Khovanov homology 
of $\lt$ is then defined to be $\kh(\lt) = H_\ast(\chain, \partial)$. 
The well-defined-ness and invariance is proved in \cite{ors}.

The group $\kh(\lt)$ is bigraded by a homological grading $\hgr$ and 
and quantum grading $\qgr$. For an element $x\in\ext^{|x|}(V_v)$ 
where $k=\dim(V_v)$, these are defined by
\[
	\hgr(x) = |v|_1 - n_-,
\]
\[
	\qgr(x) = k -2|x|+n_+ -2n_- + |v|_1.
\]
Here $n_\pm$ is the number of $\pm$ crossings in $D$.
We are interested in a $\mathbb{Z}/4$-grading defined by
\begin{equation}
	 \frac{3}{2}\qgr -\hgr +\frac{1}{2}\left(\sigma+\nu\right) \mod 4\label{eq:oddgr}
\end{equation}
where $\sigma$ is the signature of $\lt$ and $\nu$ the nullity.

\begin{figure}[t]
\includegraphics[scale=.9]{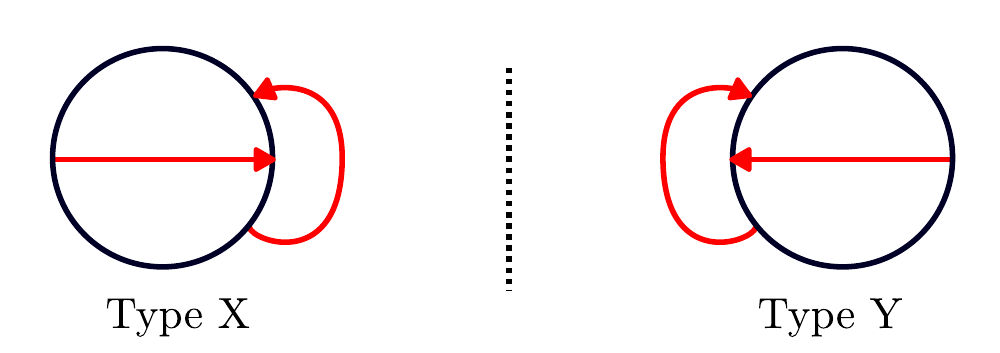}
\caption{The configurations of the two relevant arcs in 
the initial vertex of a type X and type Y face.}
\label{fig:typexy}
\end{figure}

%
%

\vspace{10px}

\subsection{Composing Homology Orientations}\label{sec:homor}

Suppose we are given $\x_1:\y_1\to\y_2$ and $\x_2:\y_2\to\y_3$. Let $\x_{12} = \x_2\circ\x_1$.
In this section we describe the rule we use to orient $\mathcal{L}(\x_{12})$ given 
orientations of $\mathcal{L}(\x_1)$ and $\mathcal{L}(\x_2)$. We recall
\[
	\mathcal{L}(\x_1) =H_1(\y_1;\mathbb{R})\oplus H_1(\x_1;\mathbb{R})\oplus H_2^+(\x_1;\mathbb{R}).
\]
Typically, 
an orientation of $\mathcal{L}(\x_i)$ will be denoted $\mu_i$. Although
the composition of homology orientations originates from the determinants of 
the relevant Fredholm operators, in our applications we prefer to have a concrete, algebro-topological description of such a rule.
Perhaps the two most important formal properties of a composition rule compatible with a construction of framed instanton homology are associativity and the existence of units. In other words,
\[
	(\mu_3\circ\mu_2)\circ \mu_1 = \mu_3\circ(\mu_2\circ\mu_1)
\]
whenever $\mu_i$ is a homology orientation of $\x_i:\y_i\to\y_{i+1}$ for $i=1,2,3$, 
and for $Y\times [0,1]$ there exists a distinguished homology orientation $\mu_{\y}^{\text{id}}$ 
such that 
\[
	\mu_{\y}^{\text{id}}\circ \mu = \mu, \qquad \mu\circ\mu_{\y}^{\text{id}}=\mu
\]
whenever $\mu$ is a homology orientation and these compositions make sense. We will first define a composition rule in an algebro-topological fashion and then show it has these two properties. At the end of this section, we will describe how the rule we have defined can be described using Fredholm determinant line bundles, using the setup of Kronheimer and Mrowka \cite[\S 20.2]{kmm}, ensuring that our rule is compatible with a construction of framed instanton homology. In this section all homology groups are assumed to have real coefficients.

We proceed to construct the composition rule.
As in the previous sections, we maintain the assumption
that the $\x_i$ and $\y_i$ are connected.
For background on the following setup, see \cite[Thm. 27.5]{dfn} and \cite[\S 7]{as}.
Let $f_{12}:H_1(\y_2)\to H_1(\x_1)\oplus H_1(\x_2)$ 
be the map in the Mayer-Vietoris sequence.
Consider the following exact sequences:
\begin{equation}
0 \to\text{im}(f_{12})\to H_1(\x_1)\oplus H_1(\x_2)\to H_1(\x_{12})\to 0\label{eq:exseq1}
\end{equation}
\begin{equation}
0 \to \text{ker}(f_{12}) \to H_1(\y_2)\to \text{im}(f_{12}) \to 0\label{eq:exseq2}
\end{equation}
\begin{equation}
0 \to H_2^+(\x_1)\oplus H_2^+(\x_2) \to H_2^+(\x_{12}) \to \text{ker}(f_{12})\to 0\label{eq:exseq3}
\end{equation}
The first exact sequence is extracted from the Mayer-Vietoris sequence, and the 
second is naturally associated to the map $f_{12}$. Our convention is 
that $f_{12}(x)=(x,-x)$ on the chain level. For the third sequence, we
choose the positive definite subspace $H_2^+(\x_{12})$ so that 
it contains the image of $H_2^+(\x_1)\oplus H_2^+(\x_2)$ under 
the map $H_2(\x_1)\oplus H_2(\x_2)\to H_2(\x_{12})$. 
The map $H_2^+(\x_{12}) \to \text{ker}(f_{12})$ is a restriction of the 
Mayer-Vietoris boundary map $H_2(\x_{12})\to H_1(\y_2)$.

There is a concrete interpretation of (\ref{eq:exseq3}). 
Upon splitting the sequence it says we can write 
\[
	H_2^+(\x_{12})=H_2^+(\x_1)\oplus H_2^+(\x_2)\oplus \text{ker}(f_{12}).
\]
To interpret the summand $\text{ker}(f_{12})$, we write down a section $s$ for the map 
$H_2^+(\x_{12}) \to \text{ker}(f_{12})$. We define
$s:\text{ker}(f_{12})\to H_2^+(\x_{12})$ on a basis of 1-cycle classes $[\gamma]$ 
in $\text{ker}(f_{12})\subset H_1(\y_2)$ as follows. 
For each such 1-cycle $\gamma$ 
in $\y_2$, choose a 2-cycle $\Sigma$ in $\y_2$ such that $\#(\gamma\cap\Sigma)=1$, and extend 
$\gamma$ to a 2-cycle $\Gamma$ in $\x_{12}$. Then $s[\gamma] = [\Gamma]+[\Sigma]$.

Choosing splittings of the above three exact sequences, summing, cancelling a copy of $\text{ker}(f_{12})$ on both sides, and then moving summands around
yields an identification
\begin{equation}
	\mathcal{L}(\x_{12}) \oplus \text{im}(f_{12})^{\oplus 2} =
	\mathcal{L}(\x_1) \oplus \mathcal{L}(\x_2). \label{eq:iso}
\end{equation}
Thus we can orient $\mathcal{L}(\x_{12})$ by using given orientations of $\mathcal{L}(\x_{1})$
and $\mathcal{L}(\x_{2})$ and equipping the two copies of $\text{im}(f_{12})$ with the 
same orientation. We will give an explicit rule for doing this, designed so as to 
be associative. We choose splittings of the above exact sequences, 
in their respective order:
\begin{equation}
F_{12}:\text{im}(f_{12})\oplus H_1(\x_{12})\xrightarrow{\sim}H_1(\x_1)\oplus H_1(\x_2),\label{eq:split1}
\end{equation}
\begin{equation}
G_{12}:\text{ker}(f_{12})\oplus\text{im}(f_{12})\xrightarrow{\sim} H_1(\y_2),\label{eq:split2}
\end{equation}
\begin{equation}
H_{12}:H_2^+(\x_1)\oplus H_2^+(\x_2)\oplus\text{ker}(f_{12})\xrightarrow{\sim} H_2^+(\x_{12}). \label{eq:split3}
\end{equation}
The space of such splittings is contractible, so these particular choices
do not matter for the following definition.

\begin{definition}\label{def:homcom}
For $i=1,2$ let $\x_i:\y_i\to \y_{i+1}$ 
be two connected cobordisms between connected, non-empty
3-manifolds. Write $\x_{12}=\x_2\circ\x_1$.
Suppose $\mu_i$ is a homology orientation of $\x_i$, i.e. an orientation of  $\mathcal{L}(\x_i)$, 
for $i=1,2$. Write $\mu_i = \beta_i\wedge\alpha_i\wedge\gamma_i$ where $\alpha_i$ is 
an orientation for $H_1(\y_i)$, $\beta_i$ for $H_1(\x_i)$, and $\gamma_i$ for $H_2^+(\x_i)$. 
Choose any orientation $\delta_{12}$ of $\text{im}(f_{12})$. 
Choose splittings of the exact sequences (\ref{eq:exseq1})-(\ref{eq:exseq3}) written 
as in (\ref{eq:split1})-(\ref{eq:split3}).
Equip $H_1(\x_{12})$ with an orientation 
$\beta_{12}$ given by the condition 
\[
	F_{12}(\delta_{12}\wedge\beta_{12}) = \beta_1\wedge\beta_2.
\]
Similarly, equip $\text{ker}(f_{12})$ with an orientation $\zeta_{12}$ which satisfies
\[
	G_{12}(\zeta_{12} \wedge \delta_{12}) = \alpha_2.
\]
Then define the composition of $\mu_1$ with $\mu_2$, which 
is an orientation of $\mathcal{L}(\x_{12})$, by
\[
	\mu_{2}\circ\mu_{1} = (-1)^{s}\beta_{12}\wedge\alpha_1\wedge H_{12}(\gamma_1\wedge\gamma_2\wedge\zeta_{12}),
\]
\[
	s = \frac{1}{2}\left(d_{12}^2-d_{12}\right) + b_1(\x_1)b_1(\y_2) + b_1(\x_1)b_2^+(\x_2)+ b_1(\y_2)b_2^+(\x_2).
\]
Here $d_{12}=\dim\left[\text{im}(f_{12})\right]$.
\end{definition}

\begin{prop}
	This composition rule for homology orientations is associative.
\end{prop}

\begin{proof}
	We first rephrase the problem in terms of linear algebra. 
	For $i=1,2$ consider quadruples $\mathscr{A}_i=(A_i,B_i,C_i,\mu_i)$ where $A_i,B_i,C_i$ 
	are vector spaces and $\mu_i$ is an orientation of $A_i\oplus B_i\oplus C_i$.
	In our application we have $A_i=H_1(\y_i)$, $B_i=H_1(\x_i)$, and $C_i=H_2^+(\x_i)$. 
	Given a linear map
	\[
		f_{12}:A_2\to B_1\oplus B_2,
	\]
	we can compose $\mathscr{A}_1$ and $\mathscr{A}_2$ along $f_{12}$ to form
	\[
		\mathscr{A}_{2}\circ_{f_{12}}\mathscr{A}_1 =(A_1,\text{coker}(f_{12}),C_1\oplus C_2\oplus \text{ker}(f_{12}),\mu_{2}\circ\mu_1).
	\]
	The orientation $\mu_{12}=\mu_{2}\circ\mu_{1}$ is adapted from Definition \ref{def:homcom} as follows. Write 
	$\mu_i = \beta_i\wedge \alpha_i\wedge\gamma_i$ where $\alpha_i,\beta_i,\gamma_i$ 
	are respective orientations of $A_i,B_i,C_i$. Choose an orientation $\delta_{12}$ 
	of $\text{im}(f_{12})$. Choose isomorphisms
	\begin{equation}
		F_{12}:\text{im}(f_{12})\oplus \text{coker}(f_{12})\xrightarrow{\sim} B_1\oplus B_2,\label{eq:iso1}
	\end{equation}
	\begin{equation}
		G_{12}:\text{ker}(f_{12})\oplus \text{im}(f_{12})\xrightarrow{\sim} A_2\label{eq:iso2}
	\end{equation}
	that are splittings of the naturally associated exact sequences.
	Orient $\text{coker}(f_{12})$ by $\beta_{12}$ and 
	$\text{ker}(f_{12})$ by $\zeta_{12}$ using 
	the conditions
	\begin{equation*}
		F_{12}(\delta_{12}\wedge\beta_{12})=\beta_1\wedge\beta_2, \quad G_{12}(\zeta_{12}\wedge\delta_{12})=\alpha_2. 
	\end{equation*}
	Then the composition $\mu_{12}$ is given by
	\[
	\mu_{12} = (-1)^{s_{12}}\beta_{12}\wedge\alpha_1\wedge\gamma_1\wedge\gamma_2\wedge\zeta_{12},
	\]
	\[
		s_{12}= b_1a_2+b_1c_2+a_2c_2 + (d_{12}^2-d_{12})/2,
	\]
	where $a_i=\dim A_i$, $b_i=\dim B_i$, $c_i =\dim C_i$, and $d_{12} = \dim \left[\text{im}(f_{12})\right]$. 
	Now suppose we have a third quadruple $\mathscr{A}_3=(A_3,B_3,C_3,\mu_3)$ and a linear map $f_{23}:A_3\to B_2\oplus B_3$. Consider
	\[
		f=f_{12}+f_{23}:A_2\oplus A_3\to B_1\oplus B_2\oplus B_3.
	\]
	The map $f$ induces further maps
	\[
		f_{1,23}:A_2\to B_1\oplus\text{coker}(f_{23}), \quad f_{12,3}:A_3\to \text{coker}(f_{12})\oplus B_3.
	\]
	We write $F_{23}, G_{23}$ for the isomorphisms associated to $f_{23}$ as in (\ref{eq:iso1}), (\ref{eq:iso2}); 
	$F_{12,3},G_{12,3}$ associated to $f_{12,3}$; and $F_{12,3}, G_{12,3}$ to $f_{1,23}$.
	We have identifications
	\begin{equation}
		\text{coker}(f_{1,23})=\text{coker}(f) = \text{coker}(f_{12,3}),\label{eq:cokerid}
	\end{equation}
	\begin{equation}
		\text{ker}(f_{23})\oplus \text{ker}(f_{1,23}) = \text{ker}(f) =\text{ker}(f_{12})\oplus\text{ker}(f_{12,3}).\label{eq:kerid}
	\end{equation}
The cokernel identifications are natural. The kernel identifications depend on some choices. 
For instance, $\text{ker}(f_{12})\oplus \text{ker}(f_{12,3}) = \text{ker}(f)$ is established as follows. 
Clearly $\text{ker}(f_{12})\subset \text{ker}(f)$. 
Now suppose $a\in\text{ker}(f_{12,3})\subset A_3$. Then $\pi_{12}(f(a))\in\text{im}(f_{12})$ where 
$\pi_{12}$ projects onto $B_1\oplus B_2$. Thus $\pi_{12}(f(a))=f_{12}(b)$ for some $b\in A_2$. Let $\sigma_{12}:\text{im}(f_{12})\to A_2$ be such that $f_{12}\sigma_{12}=\text{id}_{\text{im}(f_{12})}$.
Then we may take $b=\sigma_{12}(\pi_{12}(f(a)))$, and 
the assignment $a\mapsto(-b,a)$ injects $\text{ker}(f_{12,3})$ into $\text{ker}(f)$. 
In this way we obtain a map from $\text{ker}(f_{12})\oplus\text{ker}(f_{12,3})$ to 
$\text{ker}(f)$ which is easily seen to be an isomorphism.
 With these identifications, the associativity of our rule in Definition \ref{def:homcom} is nearly equivalent to
	\begin{equation}
		\mathscr{A}_3\circ_{f_{12,3}}(\mathscr{A}_2\circ_{f_{12}}\mathscr{A}_1) = (\mathscr{A}_3\circ_{f_{23}}\mathscr{A}_2)\circ_{f_{1,23}}\mathscr{A}_1.\label{eq:linalgass}
	\end{equation}
We have only left out the roles of the $H_{12}$ maps; these are not essential and 
we remark on their absence at the end of the proof. We proceed to establish (\ref{eq:linalgass}).
	Let us write out $\mu_{12,3} = \mu_3\circ\mu_{12}$, the orientation associated to the left side of (\ref{eq:linalgass}). 
	Let $\mu_3=\beta_3\wedge\alpha_3\wedge\gamma_3$ where $\alpha_3,\beta_3,\gamma_3$ are orientations 
	of $A_3,B_3,C_3$, respectively. Let $\delta_{12,3}$ orient $\text{im}(f_{12,3})$. 
	Orient $\text{coker}(f_{12,3})$ by $\beta_{12,3}$ and $\text{ker}(f_{12,3})$ by 
	by $\zeta_{12,3}$, where
	\begin{equation*}
		F_{12,3}(\delta_{12,3}\wedge\beta_{12,3}) = \beta_{12}\wedge \beta_3, \quad G_{12,3}(\zeta_{12,3}\wedge\delta_{12,3}) = \alpha_3.
	\end{equation*}
	Then we use our composition rule to obtain
	\[
		\mu_{12,3} = (-1)^{s_{12,3}}\beta_{12,3}\wedge\alpha_1\wedge(\gamma_1\wedge\gamma_2\wedge\zeta_{12})\wedge\gamma_3\wedge\zeta_{12,3},
	\]
	\[
		s_{12,3} = s_{12} + b_{12}a_3 +b_{12}c_3+ a_3c_3 + (d_{12,3}^2-d_{12,3})/2.
	\]
	Here $d_{12,3}=\dim\left[\dim(f_{12,3})\right]$ and $b_{12}=\dim\left[\text{coker}(f_{12})\right]$,
	so in particular
	\[
		b_{12} = b_1+b_2 -d_{12}.
	\]
	Now we write out the orientation associated to the right side 
	of (\ref{eq:linalgass}). We first write
	\[
		\mu_{23} = \mu_3\circ\mu_2 = (-1)^{s_{23}}\beta_{23}\wedge\alpha_2\wedge\gamma_2\wedge\gamma_3\wedge\zeta_{23},
	\]
	\[
		s_{23} = b_2a_3 + b_2c_3+ a_3c_3 +(d_{23}^2-d_{23})/2,
	\]
	where, given an orientation $\delta_{23}$ of $\text{im}(f_{23})$, we have imposed
	\begin{equation*}
		F_{23}(\delta_{23}\wedge \beta_{23}) = \beta_{2}\wedge\beta_3, \quad G_{23}(\zeta_{23}\wedge\delta_{23}) = \alpha_3.
	\end{equation*}
	Now we can also write
	\[
	\mu_{1,23} = (-1)^{s_{1,23}}\beta_{1,23}\wedge\alpha_1\wedge\gamma_1\wedge(\gamma_2\wedge\gamma_3\wedge\zeta_{23})\wedge\zeta_{1,23},
	\]
	\[
		s_{1,23} = s_{23} + b_1a_2+b_1c_{23}+a_2c_{23} + (d_{1,23}^2-d_{1,23})/2,
	\]
	where $c_{23}=\dim\left[C_2\oplus C_3 \oplus \text{ker}(f_{23})\right]$, so that
	\[
		c_{23} = c_2+c_3+a_3-d_{23},
	\]
	and, given an orientation $\delta_{1,23}$ of $\text{im}(f_{1,23})$, we have the conditions
	\begin{equation*}
		F_{1,23}(\delta_{1,23}\wedge\beta_{1,23}) = \beta_1\wedge\beta_{23}, \quad G_{1,23}(\zeta_{1,23}\wedge\delta_{1,23})=\alpha_2.
	\end{equation*}
	
	We will now show that $\mu_{12,3} = \mu_{1,23}$. We choose identifications
	\[
		\text{im}(f_{12})\oplus\text{im}(f_{12,3})=\text{im}(f)=\text{im}(f_{23})\oplus\text{im}(f_{1,23}).
	\]
	These depend on $F_{12}$ and $F_{23}$. For instance, 	
	let $\tau_{12}:\text{coker}(f_{12})\to B_1\oplus B_2$ be the map
	extracted from $F_{12}$ (and conversely it may define $F_{12}$). 
	Then $\text{im}(f_{12,3})$ maps into $\text{im}(f)$ by 
	$a\mapsto (\tau_{12}(\pi(a)),\pi_3(a))$ where $\pi$ projects onto 
	$\text{coker}(f_{12})$ and $\pi_3$ onto $B_3$. Since $\text{im}(f_{12})$ is 
	naturally a subset of $\text{im}(f)$, we then obtain a map from $\text{im}(f_{12})\oplus\text{im}(f_{12,3})$ into $\text{im}(f)$ which yields the above identification.
	We can thus orient $\text{im}(f)$ by $\delta_{12}\wedge\delta_{12,3}$ or 
	by $\delta_{23}\wedge\delta_{1,23}$. It suffices to show
	\begin{equation}
		\delta_{12}\wedge \delta_{12,3}\wedge\mu_{12,3}\wedge \delta_{12}\wedge \delta_{12,3} =  \delta_{23}\wedge \delta_{1,23}\wedge\mu_{1,23}\wedge \delta_{23}\wedge \delta_{1,23}\label{eq:suffices}
	\end{equation}
	as orientations of $\text{im}(f)\oplus V\oplus\text{im}(f)$, 
	where $V$ is the total space of either side of (\ref{eq:linalgass}) 
	for which $\mu_{1,23}$ and $\mu_{12,3}$ are orientations. We compute the left side of (\ref{eq:suffices}):
	\begin{align*}
		(-1)&^{s_{12,3}+d_{12}d_{12,3}}\delta_{12}\wedge \delta_{12,3} \wedge \beta_{12,3}\wedge\alpha_1\wedge\gamma_1\wedge\gamma_2\wedge\zeta_{12}\wedge\gamma_3\wedge\zeta_{12,3}\wedge\delta_{12,3}\wedge \delta_{12} \\
					&= (-1)^{s_{12,3}+d_{12}d_{12,3}+d_{12}(a_{3}+c_3)+a_2c_3} \left[( \text{id}_{\text{im}(f_{12})}\oplus F_{12,3}^{-1})(F_{12}^{-1}\oplus\text{id}_{B_3})\right](\beta_1\wedge\beta_2\wedge\beta_3)\\
					& \hfill\hfill\qquad \wedge\alpha_1\wedge\gamma_1\wedge\gamma_2\wedge\gamma_3\wedge\left[G_{12}^{-1}\oplus G_{12,3}^{-1}\right](\alpha_2\wedge\alpha_3).
	\end{align*}
	Now, choose splitting isomorphisms
	\[
		F:\text{im}(f)\oplus\text{coker}(f)\xrightarrow{\sim} B_1\oplus B_2\oplus B_3,
	\]
	\[
		G:\text{ker}(f)\oplus \text{im}(f)\xrightarrow{\sim} A_2\oplus A_3
	\]
	for the naturally associated short exact sequences. We claim we have
	\begin{equation}
		\left[(\text{id}_{\text{im}(f_{12})}\oplus F_{12,3}^{-1})(F_{12}^{-1}\oplus\text{id}_{B_3})\right](\beta_1\wedge\beta_2\wedge\beta_3) = F^{-1}(\beta_1\wedge\beta_2\wedge\beta_3),\label{eq:F}
	\end{equation}
	\begin{equation}
		\left[G_{12}^{-1}\oplus G_{12,3}^{-1}\right](\alpha_2\wedge\alpha_3) = G^{-1}(\alpha_2\wedge\alpha_3).\label{eq:G}
	\end{equation}
	We consider (\ref{eq:F}). To abstract the underlying problem, consider 
	a linear map $\phi:V\to W$ and distinguished subspaces $V'\subset V$ and $W'\subset W$ such 
that $\phi(V')\subset W'$. In other words, we have a relative linear map $\phi:(V,V')\to (W,W')$. Choose an isomorphism
\[
	\Phi:\text{im}(\phi)\oplus\text{coker}(\phi)\xrightarrow{\sim} W
\]
associated to the natural short exact sequence. Similarly, choose
\[
	\Phi':\text{im}(\phi')\oplus\text{coker}(\phi')\xrightarrow{\sim} W'
\]
where $\phi':V'\to W'$ is a restriction of $\phi$. Also, with $\phi'':V/V'\to \text{coker}(\phi')\oplus W/W'$ choose
\[
	\Phi'':\text{im}(\phi'')\oplus\text{coker}(\phi'')\xrightarrow{\sim}\text{coker}(\phi')\oplus W/W'.
\]
We can identify $\text{coker}(\phi'')=\text{coker}(\phi)$ and $\text{im}(\phi)=\text{im}(\phi')\oplus\text{im}(\phi'')$ just as we have done in our setting above. We also choose an identification $W/W'\oplus W' = W$. Then (\ref{eq:F}) is equivalent to
\[
	\text{det}\Big[\Phi^{-1}(\Phi'\oplus\text{id}_{W/W'} )(\Phi''\oplus\text{id}_{\text{im}(\phi')})\Big] > 0,
\]
by setting $\phi=f$, $V= A_2\oplus A_3$, $V'=A_2$, $W=B_1\oplus B_2\oplus B_3$, and $W'=B_1\oplus B_2$. In fact, we can choose the data so that, under these identifications,
\begin{equation}
	\Phi = (\Phi'\oplus\text{id}_{W/W'} )(\Phi''\oplus\text{id}_{\text{im}(\phi')}).\label{eq:thedecomp}
\end{equation}
This can be seen as follows. We may equip $V$ and $W$ with inner products so that we may freely take complements. 
In the following, we use the notation $V_1^{\perp}\subset V_2$ to mean that the complement $V_1^\perp$ (with $V_2$ possibly inside a larger space) was taken inside $V_2$. 
We may then identify $\text{coker}(\phi) = \text{im}(\phi)^{\perp} \subset W$, 
$\text{coker}(\phi')=\text{im}(\phi')^{\perp} \subset W'$ and 
$\text{im}(\phi'')= \text{im}(\phi')^{\perp} \subset \text{im}(\phi)$. We also identify $W/W'$ with $W'^\perp \subset W$. We use these identifications 
to define $\Phi,\Phi',\Phi''$ in the natural way.
Then $\Phi$ is just the identification $\text{im}(\phi)\oplus \text{im}(\phi)^\perp = W$. 
On the other hand, we view $\Phi''\oplus\text{id}_{\text{im}(\phi')}$ as a map
\[
	\text{im}(\phi)\oplus \text{im}(\phi)^\perp \to \text{im}(\phi')\oplus \text{im}(\phi')^\perp\oplus W'^\perp
\]
where $\text{im}(\phi')^\perp\subset W'$. This last expression uses the identification
$\text{im}(\phi) = \text{im}(\phi') \oplus \text{im}(\phi')^\perp$ 
where $\text{im}(\phi')^\perp\subset \text{im}(\phi)$, followed by 
the identification $\text{im}(\phi')^\perp \oplus \text{im}(\phi)^\perp = \text{im}(\phi')^{\perp}\oplus W'^\perp$, where on the left $\text{im}(\phi')^\perp \subset \text{im}(\phi)$ but on the right we have the larger complement $\text{im}(\phi')^\perp \subset W'$. These are just two different decompositions of $\text{im}(\phi')^\perp \subset W$. Then, $\Phi'\oplus\text{id}_{W/W'} $, viewed as a map
\[
	\text{im}(\phi')\oplus \text{im}(\phi')^\perp\oplus W'^\perp\to W,
\]
where again $\text{im}(\phi')^\perp\subset W'$, first uses the identification 
$\text{im}(\phi')\oplus \text{im}(\phi')^\perp = W'$, and then the identification 
$W'\oplus W'^\perp=W$. From this perspective, from which everything happens inside $W$ 
and uses its various orthogonal decompositions, (\ref{eq:thedecomp}) is clear, 
and thus (\ref{eq:F}) is established; (\ref{eq:G}) is similar. We return to establishing (\ref{eq:suffices}). 
We now know the left hand side is
	\begin{align*}
		& (-1)^{s_{12,3}+d_{12}d_{12,3}+d_{12}(a_{3}+c_3)+a_2c_3} F^{-1}(\beta_1\wedge\beta_2\wedge\beta_3)\\
					& \hfill\hfill\qquad \wedge\alpha_1\wedge\gamma_1\wedge\gamma_2\wedge\gamma_3\wedge G^{-1}(\alpha_2\wedge\alpha_3).
	\end{align*}
We can also compute the right side of (\ref{eq:suffices}):
\begin{align*}
		(-1)&^{s_{1,23}+d_{23}d_{1,23}}\delta_{23}\wedge \delta_{1,23} \wedge \beta_{1,23}\wedge\alpha_1\wedge\gamma_1\wedge\gamma_2\wedge\gamma_3\wedge\zeta_{23}\wedge\zeta_{1,23}\wedge\delta_{1,23}\wedge \delta_{23} \\
					&= (-1)^{s_{1,23}+d_{23}d_{1,23}+d_{23}(b_1+a_2)+a_2a_3}F^{-1}(\beta_1\wedge\beta_2\wedge\beta_3)\\
					& \qquad \qquad\wedge\alpha_1\wedge\gamma_1\wedge\gamma_2\wedge\gamma_3\wedge G^{-1}(\alpha_2\wedge\alpha_3).
\end{align*}
We have used the necessary analogues of (\ref{eq:F}) and (\ref{eq:G}).
Thus (\ref{eq:suffices}) holds if the quantity
	\begin{eqnarray*}
		s_{1,23}+d_{23}(d_{1,23}+b_1+a_2)+ a_2(a_3+c_3)+s_{12,3}+d_{12}(d_{12,3}+a_{3}+c_3)
	\end{eqnarray*}
	is even. Using $d_{23}+d_{1,23}=d_{12}+d_{12,3}$, this is easily verified. This 
establishes (\ref{eq:linalgass}).
Finally, we remark on the absence of the $H_{12}$ maps in our setup. In our application, we can choose the relevant maps $H_{12}$ and $H_{12,3}$ so that
\[
	(H_{12,3})(H_{12}\oplus \text{id}_{H_2^+(X_3)\oplus \text{ker}(f_{12,3})}) = H
\]
where we are using some chosen map
\[
	H: H_{2}^+(X_2)\oplus H_2^+(X_2)\oplus H_2^+(X_3)\oplus\text{ker}(f)\xrightarrow{\sim} H_{2}^+(X_{123})
\]
associated to the natural short exact sequence, and the identifications (\ref{eq:kerid}). 
This is established just as was (\ref{eq:F}). 
$H_{23}$ and $H_{1,23}$ can be chosen similarly, and this compatibility allows the above argument 
to carry through.
\end{proof}
\vspace{10px}

Now we define distinguished identity homology orientations. If $\x=\y\times [0,1]$, 
then $\mathcal{L}(\x)=H_1(Y)\oplus H_1(\x)$. Let $\alpha$ be 
any orientation of $H_1(Y)$, and 
choose an orientation $\beta$ of $H_1(\x)$ such that $\alpha=\beta$ under the natural 
identification of $H_1(\x)$ with $H_1(\y)$. Then define
\[
	\mu_{\y}^{\text{id}} := (-1)^{\frac{1}{2}(b_1(Y)^2+b_1(Y))}\beta\wedge\alpha
\]
to be the distinguished identity homology orientation of $\y\times[0,1]$.

\begin{prop}
	Whenever $\mu$ is a homology orientation of a cobordism $\x$ with incoming boundary $\y$, we 
	have $\mu\circ\mu_{\y}^\text{id}=\mu$. Similarly, if $\x$ has outgoing boundary $\y$, then 
	$\mu_{\y}^\text{id}\circ\mu=\mu$.
\end{prop}

\begin{proof}
Suppose $\x$ has incoming boundary $\y$, i.e. $\x:\y\to\y'$. We let $\x_1=\y\times [0,1]$ 
and $\x_2=\x$ and use the notation of Definition \ref{def:homcom}. We have $\text{im}(f_{12})=H_1(\y)$ 
and thus $d_{12}=b_1(\y)$. We identify $\x_{12}$ with $\x_2=\x$. Choose the section of the exact sequence (\ref{eq:exseq1}), which is 
a map $H_1(\x)\to H_1(\y\times [0,1])\oplus H_1(\x)$, to be of the form $y\mapsto (0,y)$. The induced isomorphism 
$F_{12}:H_1(\y)\oplus H_1(\x)\to H_1(\y\times [0,1])\oplus H_1(\x)$ is of the form $(x,y)\mapsto (x,y-\pi(x))$
where $\pi:H_1(\y)\to H_1(\x)$ is induced by inclusion. Let $\mu = \mu_2 = \beta_2\wedge\alpha_2\wedge\gamma_2$ 
where $\beta_2,\alpha_2,\gamma_2$ are respective orientations of $H_1(\x),H_1(\y),H_2^+(\x)$. 
Write $\mu_1 = \mu_{\y}^{\text{id}} = (-1)^{\frac{1}{2}(b_1(Y)^2+b_1(Y))}\beta_1\wedge\alpha_1$ as above, 
where $\alpha_1=\alpha$ and $\beta_1=\beta$. 
Choose $\delta_{12} = \alpha_1$. Then
\[
	F_{12}^{-1}(\beta_1\wedge\beta_2) = \delta_{12}\wedge\beta_{12}
\]
where $\beta_{12}=\beta_2$. We can choose $\alpha_2=\alpha_1$ so that the condition $\zeta_{12}\wedge\delta_{12} = \alpha_2$ ($G_{12}$ implicit) forces $\zeta_{12}$ to be the canonical $+1$ orientation 
of the $0$-vector space. Similarly, $\gamma_1$ is taken to be $+1$, and the expression $H_{12}(\gamma_1\wedge\gamma_2\wedge \zeta_{12})$ may be regarded as equal to $\gamma_2$. The sign $s$ in Definition \ref{def:homcom} is 
equal to $\frac{1}{2}(b_1(Y)^2+b_1(Y))$, and so cancels with the sign in $\mu_\y^\text{id}$. 
All together, Definition \ref{def:homcom} yields
\[
	\mu\circ\mu_{\y}^\text{id}=\beta_2\wedge\alpha_2\wedge\gamma_2 = \mu.
\]
Next, suppose $\x$ has outgoing boundary $\y$, i.e. $\x:\y'\to\y$. 
Now we write $\x=\x_1=\x_{12}$ and $\y\times[0,1]=\x_2$ and, correspondingly, we swap the indices 
for the above orientations and write
$\mu = \mu_1 = \beta_1\wedge\alpha_1\wedge\gamma_1$ 
and $\mu_\y^\text{id} = (-1)^{\frac{1}{2}(b_1(Y)^2+b_1(Y))}\beta_2\wedge\alpha_2 =  \mu_2$.
Choose the section of the exact sequence (\ref{eq:exseq1}), which is 
a map $H_1(\x)\to H_1(\x)\oplus H_1(\y\times [0,1])$, to be of the form $y\mapsto (y,0)$.
Now the induced map $F_{12}:H_1(\y)\oplus H_1(\x)\to H_1(\x)\oplus H_1(\y\times [0,1])$ is of the form $(x,y)\mapsto (y+\pi(x),-x)$. Choose $\delta_{12}=\alpha_2=\beta_2$ and so on, just as above. Then
\[
	F_{12}^{-1}(\beta_1\wedge\beta_2)=\delta_{12}\wedge\beta_{12}
\]
where $\beta_{12}=(-1)^{b_1(\y)b_1(\x)+b_1(\y)}\beta_1=(-1)^t \beta_1$. The exponent $s$ in 
Definition \ref{def:homcom} is given by
\[
	\frac{1}{2}\left(b_1(\y)^2-b_1(\y)\right) + b_1(\y)b_1(\x) \mod 2.
\]
We see that $s+t\equiv \frac{1}{2}(b_1(Y)^2+b_1(Y)) \mod 2$. This cancels with the sign 
put in front of $\mu_\y^\text{id}$, and we obtain from Definition \ref{def:homcom} 
the identity $\mu_{\y}^\text{id}\circ\mu=\mu$, just as before.
\end{proof}
\vspace{10px}

In the remainder of this section, we describe how our composition rule can be described in the setting of Fredholm determinant line bundles, as in \cite[\S 20.2]{kmm}, the purpose of which is to show that our rule is compatible with a construction of instanton homology. As such, the following details are not needed to understand the rest of the paper. 

In the Fredholm setting, a homology orientation of $X$ is an orientation of
$\text{det}(D)$, where $D$ is the operator $-d^\ast \oplus d^+$ acting on suitably weighted Sobolev spaces over $X$ with cylindrical ends attached. Recall that
\[
	\text{det}(D) = \ext^{\text{max}}(\text{ker}(D))\otimes \ext^{\text{max}}(\text{coker}(D)^\ast).
\]
The Sobolev weights are chosen such that we have natural identifications
\[
	\text{ker}(D) = H^1(X), \quad \text{coker}(D) = H^1(Y)\oplus H_+^2(X),
\]
where $Y$ is the incoming end of $X$, cf. \cite[Prop. 3.15]{d}. Note that an orientation of a vector space induces, in a natural way, an orientation of its dual space. Since we are working with real coefficients, homology and cohomology groups are dual to one another, so an orientation of $\text{det}(D)$ is the same as an orientation of $\mathcal{L}(X)$.

Let us now suppose we are in the situation of Definition \ref{def:homcom}, so that $\mu_i$ is an orientation of $\mathcal{L}(X_i)$, or equivalently $\text{det}(D_i)$, for $i=1,2$. We again write $\mu_i = \beta_i\wedge\alpha_i\wedge\gamma_i$ where now we view $\beta_i$ as orienting $\text{ker}(D)$ and $\alpha_i\wedge\gamma_i$ as orienting $\text{coker}(D)$ (or its dual). We will denote the composition of $\mu_1$ and $\mu_2$ as given in this setting by
\[
	\mu_2 \;\overline{\circ}\; \mu_1
\]
to distinguish it from our previous rule. The composition $\mu_2 \;\overline{\circ}\; \mu_1$ goes in two steps. First, we use the $\mu_i$ to orient $\text{det}(D_1\oplus D_2)$, which is identified with
\[
	\ext^\text{max}(\text{ker}(D_1)\oplus \text{ker}(D_2)) \otimes  \ext^\text{max}(\text{coker}(D_1)\oplus \text{coker}(D_2))^\ast.
\]
We use the following general rule for doing this: if $K_i\wedge C_i$ is an orientation for $\text{det}(D_i)$ where $K_i$ orients $\text{ker}(D_i)$ and $C_i$ orients $\text{coker}(D_i)$ (or its dual), then we orient $\text{det}(D_1\oplus D_2)$ by
\[
	 (-1)^{\dim\text{coker}(D_2)\text{index}(D_1)}(K_2\wedge K_1)\wedge (C_1\wedge C_2).
\]
This is a slight modification of the rule in \cite[Lemma 20.2.1]{kmm} but is easily seen to be associative; the difference between the two rules is the sign $(-1)^s$ where
\[
	s= \dim\text{coker}(D_1)\dim\text{ker}(D_2) + \text{index}(D_2)\dim\text{ker}(D_1).
\]
Applying this procedure to $\mu_1$ and $\mu_2$, we obtain the orientation
\[
	\mu' := (-1)^{(a_2 + c_2)(a_1+b_1+c_1)}(\beta_2\wedge\beta_1)\wedge(\alpha_1\wedge\gamma_1\wedge \alpha_2\wedge\gamma_2) 
\]
of $\text{det}(D_1\oplus D_2)$, where $a_i=\dim H^1(Y_i)$, $b_i=\dim H^1(X_i)$ and $c_i=\dim H_+^2(X_i)$.

The second step in describing the composition rule in this setting involves relating $\text{det}(D_1\oplus D_2)$ to $\text{det}(D_{12})$ by means of a (Fredholm) homotopy from the operator $D_1\oplus D_2$ to $D_{12}$, where $D_{12}$ is the operator associated to $X_{12}$. We will use the notation of \cite[\S 20.2]{kmm}. Let $P_s$ for $s\in [0,1]$ be such a homotopy, so that $P_0=D_1\oplus D_2$ and $P_1=D_{12}$. To be precise, we should understand these two aforementioned operators as having the same domain and codomain; this may be achieved using the finite cylinder setup as in \cite{kmm}. Denoting our codomain by $B$, choose $J\subset B$ so that $P_s^{-1}J + J = B$ for all $s$. We have for each $s$ an exact sequence
\begin{equation}
	0 \to \text{ker}(P_s) \xrightarrow[]{j}  P_s^{-1}J \xrightarrow[]{k} J \xrightarrow[]{l} \text{coker}(P_s) \to 0.\label{eq:detexseq}
\end{equation}
We use the following general rule for orienting $\text{det}(P_s)$  given an orientation $\mu''$ of the line $\ext^\text{max}P_s^{-1}J \otimes \ext^\text{max} J^\ast$ using the exact sequence (\ref{eq:detexseq}): write
\begin{equation}
	\mu'' = (K \wedge D) \wedge (k(D) \wedge C)\label{eq:mupp}
\end{equation}
where $K$ is an orientation of $\text{im}(j)$, $D$ of $\text{im}(j)^\perp$, and $C$ of $k(\text{im}(j)^\perp)^\perp$; then orient $\text{det}(P_s)$ by
\begin{equation}
	(-1)^{\phi(d)}j^{-1}(K)\wedge l(C)\label{eq:genrule}
\end{equation}
where $\phi(x) := (x^2-x)/2$ and $d:=\dim(\text{im}(j)^\perp)$. In our situation, we choose $J$ to be a complement of $\text{im}(P_0) = \text{im}(D_1\oplus D_2)$, and we make the identification
\[
	J = H^1(Y_1)\oplus H_+^2(X_1) \oplus H^1(Y_2) \oplus H^2_+(X_2).
\]
We choose the homotopy so that $P_s^{-1}J = \text{ker}(P_0)=\text{ker}(D_1\oplus D_2)$ for all $s$, so that
\[
	P_s^{-1}J = H^1(X_1)\oplus H^1(X_2).
\]
In particular, we have an identification of $\ext^\text{max}P_1^{-1}J \otimes \ext^\text{max} J^\ast$ with $\text{det}(D_1\oplus D_2)$, which is oriented by $\mu'$. Noting that the maps in (\ref{eq:detexseq}) for $s=1$ come from the Mayer-Vietoris maps as in Definition \ref{def:homcom}, we can write $\mu''$ from $\mu'$ as in (\ref{eq:mupp}):
\[
	\mu'' = (-1)^t(\beta_{12}\wedge \delta_{12}) \wedge (\delta_{12}\wedge \gamma_{12}).
\]
In this expression, and in all to follow, the maps $F_{12}$, $G_{12}$ and $H_{12}$ from Definition \ref{def:homcom} as well as the maps in (\ref{eq:detexseq}) will be implicitly understood, e.g. $F_{12}(\delta_{12}\wedge\beta_{12})$ is the same as $\delta_{12}\wedge\beta_{12}$. The orientation $\beta_{12}$ plays the role of $K$ above, $\gamma_{12}$ that of $C$, and $\delta_{12}$ that of $D$. The sign $(-1)^t$ is given by
\[
	t = (a_2 + c_2)(a_1+b_1+c_1) + d_{12}(b_1+b_2 + d_{12}) + b_1b_2,
\]
where $d_{12}$ is as in Definition \ref{def:homcom}. The first term in $t$ is from $\mu'$ and the rest are added to ensure that $\beta_{12}$ is defined by the condition $\delta_{12}\wedge \beta_{12} = \beta_1\wedge \beta_2$, to match Definition \ref{def:homcom}. The orientation $\gamma_{12}$ is defined by the condition $\delta_{12} \wedge\gamma_{12} = \alpha_1 \wedge \gamma_1\wedge\alpha_2 \wedge \gamma_2$. The general rule that takes $\mu''$ to (\ref{eq:genrule}), applied to our $\mu''$, tells us the final orientation of $\text{det}(D_{12})$:
\[
	\mu_2 \;\overline{\circ}\; \mu_1 = (-1)^{\phi(d_{12})+t}\beta_{12}\wedge\gamma_{12}.
\]
Now write $\alpha_2=\zeta_{12}\wedge\delta_{12}$ as in Definition \ref{def:homcom}. We compute
\[
	\mu_2 \;\overline{\circ}\; \mu_1 = (-1)^{r}\beta_{12}\wedge\alpha_1\wedge\gamma_1\wedge\gamma_2\wedge\zeta_{12},
\]
\[
	r = \phi(d_{12}) + t + d_{12}(a_2 + d_{12} + c_1 + a_1) + c_2(d_{12}+a_2).
\]
The sign given by $r$ does not match the sign given by $s$ in Definition \ref{def:homcom}, and so this composition rule is not the same as the one previously defined. However, there is an automorphism $\mu\mapsto \overline{\mu}$ on the class of all homology orientations that intertwines the two rules. Given a homology orientation $\mu$ of a cobordism $X$, we set
\[
	\overline{\mu} = (-1)^{\phi(b_1(X)) + \phi(b_1(Y)+b_2^+(X))}\mu
\]
where $Y$ is the incoming end of $X$. Then we have
\[
	\overline{(\overline{\mu_1}\;\overline{\circ}\;\overline{\mu_2})} = \mu_1\circ\mu_2.
\]
The verification is a straightforward computation that we omit. It follows that the composition rule $\mu_1\circ\mu_2$ of Definition \ref{def:homcom} is compatible with a construction of Floer homology, and it is this rule that we will use in our computations below.

\vspace{10px}

\subsection{The $E^1$-page}\label{sec:branched}

In this section we identify the $E^1$-page of (\ref{sps}) 
with the chain complex that computes reduced odd Khovanov homology.
We fix as before a diagram $D$ for the $m$-component link $\lt$ with crossings decorated by 
arcs as in \S \ref{sec:oddkh}. We let $\y_v=\Sigma(D_v)$
for each $v\in\{0,1\}^m$ so that $\y_v$ is homeomorphic to $\#^k S^1\times S^2$ 
when $D_v$ has $k+1$ circles. The $E^1$-page and differential of
the spectral sequence we are considering is given by
\[
	E^1=\bigoplus_{v\in\{0,1\}^m} \ih^\#(\y_v), \quad d^1=\sum (-1)^{\delta(v,w)}\ih^\#(\x_{vw}),
\] 
where the sum runs over $v<w$ with $|w-v|_1=1$ and $v,w\in\{0,1\}^m$.
In writing $d^1$, 
we have chosen homology orientations $\mu_{vw}$ of the $\x_{vw}$ so that $\mu_{uw}\circ\mu_{vu}=
\mu_{tw}\circ\mu_{vt}$ always holds.
We are also using that the relevant bundles $\xbb_{vw}$ 
are trivial. This is because each such bundle lies over a cobordism which is
$D^2\times S^2\setminus \text{int}(D^4)$ running along a product cobordism, see (\ref{eq:easycob}); 
since we have arranged that the restriction of each such bundle over the boundary is trivial, 
for topological reasons the bundle must be trivial.

Let $\chain = \bigoplus \chain_v$ be the reduced odd Khovanov chain group for the diagram
$D$ and $\partial'=\sum \partial_{vw}'$ its pre-differential. For each $v\in\{0,1\}^m$ we define 
an isomorphism
\[
	\Phi_v:\chain_v\to  \ih^\#(\y_v)
\]
defined as a composition $\Phi_v = \phi_v\circ\rho_v$ 
where $\phi_v:\ext^\ast(H_1(\y_v;\mathbb{Z}))
	\to\ih^\#(\y_v)$ is from 
\S \ref{sec:example} and $\rho_v:\chain_v \to\ext^\ast(H_1(\y_v;\mathbb{Z}))$ is defined by lifting arcs in $D_v$ to loops in $\y_v$, and is explained in the following paragraph. For the $\phi_v$ maps, we of course fix orientations $\mu_v$ for each $H_1(\y_v;\mathbb{R})$. We write $\Phi:\chain\to E^1$ for the sum of the $\Phi_v$ maps.

Recall $\chain_v=\ext^\ast(V_v)$, 
and that $\y_v$ is branched over $D_v\subset S^3$. Let 
$S$ be the union of disks in the plane enclosed by the circles in $\dt_v$. 
They can be pushed out so that they are disjoint and form a Seifert surface for the union 
of circles. Let $N$ be a neighborhood of the circles, a union of solid tori. Then 
$\y_v$ can be written as a gluing
\[
	\y_v = \y_- \cup N \cup \y_+
\]
where $\y_{\pm}=S^3\setminus(S\cup N)$. Distinguishing one of the copies of $S^3\setminus(S\cup N)$, 
say $\y_+$, allows us to lift an arc $x$ in $\dt_v$ to an 
\textit{oriented} loop $\widetilde{x}$ in $\y_v$: the orientation is obtained by 
locally lifting the orientation of $x$ to the part of $\widetilde{x}$ in $\y_+$. 
Then $x \mapsto [\widetilde{x}]$ is an 
isomorphism from $V_v$ to $H_1(\y_v;\mathbb{Z})$, and $\rho_v$ is taken to 
be the extension of this map to exterior algebras. 
We can construct the $\rho_v$ in this way so that it is uniform among all $v$, in the 
sense that there are natural ways of identifying $\y_v$ with $\y_w$ away from surgery 
(or resolution) areas, and in these areas we can lift arcs the same way. 

In summary, the map 
$\Phi_v$ is described as follows. Let $x=x_1\wedge\cdots\wedge x_i$ be a wedge 
of arcs in $\chain_v$. Lift the arcs to embedded loops $\widetilde{x}_{j}$ in the branched double cover $\y_v$ as above. Choose $x_{i+1},\ldots,x_k$ and their lifts such that $\mu_v=[\widetilde{x}_1\wedge \cdots \wedge \widetilde{x}_k]$. Attach 2-handles to $\widetilde{x}_{1},\ldots,\widetilde{x}_i$ and 3-handles and a 4-handle as in \S \ref{sec:example} to 
obtain a cobordism $\x:\emptyset\to\y_v$ homology oriented by 
$[\widetilde{x}_{i+1}\wedge\cdots\wedge\widetilde{x}_k]$. Then $\Phi_v(x)=[\x]^\#$.
The following completes the proof of 
Theorem \ref{thm:1} up to gradings, which are dealt with in the next section.

\begin{lemma}
$\Phi^{-1}d^1\Phi=\sum \varepsilon_{vw}\partial_{vw}'$ where $\varepsilon_{vw}$ is a valid edge assignment.
\end{lemma}

\begin{proof}
Let $v,w\in\{0,1\}^m$ with $v < w$ and $|w-v|_1=1$.
There are two cases to consider,
depending on whether $D_{vw}$ is a split or a merge diagram. 
We retain the convention from \S \ref{sec:homor} that singular homology $H_\ast(X)$ 
is taken with real coefficients. 
For most of the proof, we conflate the symbols $x$ and $\widetilde{x}$, 
where $x$ is an arc (usually viewed as a class in $V_v$) and $\widetilde{x}$ is its lift to $Y_v$ 
(usually viewed as a class in $H_1(\y_v)$). That is, the maps $\rho_v$ from above are implicit.
Suppose first we are in the split case.
Let $k=b_1(\x_{vw})$. Note that
$b_2^+(\x_{vw})=0$, and $\x_{vw}:\y_v\to \y_w$ is homeomorphic to
\begin{equation}
	\left(\y_v\times [0,1]\right)\Join\left(D^2\times S^2\setminus \text{int}(D^4)\right).\label{eq:easycob}
\end{equation}
We note that we may also view $\x_{vw}$ as the branched double cover of a pair of pants properly 
embedded in $S^3\times [0,1]$. We have $\mathcal{L}(\x_{vw})=H_1(\y_v)\oplus H_1(\x_{vw})$.
We will follow the notation of Definition \ref{def:homcom}, 
setting $\x_1=\x$ and $\x_2=\x_{vw}$. 
Choose orientations $\alpha_2$ and $\beta_2$ of $H_1(\y_v)$ and $H_1(\x_{vw})$,
respectively. We can identify $H_1(\y_v)=H_1(\x_{vw})$ using the map induced by inclusion, 
and we choose to impose the condition $\alpha_2=\beta_2$. Define $\varepsilon_{vw}'=\pm 1$ by
\[
	\mu_{vw} = \varepsilon_{vw}'\beta_2\wedge\alpha_2.
\]
Let $x=x_1\wedge\cdots\wedge x_i\in \chain_v$.
Recall that $\Phi_v(x)=[\x]^\#$ where $\x$ is obtained by attaching 2-handles to 
$x_1,\ldots,x_i$ along with some 3-handles and a 4-handle. Choose $x_{i+1},\ldots,x_k$ so that 
$\mu_v=[x_1\wedge\cdots\wedge x_k]$. Then $\mathcal{L}(\x)=H_1(\x)$ is generated by 
$x_{i+1},\ldots,x_k$ and $\x$ is homology 
oriented by $\beta_1 := [x_{i+1}\wedge\cdots\wedge x_k]$.
We can identify $\mathcal{L}(\x_{vw}\circ\x)=H_1(\x_{vw}\circ\x)$. 
Note that $\text{im}(f_{12})=H_1(\y_v)$, so $d_{12}=k$.
Choose the section in the exact sequence (\ref{eq:exseq1}), 
which in this case is a map $H_1(\x_{vw}\circ\x)\to H_1(\x)\oplus H_1(\x_{vw})$,
to be of the form $y \mapsto (y,0)$. The induced 
isomorphism $F_{12}:H_1(\y_v)\oplus H_1(\x_{vw}\circ\x)\to H_1(\x)\oplus H_1(\x_{vw})$, 
written as in (\ref{eq:split1}), can be written
\[
	F_{12}:\mathbb{R}\{x_1,\ldots,x_{k}\}\oplus\mathbb{R}\{x_{i+1},\ldots,x_k\}\to\mathbb{R}\{x_{i+1},\ldots,x_k\}\oplus\mathbb{R}\{x_1\ldots,x_{k}\},
\]
\[
	F_{12}(x_p,x_q)=(x_q + \pi(x_p),-x_p),
\]
where $\pi:H_1(\y_v)\to H_1(\x)$ is a projection induced by inclusion.
Writing $\beta_2=\delta_{12}$, we have
\[
	F_{12}^{-1}(\beta_1\wedge\beta_2) =(-1)^{k}\beta_1\wedge\beta_2=\delta_{12}\wedge\beta_{12}
\]
where $\beta_{12}=(-1)^{(k-i)k + k}\beta_1=(-1)^{ki}\beta_1$. Using Definition \ref{def:homcom}, we obtain
\[
	\ih^\#(\x_{vw})\Phi_v(x) = (-1)^{(k^2+k)/2}\varepsilon_{vw}' [\x_{vw}\circ\x]^\#
\]
where $\x_{vw}\circ\x$ is homology oriented by $\beta_1$. 
The sign $(-1)^{(k^2+k)/2}$ is obtained by computing
\[
	 ki + \left((k^2-k)/2 + (k-i)k\right),
\]
where the term $ki$ is from $\beta_{12}$, and the the expression
inside the parentheses is from Definition \ref{def:homcom}. 
We mention that the condition $G_{12}(\zeta_{12}\wedge\delta_{12})=\alpha_2$ 
holds by $\alpha_{2}=\delta_{12}=\beta_2$ and setting $\zeta_{12}$ to 
be the canonical $+1$ orientation of the $0$ vector space.
Note that $[\x_{vw}\circ\x]^\# = \Phi_w(x_{vw}\wedge x)$ if and only if $\mu_w = [x_{vw}\wedge x_1 \wedge\cdots \wedge x_k]=x_{vw}\wedge \mu_v$; otherwise they differ in sign.
We record a sign $\varepsilon''_{vw}=\pm 1$ measuring this possible discrepancy between $\mu_v$ and $\mu_w$:
\[
	\mu_w = \varepsilon_{vw}'' x_{vw}\wedge\mu_v.
\]
Recalling that $d_{vw}^1=(-1)^{\delta(v,w)}\ih^\#(\x_{vw})$ and $\partial_{vw}'(x)=x_{vw}\wedge x$, we conclude
\[
	\Phi_{w}(\partial_{vw}'(x))=\varepsilon_{vw}d^1_{vw}(\Phi_v(x))
\]
where $\varepsilon_{vw}=\pm 1$ is given by
\[
	\varepsilon_{vw}=(-1)^{(k^2+k)/2+\delta(v,w)}\varepsilon'_{vw}\varepsilon''_{vw}.
\]
Now suppose we are in the merge case. Again, let $k=b_1(\x_{vw})$. 
As before, $b_2^+(\x_{vw})=0$ and
the cobordism $\x_{vw}:\y_v\to \y_w$ is now homeomorphic to
\[
	\left(\y_w\times [0,1]\right)\Join\left(D^2\times S^2\setminus \text{int}(D^4)\right).
\]
We identify $H_1(\x_{vw})=H_1(\y_w)$, and 
write $\mathcal{L}(\x_{vw})=H_1(\y_v)\oplus H_1(\y_w)$. 
Note the natural codimension 1 inclusion $H_1(\y_w)\subset H_1(\y_v)$. 
A complement for $H_1(\y_w)$ is generated by $x_{vw}$.
Let $\alpha_2$ be an orientation for $H_1(\y_v)$.
Define $\varepsilon_{vw}'=\pm 1$ by
\[
	\mu_{vw} = \varepsilon_{vw}' \beta_2 \wedge \alpha_2, \quad \beta_2 = \alpha_2 \;\llcorner\;  x_{vw}.
\]
The condition $\beta_2 = \alpha_2 \;\llcorner\;  x_{vw}$ is equivalently expressed (or is defined) by $\beta_2\wedge x_{vw} = \alpha_2$. Let $x=x_1\wedge\cdots\wedge x_i\in\ext^i(V_v)$.
If $x_{vw}$ is among 
$x_1,\ldots,x_i$ (or linearly dependent on them), the 4-manifold $\x$ constructed by attaching 
2-handles to $x_1,\ldots,x_i$ and some 3-handles and a 4-handle,
once paired with $\x_{vw}$ to form $\x_{vw}\circ\x$, contains 
a non-trivial $S^2$-bundle over $S^2$ as in \S \ref{sec:example}, so $[\x_{vw}\circ\x]^\#=0$.
Choose $x_{i+1},\ldots,x_{k+1}$ so that $\mu_v=[x_1\wedge\cdots \wedge x_{k+1}]$; we may assume that $x_{vw}=x_{k+1}$.
We may also set $\alpha_2 = \mu_v$, so that $\beta_2 = [x_1\wedge\cdots\wedge x_k]$. 
Recall $\Phi_v(x)=[\x]^\#$ where $\x$ is homology oriented by 
$\beta_1=[x_{i+1}\wedge\cdots\wedge x_{k+1}]$.
There is a codimension 1 
inclusion $H_1(\x_{vw}\circ\x)\subset  H_1(\x)$. The vector space 
$H_1(\x_{vw}\circ\x)$ is generated by $x_{i+1},\ldots,x_{k}$ and a complement for 
$H_1(\x_{vw}\circ\x)$ in $H_1(\x)$ is generated by $x_{vw}=x_{k+1}$.
Choose the section in the exact sequence (\ref{eq:exseq1}), 
which is a map $H_1(\x_{vw}\circ\x)\to H_1(\x)\oplus H_1(\x_{vw})$,
to be of the form $y\mapsto (y,0)$. 
As in the split case, $\text{im}(f_{12})=H_1(\y_v)$. We obtain an isomorphism
$F_{12}:H_1(\y_v)\oplus H_1(\x_{vw}\circ\x) \to H_1(\x)\oplus H_1(\x_{vw})$ that 
takes the form
\[
	F_{12}:\mathbb{R}\{x_1,\ldots,x_{k+1}\}\oplus\mathbb{R}\{x_{i+1},\ldots,x_{k}\}\to\mathbb{R}\{x_{i+1},\ldots,x_{k+1}\}\oplus\mathbb{R}\{x_1,\ldots,x_{k}\},
\]
\[
	F_{12}(x_p,x_q)=(x_q + \pi_1(x_p),-\pi_2(x_p)),
\]
where $\pi_1:H_1(\y_v)\to H_1(\x)$ and $\pi_2:H_1(\y_v)\to H_1(\x_{vw})$ are projections 
induced by inclusion maps. 
\begin{figure}[t]
\includegraphics[scale=1]{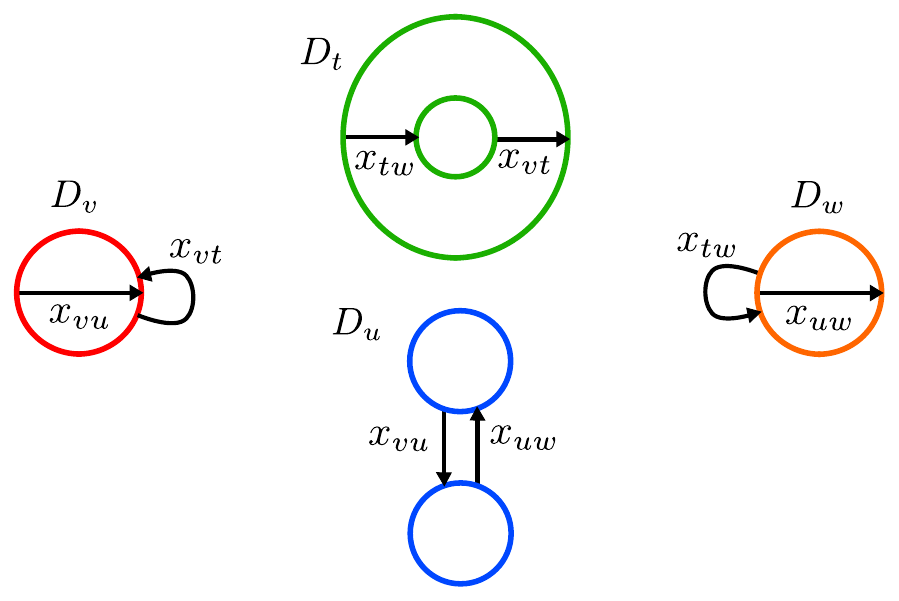}
\caption{Local pictures for four diagrams 
appearing in a type X face, starting at the diagram $D_v$ and ending at $D_w$. 
The circles in each diagram are colored so as to distinguish their roles in
Figure \ref{fig:thesign}.}
\label{fig:face}
\end{figure}
In particular, $\pi_1(x_p)=x_p$ if $p\geq i+1$ and is otherwise $0$, and 
$\pi_2(x_p)=x_p$ if $p\neq k+1$ and $\pi_2(x_{k+1})=0$. Recalling that 
$\beta_2=\alpha_2\;\llcorner\; x_{vw}$ and choosing $\delta_{12}=\alpha_2$, we have
\[
	F_{12}^{-1}(\beta_1\wedge\beta_2) =  (\beta_1\;\llcorner\; x_{vw})\wedge\alpha_2 = \delta_{12}\wedge\beta_{12}
\]
where $\beta_{12}=(-1)^{ki+i}(\beta_1\;\llcorner\; x_{vw})$. Note $\beta_1\;\llcorner\; x_{vw} =[x_{i+1}\wedge\cdots\wedge x_k]$. From Definition \ref{def:homcom} we obtain
\[
	I^\#(\x_{vw})[\x]^\# = (-1)^{(k^2-k)/2+1}\varepsilon_{vw}'[\x_{vw}\circ\x]^\#
\]
where $\x_{vw}\circ\x$ is homology oriented by $\beta_1\;\llcorner\; x_{vw}$.
We have computed the sign from
\[
	ki+i+\left(((k+1)^2-(k+1))/2 +(k+1-i)(k+1) \right),
\]
where $ki+i$ is from $\beta_{12}$, and 
the expression inside the parentheses is from the sign in Definition \ref{def:homcom}.
On the other hand, $[\x_{vw}\circ\x]^\# = \Phi_w(x)$ exactly when $\mu_w = [x_1\wedge \cdots \wedge x_k] = \mu_v\;\llcorner\; x_{vw}$.  Accounting for this, we define $\varepsilon''_{vw}=\pm 1$ by
\[
	\mu_v = \varepsilon_{vw}''\mu_w\wedge x_{vw}.
\]
Recalling that $\partial'_{vw}(x)=x$, we conclude
\[
	\Phi_{w}(\partial_{vw}'(x))=\varepsilon_{vw}d^1_{vw}(\Phi_v(x))
\]
where $\varepsilon_{vw}=\pm 1$ is given by
\[
	\varepsilon_{vw}=(-1)^{(k^2-k)/2+1+\delta(v,w)}\varepsilon'_{vw}\varepsilon''_{vw}.
\]
In summary, we have shown that
\[
	\Phi^{-1}d^1\Phi=\sum_{\substack{v<w\\|w-v|_1=1}}  \varepsilon_{vw}\partial'_{vw}
\]
where we have determined $\varepsilon_{vw}$ in the split and merge cases separately.
It remains to show that $\varepsilon_{vw}$ is a valid edge assignment.
The first condition, that the total differential squares to zero, already falls 
out from the spectral sequence. We now show that the $\varepsilon_{vw}$ satisfy the second 
condition, that is, if $v,u,t,w$ form a type X face, then the product
\begin{equation}
	\varepsilon_{vu}\varepsilon_{vt}\varepsilon_{uw}\varepsilon_{tw}\label{eq:eps}
\end{equation}
is always +1 or -1, 
independently of the particular face chosen; and if they form a type Y face, the same is true, 
and the sign is opposite the type X case.
We fix such a type X face. Note
\[
	\delta(v,u) + \delta(v,t) + \delta(u,w) + \delta(t,w) \equiv 0 \mod 2.
\]
Next we consider the $\varepsilon_{vw}''$ terms. We compute
\[
	x_{vu}\wedge\mu_v = \varepsilon_{vu}''\mu_u = \varepsilon_{vu}''\varepsilon_{uw}''\mu_w\wedge x_{uw}.
\]
Since $x_{vu}=-x_{uw}$ in $D_u$, the above can be abbreviated to $\mu_v = (-1)^{k+1}\varepsilon_{vu}''\varepsilon_{uw}''\mu_w$. Similarly, we obtain $\mu_v = (-1)^{k+1}\varepsilon_{vt}''\varepsilon_{tw}''\mu_w$, implying $\varepsilon_{vu}''\varepsilon_{vt}''\varepsilon_{uw}''\varepsilon_{tw}''=1$. A similar argument in the type Y case yields $\varepsilon_{vu}''\varepsilon_{vt}''\varepsilon_{uw}''\varepsilon_{tw}''=1$ as well.  
To summarize, we may reconsider the problem with (\ref{eq:eps}) replaced by 
the expression $\varepsilon_{vu}'\varepsilon_{vt}'\varepsilon_{uw}'\varepsilon_{tw}'$.

Note $\mathcal{L}(\x_{vu}) = H_1(\y_v)\oplus H_1(\x_{vu})$, 
and, as this is a split cobordism, we have a natural identification
$H_1(\y_v)= H_1(\x_{vu})$. Choose respective orientations $\alpha_1$ and $\beta_1$ of 
$H_1(\y_v)$ and $H_1(\x_{vu})$ that agree under this identification. Recall that 
$\varepsilon_{vu}'$ has been defined by
\[
	\mu_{vu} = \varepsilon_{vu}'\beta_1\wedge\alpha_1.
\]
On the other hand, $\mathcal{L}(\x_{uw})=H_1(\y_u)\oplus H_1(\x_{uw})$, 
and, as this is a merge cobordism, there is a codimension 1 inclusion 
$H_1(\x_{uw})\subset H_1(\y_u)$ with a complement generated by $x_{uw}$. 
Let $\alpha_2$ be an orientation of $H_1(\y_u)$ and set $\beta_2 =\alpha_2\;\llcorner \; x_{uw}$. 
Then $\varepsilon_{uw}'$ has been defined by
\[
	\mu_{uw} = \varepsilon_{uw}'\beta_2\wedge\alpha_2.
\]
In this situation, the map $f_{12}$ of (\ref{eq:exseq1}) has a 1-dimensional 
kernel spanned by $x_{uw}$. 
In this way $\text{im}(f_{12})$ can be identified with $H_1(\y_v)$ and $H_1(\x_{vu})$. 
Let a section for the exact sequence (\ref{eq:exseq1}), here a map $H_1(\x_{vw})\to H_1(\x_{vu})\oplus H_1(\x_{uw})$, 
be given by $y\mapsto (y,0)$. The map $F_{12}:\text{im}(f_{12})\oplus H_1(\x_{vw}) \to 
H_1(\x_{vu})\oplus H_1(\x_{uw})$ of (\ref{eq:split1}) can be written
\[
	F_{12}:\mathbb{R}\{x_1,\ldots,x_{k}\}\oplus\mathbb{R}\{x_1,\ldots,x_k\}\to\mathbb{R}\{x_1,\ldots,x_k\}\oplus\mathbb{R}\{x_1,\ldots,x_{k}\},
\]
\[
	F_{12}(x_p,x_q)=(x_q + x_p,-x_p).
\]
Proceeding with the conditions of Definition \ref{def:homcom}, we find
\[
	F_{12}^{-1}(\beta_1\wedge\beta_2)= \delta_{12}\wedge\beta_{12}
\]
where $\delta_{12}=\alpha_1=\beta_2$ and $\beta_{12}=\beta_1$. 
We can arrange that $\beta_2=\beta_1$ under the appropriate identification.
The condition $G_{12}(\zeta_{12}\wedge\delta_{12})=\alpha_2$, 
having that $\alpha_2 =  \beta_2\wedge x_{uw}$, yields $\zeta_{12}=(-1)^k x_{uw}$. 
Using Definition \ref{def:homcom} we obtain
\[
	\mu_{uw}\circ\mu_{vu} = (-1)^{(k^2-k)/2+k(k+1)+k}\varepsilon_{vu}'\varepsilon_{uw}'\beta_1\wedge\alpha_1\wedge H^u_{12}(x_{uw})
\]
where we've used $k=b_1(\x_{vu})=d_{12}$. The superscript $u$ in $H^u_{12}$ distinguishes this map from the map $H^t_{12}$ which appears when $u$ is replaced by $t$. We obtain a similar equation 
for $\mu_{tw}\circ\mu_{vt}$ with $x_{uw}$ replaced by $x_{tv}$ and 
$\varepsilon_{vu}'\varepsilon_{uw}'$ replaced by $\varepsilon_{vt}'\varepsilon_{tw}'$. 
Because our setup includes the compatibility condition $\mu_{tw}\circ\mu_{vt}=\mu_{uw}\circ\mu_{vu}$, 
we conclude
\[
	\varepsilon_{vu}'\varepsilon_{vt}'\varepsilon_{tw}'\varepsilon_{uw}' = H^u_{12}(x_{uw})/H^t_{12}(x_{uw})=:\varepsilon.
\]
\begin{figure}[t]
\includegraphics[scale=1.20]{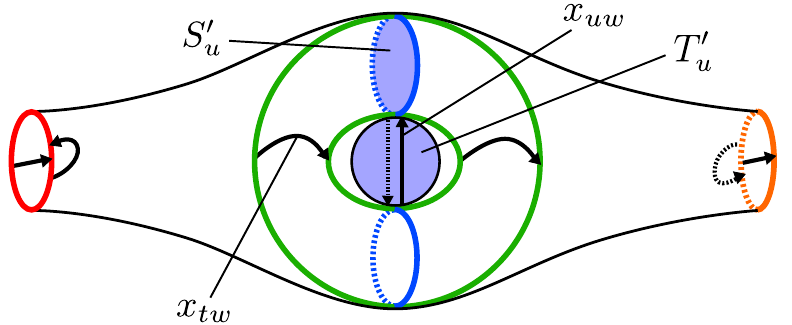}
\caption{This is an illustration (missing a dimension) of $S^4$ minus two 4-balls, 
with a properly embedded surface $F$, a torus with two disks removed, with
the local portions of the diagrams of the type X face from Figure \ref{fig:face} embedded; 
the circles of the diagrams lie on $F$, while only the endpoints of the arcs lie on $F$. 
The cobordism $\x_{vw}$ is the double cover over $S^4$ minus two 4-balls 
branched over $F$. The disk $S_u'$ lifts to a 2-sphere $S_u\subset\x_{vw}$ 
intersecting $\widetilde{x}_{uw}$ (the lift of $x_{uw}$) in one point. 
The disk $T_u'$ lifts to a 2-sphere $T_u\subset \x_{vw}$ intersecting 
$\y_u$ in $\widetilde{x}_{uw}$.}
\label{fig:thesign}
\end{figure}
In summary, 
we see that $\varepsilon$ is the sign determined by comparing the result of orienting
$H_2^+(\x_{vw})$ by $x_{uw}$ versus the result by using $x_{tw}$. 
Using the interpretation of the splitting map (\ref{eq:exseq3}) 
from \S \ref{sec:homor}, we obtain the following interpretation of $\varepsilon$. 
Here is a suitable moment to reintroduce the distinction between each arc $x$ 
and its lift $\widetilde{x}$.
Choose an oriented surface $S_u\subset \y_u$ transverse to 
$\widetilde{x}_{uw}$ with intersection product $[S_u]\cdot[\widetilde{x}_{uw}]=1$.
Choose an oriented surface $T_u$ with $T_u\cap \y_u =\widetilde{x}_{uw}$. 
Then 
\[
	[S_u] + [T_u] = H^u_{12}(\widetilde{x}_{uw}).
\] 
To illustrate this, we supply Figure \ref{fig:thesign}, where we use that
$\x_{vw}$ is a double cover of $S^4$ minus two 4-balls branched over 
a properly embedded torus with two disks removed.
Similarly, we can write $[S_t]+[T_t]=H^t_{12}(\widetilde{x}_{tw})$.
The sign $\varepsilon$ is then the intersection product of these classes:
\[
	\varepsilon = ([S_u] + [T_u])\cdot([S_t]+[T_t]).
\]
In fact, $[T_t]=\varepsilon[S_u]$.
From this it is clear that $\varepsilon$ only depends 
on the topology of the type X configuration. A type Y face is obtained from a 
type X face by reversing the direction of either $\widetilde{x}_{uw}$ 
or $\widetilde{x}_{tw}$, and $\varepsilon$ correspondingly changes sign.
\end{proof}
\vspace{10px}

\subsection{Gradings}\label{sec:finalgr}
In this section we prove that the spectral sequence preserves 
the relevant $\mathbb{Z}/4$-gradings, completing the proof of Theorem \ref{thm:1}. 
We then deduce Corollary \ref{cor:2}.
As usual, let $k=\dim(V_v)$. For $x\in\ext^i(V_v)\subset\chain$, the grading of $\Phi(x)$ in $E^1$ is given in (\ref{eq:grt}) by
\begin{equation}
	\text{gr}[E^1](\Phi(x))\equiv \text{gr}[\y_v](\Phi(x)) -\text{deg}(\xbb_{\boldsymbol{\infty}v})-|v|_1 \mod 4.\label{eq:gr4}
\end{equation}
We know, by the remark at the end of \S \ref{sec:example}, 
that $\text{gr}[\y_v](\Phi(x))\equiv 2k+i$. 
We have $\text{deg}(\xbb_{\boldsymbol{\infty}v})=\text{deg}(\xbb_{\boldsymbol{\infty}\mathbf{1}})-\text{deg}(\x_{v\mathbf{1}})$, since $\xbb_{v\mathbf{1}}$ is trivial. From (\ref{degofmap}) we compute
\[
	\text{deg}(\x_{v\mathbf{1}}) = -\frac{3}{2}(m-|v|_1)+\frac{1}{2}(b_1(\y_\mathbf{1})-k)
\]
using $\chi(\x_{vw})=|w-v|_1$, $\sigma(\x_{v\mathbf{1}})=0$ and $b_1(\y_v)=k$. We also compute
\[
	\text{deg}(\x_{\boldsymbol{\infty}\mathbf{1}}) = -\frac{3}{2}(2m + \sigma(\x_{\boldsymbol{\infty}\mathbf{1}}))+\frac{1}{2}(b_1(\y_\mathbf{1})-b_1(\Sigma(L)))
\]
knowing $\Sigma(L)=\overline{\y}_{\boldsymbol{\infty}}$. Recall from (\ref{eq:degmaster}) that $\text{deg}(\xbb_{\boldsymbol{\infty}\mathbf{1}}) \equiv \text{deg}(\x_{\boldsymbol{\infty}\mathbf{1}})+2\relp(\xbb_{\boldsymbol{\infty}\mathbf{1}})$.

\begin{lemma}\label{lem:gr}
$\relp(\xbb_{\boldsymbol{\infty}\mathbf{1}})\equiv \sigma(\x_{\boldsymbol{\infty}\mathbf{1}}) \mod 2$.
\end{lemma}

\noindent Before proving this lemma, we make our conclusion. In \cite{bloom}, Bloom computes 
$\sigma(\x_{\mathbf{0}\boldsymbol{\infty}})=\sigma-n_+$ and $b_1(\Sigma(L))=\nu$, where 
$\sigma$ and $\nu$ are the signature and nullity of $L$, respectively, 
and $n_{\pm}$ is the number of $\pm$ crossings of the diagram $D$. Note that $\x_{\boldsymbol{\infty}\mathbf{1}}$ and $\x_{\mathbf{1}\boldsymbol{\infty}}$ compose along $\y_\mathbf{1}$ to give a cobordism which, away from a manifold of signature $0$, has $m$ copies of $\cp$ connected summed to it (cf. $E$ from \S \ref{sec:decomp1}). In addition, since $\sigma(\x_{\mathbf{0}\mathbf{1}})=0$, we have $\sigma(\x_{\boldsymbol{\infty}\mathbf{1}})=\sigma(\x_{\boldsymbol{\infty}\mathbf{0}})$. Additivity of the signature again implies that $\sigma(\x_{\boldsymbol{\infty}\mathbf{1}}) = -m-\sigma+n_+$. Note $m=n_+ +n_-$. All together, (\ref{eq:gr4}) computes to 
\[
	i + 2n_- +\frac{3}{2}\left(n_+ + k\right)+\frac{1}{2}\left(|v|_1+\nu+\sigma\right) \mod 4,
\]
which is congruent to (\ref{eq:oddgr}). This completes the proof of Theorem \ref{thm:1}.

\begin{figure}[t]
\includegraphics[scale=.37]{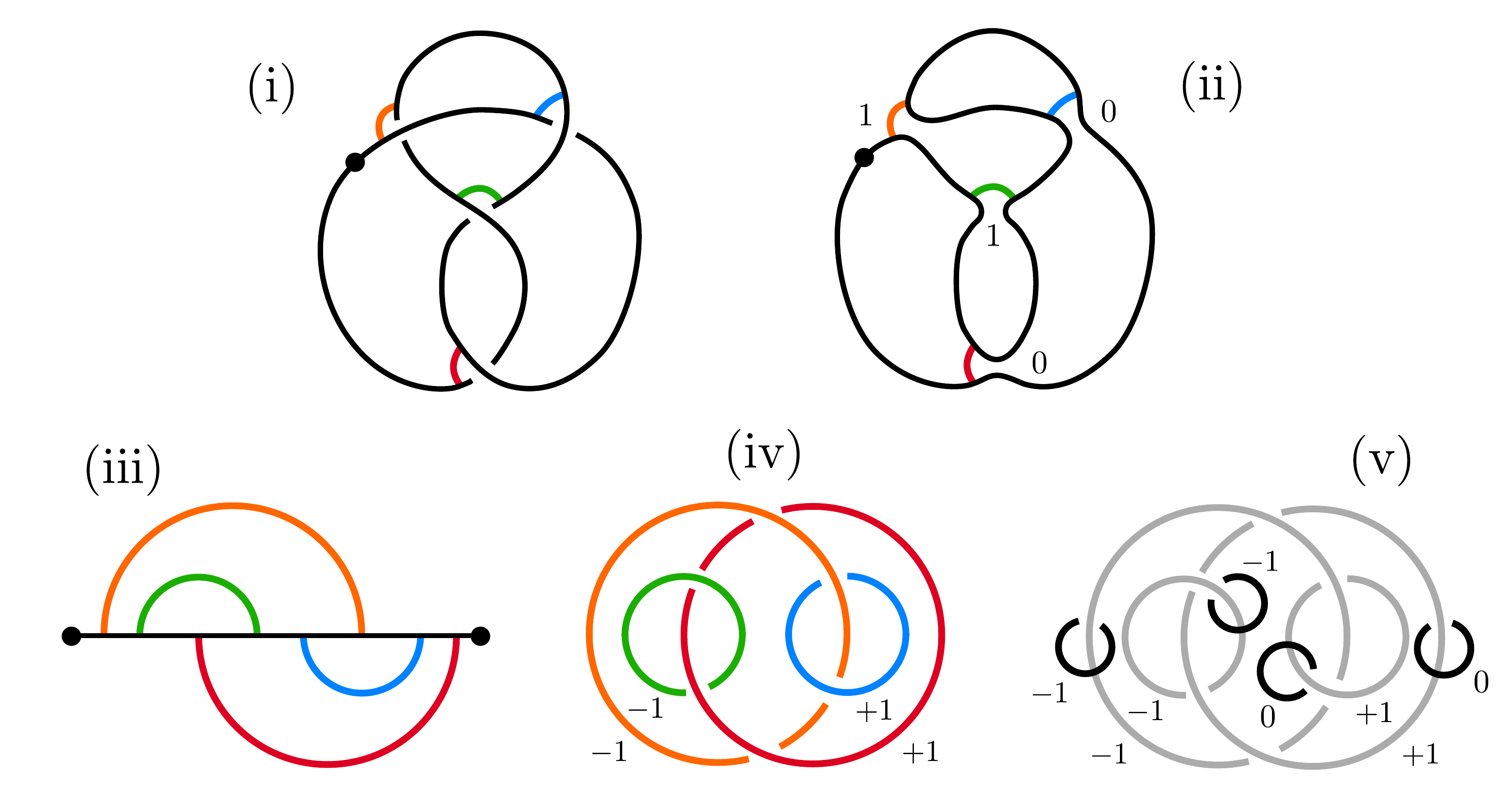}
\caption{To obtain a relative Kirby diagram for $(\x_{\boldsymbol{\infty}\mathbf{0}},\y_{\boldsymbol{\infty}})$ 
where $\y_{\boldsymbol{\infty}}=\overline{\Sigma(\lt)}$, we borrow some constructions from Bloom \protect\cite{bloom}.
With a diagram of the figure eight knot in (i) as an example, we first choose a resolution 
that yields one connected circle as in (ii), drawing small arcs where crossings used to be. 
We then cut the connected circle at the dot, and straighten it out, as in (iii). 
Reflecting this picture across the line, we obtain a surgery diagram (iv) for $\y_{\boldsymbol{\infty}}$ 
by choosing a $+1$ framing for each circle corresponding to a $0$-resolution, 
and a $-1$ framing for $1$-resolution circles. Finally, 
the relative Kirby diagram (v) is obtained by placing a small meridional circle on 
each circle in (iv) framed by $0$ or $-1$, depending on whether the circle 
corresponds to a $0$- or $1$-resolution, respectively.}
\label{fig:figureeight}
\end{figure}

\begin{proof}[Proof of Lemma \ref{lem:gr}]
By additivity and the fact that $\relp(\xbb_{\mathbf{0}\mathbf{1}})\equiv \sigma(\xbb_{\mathbf{0}\mathbf{1}})\equiv 0$ mod 2, it suffices to show that $\relp(\xbb_{\boldsymbol{\infty}\mathbf{0}})\equiv \sigma(\x_{\boldsymbol{\infty}\mathbf{0}})$. Write $\xbb=\xbb_{\boldsymbol{\infty}\mathbf{0}}$ and $\x$ for its base space. We have
\[
	\xbb = ([0,1]\times\ybb)\cup_{\lbb}(\cup^m_{i=1}\hbb)
\]
where $\lbb=\lbb_1\cup\cdots\cup\lbb_m$ 
is an $\so$-thickening of $L=L_1\cup\cdots\cup L_m$, and each $\lbb_i:\hbb_1\to\ybb\times\{1\}$ is as in \S \ref{sec:x}. 
Here we are viewing
\[
	\ybb = (\y\times\so)_\Psi(\lbb)
\]
as a bundle over $\y=\y_{\boldsymbol{\infty}}=\overline{\Sigma(L)}$ 
built from the geometric representative $L$ as in \S \ref{sec:geomrep}. 
In \S \ref{sec:main} 
we saw that $L$ is the boundary of a surface $S\subset \y$, 
so in fact $\ybb$ is a trivial bundle.
Note $\xbb$ is reducible to an $S^1$-bundle by its very construction. 
Let $\mathscr{L}$ be the associated complex line bundle. 
The Poincar\'{e} dual of a pre-image of $c_1(\mathscr{L})\in H^2(X;\mathbb{Z})$ in $H^2(X,\partial X;\mathbb{Z})$
is represented by the closed surface $S'\subset \text{int}(\x)$ 
which is the union of the cores of the 2-handles together with $S\subset \y\times\{1\}$. 
Indeed, it is straightforward to define a section of $\mathscr{L}$ with zero set $S'$. 
By the definition of $\relp(\xbb)$, it suffices to show that
\[
	[S']\cdot [S']\equiv \sigma(\x)\mod 2
\]
where $[S']\cdot [S']$ is the intersection product.
To do this we write down a relative Kirby diagram for $(\x,\y)$. 
We start by writing a surgery diagram for $\y=\overline{\Sigma(L)}$ 
using the chosen diagram $D$. 
For this we follow Bloom \cite{bloom}. 
First, choose $v\in\{0,1\}^m$ for which the resolution $D_v$ has 1 circle. 
We can always choose $D$ so that there is such a resolution.
Then, in $D_v$, having placed arcs where crossings once were,
cut the lone circle at an isolated point $p$ and 
unravel it, with the arcs attached, into a horizontal segment; 
then double it as in Figure \ref{fig:figureeight} (iv). Place a $+1$ framing on a circle 
in the resulting picture if that circle came from a $0$-resolution, 
and a $-1$ framing otherwise. This gives a surgery diagram 
for $\y=\overline{\Sigma(L)}$.

To turn this into a relative Kirby diagram for $(\x,\y)$, 
we simply add small meridians around each circle, 
framed with a $0$ if the circle is +1 framed and a $-1$ if 
the circle is $-1$ framed. The intersection number $[S']\cdot [S']$ 
is concentrated at the attaching locations of the 2-handles,
represented by the meridional circles in the relative Kirby diagram.
Thus there is a $-1$ contribution to $[S']\cdot [S']$ from each $1$-resolution in $D_v$. 
We conclude $[S']\cdot [S']=-|v|_1$. 
According to \cite{bloom} Prop. 1.7 and Lemma 9.4, the signature $\sigma(\x)$ is mod 2 
congruent to the vertex weight of a 1-circle resolution.
\end{proof}

We now discuss Corollary \ref{cor:2}. As alluded to in the introduction, 
\cite[Thm. 1]{man}, \cite[\S 5]{ors} and the remarks in \cite[\S 9.3]{ls} 
tell us that, for a quasi-alternating link $L$, the gradings $q$ and $t$ of $\kh(L)$ 
satisfy $q/2-t-\sigma/2=0$. Let us write $\delta=q/2-t-\sigma/2$; then
we may say that $\kh(L)$ is supported in $\delta$-grading $0$. 
Note that $\nu=0$ when $L$ is quasi-alternating. 
Further, as is described in \cite{man}, 
the rank of $\kh(L)_{q,t}$ is given by $|a_q|$, where
\[
	J_{L}(x)=\sum a_q x^q
\]
is the Jones polynomial, with conventions as given in \cite{ors}. 
The grading (\ref{eq:oddkhgr}) of Theorem \ref{thm:1}, 
which we shall call $\delta^\#$, is given by $\delta^\# = \delta + q + \sigma$ 
for quasi-alternating links. Note that $\delta$ and $\delta^\#$ agree modulo 2, 
implying that the spectral sequence collapses at the $E^2$-page.
Write $\kh(L)_j$ with $j\in\{0,2\}\subset\mathbb{Z}/4\mathbb{Z}$ for the $\delta^\#$-grading.
Then
\[
	\text{rk}_\mathbb{Z}\kh(L)_j = \sum_{q + \sigma \equiv j} |a_q|
\]
where the congruence is modulo 4. We remark that $\sum |a_q| = \text{det}(L)$, and
the sign of $a_q$ is $(-1)^{q/2+\sigma/2}$, 
cf. \cite[\S 9.1]{bloom}. It follows that
\[
	\text{rk}_\mathbb{Z}\kh(L)_j = \frac{1}{2}\left[\sum |a_q|+(-1)^{j/2}\sum a_q\right].
\]
Now we obtain the result of Corollary \ref{cor:2} using the fact that 
$J_L(1)=\sum a_q=2^{m-1}$, where $m$ is the number of components of $L$.

%% file: homology3spheres.tex
\section{Connected Sum Formulas}\label{sec:connect}
 
A closed, connected, oriented 3-manifold $\y$ is called an integral homology 3-sphere
if $H_1(\y;\mathbb{Z})=0$, or equivalently, if $\y$ has 
the same integral homology as the 3-sphere. In this section we study 
$\ih^\#(\y)$ when $\y$ is an integral homology 3-sphere. We relate the
$\mathbb{Z}/4$-graded group $\ih^\#(\y)$ to Floer's original $\mathbb{Z}/8$-graded
instanton homology $\ih(\y)$ using the trivial connection and $u$-maps studied 
by Donaldson \cite{d} and Fr{\o}yshov \cite{froy}.\\

\subsection{Gradings} Let $\y$ be an integral homology 3-sphere.
Let $\theta$ be the distinguished trivial connection on $\y\times\so$.
We can use $\theta$ to fix an absolute $\mathbb{Z}/8$-grading on 
$\ih(\y)$ as was done in Floer's original 
construction \cite{f1}. For a connection $\conna$ on $\y\times\so$ we set
\[
	\text{gr}(\conna) = -3-\mu(\theta,\conna)
\]
and on $\mathscr{G}$-classes $\connag$ this descends to a function with 
$\text{gr}(\connag)\in\mathbb{Z}/8$. Note that $\mathscr{G}_\text{ev}=\mathscr{G}$ 
in this setting. When $a$ is irreducible, $\text{gr}(\conna)=\mu(\conna,\theta)$. 
The trivial connection has $\text{gr}(\theta)=0$. 
The differential shifts this grading by $-1$
and the grading descends 
to the $\mathbb{Z}/2$-grading defined in \S \ref{sec:instgrading}. 
We write $\triv$ for the $\mathscr{G}$-class of $\theta$.
Our $\ih(\y)_i$ agrees with Donaldson's $\text{HF}(\y)_i$ in \cite{d}. 
Note that $\ih(\overline{\y})_i$ is the same as the cohomology group 
$\ih(\y)^{5-i}$. In particular, 
by the universal coefficients theorem, the vector spaces 
$\ih(\overline{\y})_i\otimes\mathbb{Q}$ and $\ih(\y)_{5-i}\otimes\mathbb{Q}$
are isomorphic. Our $\ih(\y)_i$ is the same as Fr{\o}yshov's $\text{HF}(\y)^{5-i} = 
\text{HF}(\overline{\y})_i$ from \cite{froy}.\\

\subsection{Other Boundary Maps}

From here on we fix a field $\ring$ which has char$(F)\neq 2$
and take all homology with $\ring$-coefficients. 
With an integral homology 3-sphere $\y$ fixed, we write $\chain_i=\chain(\y)_i$ and 
$\ih_i=\ih(\y)_i$. Following \cite{d,froy} we have maps
\[
	\delta:\chain_1\to \ring,\qquad \delta':\ring \to\chain_4
\]
defined using the trivial connection. 
For $\connag\in\mathfrak{C}^\text{irr}(\y)$ with $\text{gr}(\conna)\equiv 1$ we define 
$\delta(\connag)=\#\check{\moduli}(\connag,\triv)_0$, and for $\connbg$ 
with $\text{gr}(\connbg)\equiv -4$,
we define $\langle\delta'(1),\connbg\rangle=\#\check{\moduli}(\triv,\connbg)_0$. 
More precisely, one writes $\ring=\ring\Lambda(\triv)$ and 
$\epsilon[A]:\Lambda(\triv)\to\Lambda(\connag)$, and $\delta=\sum\epsilon[A]$ for each 
$[A]\in\check{\moduli}(\connag,\triv)_0$, and so on, as in \S \ref{sec:instantongroups}.
We will often conflate $\delta'$ with $\delta'(1)\in\chain_4$. These are chain maps, 
in the sense that $\delta\partial=\partial\delta'=0$, and we write
\[
	\boldsymbol{\delta}:\ih_1\to \ring, \qquad \boldsymbol{\delta}':\ring \to \ih_4
\]
for the induced maps on homology.

We also have maps that record data from the 3-dimensional 
moduli spaces $\check{\moduli}(\connag,\connbg)_3$,
\[
	v:\chain_i\to\chain_{i+4}.
\]
Our $v$ is $1/2$ times the $v$ of Fr{\o}yshov, and $4$ times the $U$ of Donaldson. 
That is, it is defined, roughly, by evaluating the 4-dimensional class $2\mu(\text{pt})$ 
over 4-dimensional moduli spaces $\moduli(\connag,\connbg)_4$. We refer to \cite[\S 3.1]{froy} and 
\cite[\S 7.3.1]{d} for precise definitions of $v$. We have in mind the following interpretation. First suppose $\check{\moduli}(\connag,\connbg)_3$ is connected.
We obtain a map $h:\check{\moduli}(\connag,\connbg)_3\to\so$ by 
evaluating the holonomy of a connection along the path from $(-\infty,y)$ to $(\infty,y)$ on the cylinder
$\mathbb{R}\times\y$. With some modifications, see \cite[\S 7.3.2]{d}, 
$\langle v(\connag),\connbg\rangle = \deg(h)$. If $\check{\moduli}(\connag,\connbg)_3$ has more than one component, the evaluation is done on each component, and then added together.

The map $v$ is not quite a chain map. As explained in \cite[\S 7.3.3]{d}, 
when $\text{gr}(\connag)\equiv 1$ and $\text{gr}(\connbg)\equiv -4$, 
there are ends of $\check{\moduli}(\connag,\connbg)_4$ 
modelled on $\so$, i.e. cylinders $\mathbb{R}\times\so$, 
one for each pair of instantons in $\check{\moduli}(\connag,\triv)_0\times\check{\moduli}(\triv,\connbg)_0$.
Each copy of $\so$ records the choices for gluing parameters. The holonomy 
at a cross-section is captured by the gluing parameter and has degree 1. Accounting for the 
other usual ends, modelled on 
$\mathbb{R}\times\check{\moduli}(\connag,\conncg)_i\times\check{\moduli}(\conncg,\connbg)_j$,
where $i=0$ and $j=3$, or vice versa, one is led to the relation
\begin{equation}
	\partial v - v \partial + \delta'\delta = 0, \label{eq:umap}
\end{equation}
see \cite[Thm. 4]{froy} and \cite[Prop. 7.8]{d}. Here $\delta = 0$ in 
gradings different from $1\in\mathbb{Z}/8$. In particular we obtain the maps
\begin{align*}
	\boldsymbol{v}:\ih_i\to\ih_{i+4},\qquad i\neq 0,1 \mod 8\\
	\boldsymbol{v}:\ih_0\to\text{coker}(\boldsymbol{\delta}'), \qquad \boldsymbol{v}:\text{ker}(\boldsymbol{\delta})\to\ih_5.
\end{align*}
\vspace{10px}

\subsection{Reduced Instanton Groups}

Fr{\o}yshov defined a $\mathbb{Z}/8$-graded group $\widehat{\ih}=\widehat{\ih}(\y)$ by
cutting down $\ih(\y)$ using the maps introduced above. Precisely,
\begin{align*}
	&\widehat{\ih}_i = \ih_i, \quad i \not\equiv 0,1,4,5 \mod 8\\
	&\widehat{\ih}_0 = \ih_0/\left(\text{$\sum$}\text{im}(\boldsymbol{v}^{2k+1}\boldsymbol{\delta}')\right),
	&\widehat{\ih}_4 = \ih_4/\left(\text{$\sum$}\text{im}(\boldsymbol{v}^{2k}\boldsymbol{\delta}')\right),\\
	&\widehat{\ih}_1 = \bigcap \text{ker}(\boldsymbol{\delta}\boldsymbol{v}^{2k})\subset\ih_1,
	&\widehat{\ih}_5 = \bigcap \text{ker}(\boldsymbol{\delta}\boldsymbol{v}^{2k+1})\subset\ih_5.
\end{align*}
Using these groups Fr{\o}yshov defined his $h$-invariant by
\[
	h(\y) = -\frac{1}{2}\left(\chi(\ih(\y))-\chi(\widehat{\ih}(\y))\right).
\]
This has several nice properties, among them
\[
	h(\overline{\y})=-h(\y),\qquad h(\y\#\y')=h(\y)+h(\y').
\]
It also descends to a homomorphism $h:\Theta_H^3\to\mathbb{Z}$, 
where $\Theta_H^3$ is the integral homology cobordism group.
Fr{\o}yshov showed that both $\ih$ and $\widehat{\ih}$ are 4-periodic
(recall that we are working with $\ring$-coefficients). By the chain level 
relation (\ref{eq:umap}) either $\boldsymbol{\delta}$ or $\boldsymbol{\delta}'$ is 
zero. It follows that, over $\mathbb{Q}$, 
we can go between $\ih$ and $\widehat{\ih}$ using only $h$. 
For example, if $h(\y)=0$, then $\widehat{\ih}=\ih$, whereas if $h(\y)>0$ then
$\widehat{\ih}_i=\ih_i$ for $i\neq 0,4$ and $\text{rk}(\widehat{\ih}_i)=\text{rk}(\ih_i)-h(\y)$ 
for $i=0,4$.

The maps $\boldsymbol{v}$ above induce maps $\widehat{v}:\widehat{\ih}(\y)_i\to\widehat{\ih}(\y)_{4+i}$ for 
each grading $i\in\mathbb{Z}/8$. As mentioned, 
this is half of Fr{\o}yshov's $u$ mentioned in 
the introduction:
\[
	\widehat{v} = u/2.
\]
We've chosen this normalization to avoid writing in certain factors of $2$.
Fr{\o}yshov showed that each $\widehat{v}$ is an isomorphism, and that 
$\widehat{v}^2-16$ is nilpotent, i.e.
\[
	(\widehat{v}^2-16)^n = 0
\]
for some $n>0$.
If $\ybb$ is admissible and $b_1(\y)>0$, there is no trivial connection to work with, 
and the maps $v:\chain_i\to\chain_{i+4}$ are indeed chain maps, inducing maps 
$\widehat{v}:\ih(\ybb)_i\to\ih(\ybb)_{i+4}$ 
for each grading $i$ (here we arbitrarily fix an absolute grading). 
Again, each $\widehat{v}$ is half of Fr{\o}yshov's $u$, is an isomorphism,
and $\widehat{v}^2-16$ is nilpotent. The hat notation in this case is 
used only for uniformity.\\

\subsection{Connected Sums}

In this section we recall the connected sum theorem of Fukaya \cite{fukaya}, reviewing the proof 
exposited by Donaldson in \cite[\S 7.4]{d}. This problem was also considered 
in \cite{li}.
 In the following sections we will 
adapt the proof to the settings of interest to us. Let $\y_1$ and $\y_2$ be 
integral homology 3-spheres. For $i=1,2$ write $\chain_{(i)}=\chain(\y_i)$ and 
$\partial_{(i)}$ for the corresponding differentials, and
$\delta_{(i)},\delta_{(i)}',v_{(i)}$ for the relevant boundary maps. For a graded $\ring$-module
$A$ define the shifted module $A[n]$ by $A[n]_i = A_{i-n}$. 
We define a chain complex
\[
	\chain = \left( \chain_{(1)}\otimes\chain_{(2)}\right) \oplus \left( \chain_{(1)}[3]\otimes\chain_{(2)}\right) \oplus \left(\chain_{(1)}\otimes \ring \right) \oplus \left( \ring \otimes\chain_{(2)}\right)
\]
\[
\partial = \left(
\begin{array}{cccc}
	\partial_{(12)} & 0 & 0 & 0\\
	v_{(12)} & -\partial_{(12)} & 
		1\otimes\delta'_{(2)} & \delta_{(1)}'\otimes 1 \\
	-1\otimes\delta_{(2)} & 0 & \partial_{(1)}\otimes 1 & 0 \\
	\delta_{(1)}\otimes 1 & 0 & 0 & \epsilon \otimes \partial_{(2)}\\
\end{array}\right)
\]
where $\partial_{(12)}=\partial_{(1)}\otimes 1 + \epsilon \otimes \partial_{(2)}$, 
$v_{(12)}=v_{(1)}\otimes 1 + 1 \otimes v_{(2)}$, and $\epsilon$ is equal,  in grading $k$, to $(-1)^k$ times the identity map on $\chain_{(1)}$.

\begin{theorem}[Fukaya] As $\mathbb{Z}/8$-graded $\ring$-modules, {\em $\ih(\y_1\#\y_2)\simeq H_\ast(\chain,\partial)$}.
\end{theorem}

\noindent For example, let $\y$ be the Poincar\'{e} homology 3-sphere $\Sigma(2,3,5)$. 
The reader can verify that
\[
	\ih(\y\#\y)\simeq \ring_1^2\oplus \ring_5^2, \qquad \ih(\y\#\overline{\y})=0
\]
using that $\chain(\y)=\ring_1\oplus \ring_5$ and $\delta,v$ are isomorphisms. 
Recall that subscripts indicate gradings.
These examples appear in \cite{fukaya}. 
Note that, generally, the $\delta,\delta',v$ maps for $\overline{\y}$ are the duals 
of the maps $\delta',\delta,v$ for $\y$, respectively.

\begin{figure}[t]
\includegraphics[scale=.45]{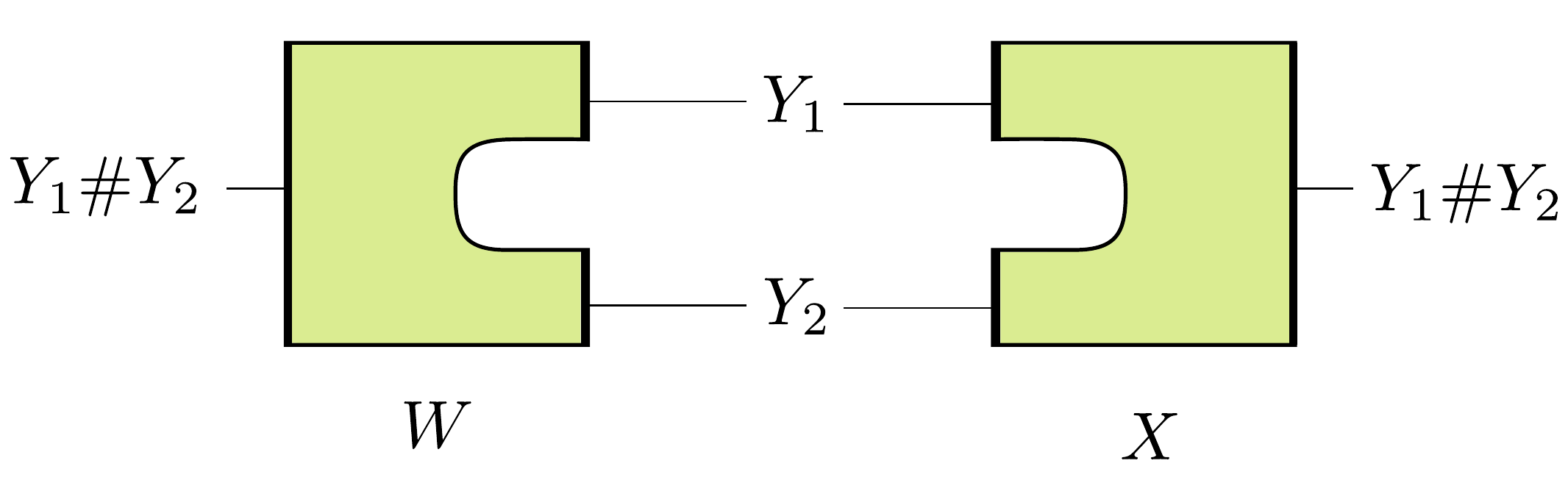}
\caption{The cobordism $W:\y_1\#\y_2\to\y_1\sqcup\y_2$ and its reverse, $\x$.}
\label{fig:cobs}
\end{figure}

We now review the proof that appears in \cite[\S 7.4]{d}. We mention at the outset 
that to avoid certain factors of $2$ that appear in the composition law (since 
we will glue along a disconnected 3-manifold), we enlarge the gauge transformation 
group when necessary, as in \cite[\S 5.1]{kmu}.
Let $\chain' = \chain(\y_1\#\y_2)$ and $\partial'$ 
be its differential. Let $\x:\y_1\sqcup\y_2\to\y_1\#\y_2$ be the cobordism which is $([0,1]\times\y_1)\natural ([0,1]\times\y_2)$, where the boundary sum is taken near $1$, and let $W:\y_1\#\y_2\to\y_1\sqcup\y_2$ be the corresponding cobordism when the boundary sum is taken near $0$. See Figure \ref{fig:cobs}. We define chain maps
\[
	m_\x:\chain \to \chain', \quad m_W:\chain'\to\chain
\]
as follows. The map $m_\x$ is given by four components:
\begin{align*}
	v_\x &:\chain_{(1)}\otimes\chain_{(2)} \to \chain',\\
	i_\x &:\chain_{(1)}[3]\otimes\chain_{(2)}\to\chain',\\
	{\delta'_\x}_{(2)} &:\chain_{(1)}\otimes \ring\to\chain',\\
	{\delta'_\x}_{(1)} &:\ring\otimes\chain_{(2)}\to\chain'.
\end{align*}
In the following, $\connag\in\mathfrak{C}^\text{irr}(\y_1)$, $\connbg\in\mathfrak{C}^\text{irr}(\y_2)$, 
and $\conncg\in\mathfrak{C}^\text{irr}(\y_1\#\y_2)$.
The map $i_\x$ counts 0-dimensional moduli spaces $\moduli(\connag,\connbg,\x,\conncg)_0$. 
The map $v_\x$ evaluates the holonomy of 3-dimensional 
moduli spaces $\moduli(\connag,\connbg,\x,\conncg)_3$
along a curve $\gamma_\x$ running from $\y_1$ to $\y_2$ on the incoming end of $\x$. The map 
${\delta'_\x}_{(2)}$ counts 0-dimensional moduli spaces $\moduli(\connag,\triv,\x,\conncg)_0$ 
where $\triv$ is a trivial connection class on $\y_2$, 
and ${\delta'_\x}_{(1)}$ is defined similarly, with $\triv$ on $\y_1$.
Now, $m_\x$ is a chain map because of the following relations. First,
\[
	i_\x\partial_{(12)}=\partial'i_\x
\]
is the usual relation for the map involving only irreducibles. Second, 
\begin{align}
	& i_\x v_{(12)} + v_\x\partial_{(12)} + {\delta'_\x}_{(1)}(\delta_{(1)}\otimes 1) - {\delta'_\x}_{(2)}(1\otimes \delta_{(2)}) = \partial' v_\x \label{eq:pieces}
\end{align}
records how the holonomy interacts with the ends of a 4-dimensional moduli space $\moduli(\connag,\connbg,\x,\conncg)_4$. This is essentially \cite[Thm. 6]{froy}. See Figure \ref{fig:pieces}.
Third, the relation
\[
	i_\x(\delta_{(1)}'\otimes 1) + {\delta'_\x}_{(1)}(\epsilon\otimes\partial_{(2)}) = \partial'{\delta'_\x}_{(1)} 
\]
and its analogue with indices swapped, records the ends of a 1-dimensional moduli 
space $\moduli(\triv,\connbg,\x,\conncg)_1$ where $\triv$ is the trivial connection class on $\y_1$. 
This is a variation of \cite[Lemma 1]{froy}.

\begin{figure}[t]
\includegraphics[scale=.4]{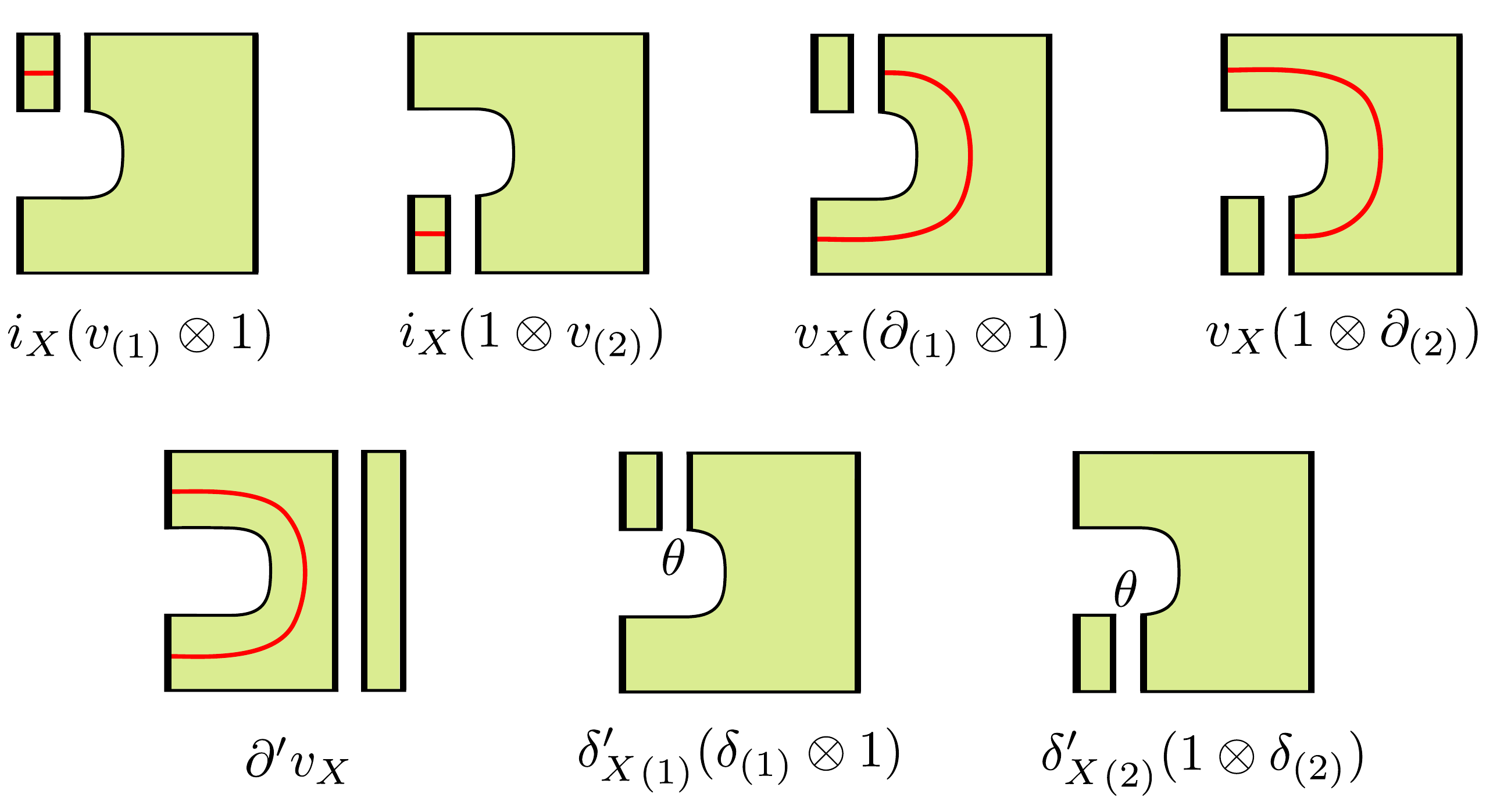}
\caption{A representation of the terms appearing in (\ref{eq:pieces}). 
The pieces represent counts of isolated instantons, 
unless there is a curve present in the interior, 
indicating a contribution from a $v$-map. 
All limiting connections are irreducible, except in the last two 
diagrams, where trivial limits $\theta$ are present.
The first two diagrams make up $i_\x v_{(12)}$ and the second two 
make up $v_\x \partial_{(12)}$.}
\label{fig:pieces}
\end{figure}

The map $m_W$ is defined similarly, this time with components
\begin{align*}
	i_\text{W} &:\chain'\to\chain_{(1)}\otimes\chain_{(2)}, \\
	v_\text{W} &:\chain'\to\chain_{(1)}[3]\otimes\chain_{(2)},\\
	{\delta_\text{W}}_{(2)} &:\chain'\to\chain_{(1)}\otimes \ring,\\
	{\delta_\text{W}}_{(1)} &:\chain'\to \ring\otimes\chain_{(2)}.
\end{align*}
Now, we argue that $m_\x$ and $m_W$ are chain homotopy inverse to one another. 
First consider $m_\x m_W$. We have
\[
	m_\x m_W = v_\x i_W + i_\x m_W +{\delta'_\x}_{(2)}{\delta_W}_{(2)} +{\delta'_\x}_{(1)}{\delta_W}_{(1)}.
\]
We claim that $m_\x m_W$ 
is chain homotopic to the map $m(Z,\gamma):\chain'\to\chain'$ obtained by evaluating $2\mu(\gamma)$ on the 
composite $Z=\x\circ W$ where $\gamma=\gamma_\x\cup\gamma_{W}$. This 
is the the same as the map defined by taking the degrees of modified holonomy maps 
$\moduli(\connag,Z,\conndg)_3\to\so$ 
along $\gamma$, see \cite[\S 5.1.2]{dk}. The chain homotopy is obtained by 
stretching the middle copies of $\y_1$ and $\y_2$. The 3-dimensional space $\moduli(\connag,Z,\conndg)_3$
where $\connag,\conndg$ are irreducible has four components after stretching:
\begin{align*}
	&\moduli(\connag,\x,\connbg,\conncg)_0\times\moduli(\connbg,\conncg,W,\conndg)_3\\
	&\moduli(\connag,\x,\connbg,\conncg)_3\times\moduli(\connbg,\conncg,W,\conndg)_0\\
	&\moduli(\connag,\x,\connbg,\triv)_0\times\so\times\moduli(\connbg,\triv,W,\conndg)_0\\
	&\moduli(\connag,\x,\triv,\conncg)_0\times\so\times\moduli(\triv,\conncg,W,\conndg)_0
\end{align*}
As in (\ref{eq:umap}), in the last two cases the holonomy is 
captured by the gluing space $\so$. The four components correspond, in order, 
to the four components of $m_\x m_W$ above. In this way, the chain homotopy 
from $m_\x m_\text{W}$ to $m(Z,\gamma)$ may be defined as a map using the 1-dimensional 
metric family that simultaneously stretches along $\y_1,\y_2$.

The next step is to use a surgery property, interesting in its own right,
due to Donaldson. 
We state it in a form convenient for our purposes. Let $\xbb:\ybb_1\to\ybb_2$ be an $\so$-bundle 
over a cobordism which restricts to admissible bundles over its boundary components. Let
$\gamma$ be a loop in the interior of the base of $\xbb$. Let $\xbb_\gamma$ be the bundle obtained 
by cutting out a neighborhood $S^1\times D^3\times\so$ lying over $\gamma$ and gluing back in 
a copy of $D^2\times S^2\times\so$. Denote by $m(\xbb,\gamma):\chain(\ybb_1)\to\chain(\ybb_2)$ 
the map obtained by evaluating $\mu(\gamma)$ on 3-dimensional moduli spaces $\moduli(\connag,\xbb,\connbg)_3$.

\begin{theorem}[see \cite{d} Thm. 7.16]
	$m(\xbb,\gamma)$ is chain homotopic to $m(\xbb_\gamma)$.\label{thm:surgery}
\end{theorem}

\noindent In our situation, observe that the surgered manifold $Z_\gamma$ is the 
product $[0,1]\times(\y_1\#\y_2)$. 
It follows that $m_\x m_W$ is chain homotopic to the identity.

Now consider $m_W m_\x$. This has 16 components
\[
	i_Wv_\x,\quad v_Wi_\x,\quad {\delta_W}_{(1)}{\delta'_\x}_{(1)},\quad \ldots
\]
It is chain homotopic to a map $f$ that counts similar data on the cobordism $W\circ\x$ 
with metric stretched very long along the internal connected sum portion between $[0,1]\times\y_1$ 
and $[0,1]\times\y_2$. The map $f$ has components corresponding to the components of $m_W m_\x$, 
but most of them vanish. For instance, the 7 components of $f$ corresponding to
\[
	i_Wi_\x, \quad {\delta_W}_{(i)}i_\x, \quad i_W{\delta'_\x}_{(i)},\quad {\delta_W}_{(i)}{\delta'_\x}_{(j)}\quad (i\neq j)
\]
all vanish by index arguments. Each counts instantons $A$ with $\mu(A)=0$ obtained by gluing an instanton $A_1$ on $\mathbb{R}\times\y_1$ to an instanton $A_2$ over $\mathbb{R}\times\y_2$ along a 3-sphere. For $i=1,2$ at least one of the limits on $\mathbb{R}\times\y_i$ is irreducible. Thus both $A_1,A_2$ are irreducible. It follows from $0=\mu(A)=\mu(A_1)+\mu(A_2)+3$ and $\mu(A_i)\geq 0$ that no such $A$ exist. Similarly, the 4 components of $f$ corresponding to
\[
	{\delta_W}_{(i)}v_\x,\quad v_W{\delta'_\x}_{(i)}
\]
are zero. These components require 3-dimensional moduli spaces. However, 
with the neck stretched, the relevant 3-dimensional moduli spaces are $\moduli(\connag,\connbg)_0\times\so\times\moduli(\conncg,\conndg)_0$
where one of $\connag,\connbg,\conncg,\conndg$ is the trivial class $\triv$ and $\connag,\connbg$ are connection classes on $\y_1$ and $\conncg,\conndg$ on $\y_2$. But $\moduli(\connag,\triv)_0$ is empty for any irreducible $\connag$. Next, the 4 components of $f$ corresponding to
\[
	i_\text{W}v_\x,\quad v_\text{W}i_\x, \quad {\delta_{\text{W}}}_{(i)}{\delta'_{\x}}_{(i)}
\]
are identity maps. For instance, the first one uses 3-dimensional spaces modelled on 
$\moduli(\connag,\connbg)_0\times\so\times\moduli(\conncg,\conndg)_0$ from gluing; the holonomy map $v_\x$ captures 
the gluing parameter just as in (\ref{eq:umap}), leaving us to count 
$\moduli(\connag,\connbg)_0\times\moduli(\conncg,\conndg)_0$. 
Of course $\moduli(\connag,\connbg)_0$ forces $\connag=\connbg$ and has one translation invariant irreducible flat connection. 
Finally, we are left with 1 component of $f$ corresponding to 
\[
	v_Wv_\x
\]
which may be nonzero. However, we know that $f$ is the identity plus this off-diagonal term, 
and thus induces an isomorphism on homology. So $m_W m_\x$ also induces an isomorphism on homology. 
Because $m_\x m_W$ induces the identity on $\ih(\y_1\#\y_2)$, so does $m_\text{W} m_\x$. 
This completes the proof.\\

\subsection{Connected sum with non-trivial bundles}

In this section we state two variants of the connected sum theorem, 
when one or both of $\y_1$ and $\y_2$ is replaced by a non-trivial admissible bundle. 
We then explain how the proof above adapts to these cases. These are
simpler than the above, having fewer trivial connections to deal with.

We first consider the case where $\ybb_1$ is trivial 
and $\y_1$ is an integral homology 3-sphere, 
but $\ybb_2$ is non-trivial and admissible. Let 
$\chain_{(1)}=\chain(\ybb_1)$ with maps $\partial_{(1)},\delta,\delta',v_{(1)}$. 
Let $\chain_{(2)}=\chain(\ybb_2)$ with maps $\partial_{(2)},v_{(2)}$. Define
\[
	\chain = \left( \chain_{(1)}\otimes\chain_{(2)}\right) \oplus \left( \chain_{(1)}[3]\otimes\chain_{(2)}\right)\oplus \left( \ring\otimes\chain_{(2)}\right)
\]
\[
\partial = \left(
\begin{array}{ccc}
	\partial_{(12)} & 0 & 0\\
	v_{(12)} & -\partial_{(12)} &  \delta'\otimes 1 \\
	\delta\otimes 1 & 0  & \epsilon \otimes \partial_{(2)}\\
\end{array}\right)
\]
with notation as before.

\begin{theorem} 
	 Let $\ybb_1$ and $\ybb_2$ be admissible bundles, with $\ybb_1$ trivial and $\ybb_2$ non-trivial. As $\mathbb{Z}/8$-graded $\ring$-modules, {\em $\ih(\ybb_1\#\ybb_2)\simeq H_\ast(\chain,\partial)$}.\label{thm:connect2}
\end{theorem}

\noindent As before, we let $\chain'=\chain(\ybb_1\#\ybb_2)$ and let $\partial'$ be its differential. 
Let
\[
	\xbb:\ybb_1\sqcup\ybb_2 \to \ybb_1\#\ybb_2
\]
be the cobordism bundle obtained 
from a boundary sum between $[0,1]\times\ybb_1$ and $[0,1]\times\ybb_2$ near $1$, 
making some inessential choices in gluing the bundles. 
Let $\wbb$ be the cobordism in the reverse direction obtained
from the boundary sum near $0$. We define chain maps
\[
	m_\xbb:\chain\to\chain', \quad m_\wbb:\chain'\to\chain.
\]
The map $m_\xbb$ is given by three components:
\begin{align*}
	v_\xbb &:\chain_{(1)}\otimes\chain_{(2)} \to \chain',\\
	i_\xbb &:\chain_{(1)}[3]\otimes\chain_{(2)}\to\chain',\\
	{\delta'_\xbb} &:\ring \otimes\chain_{(2)}\to\chain'.
\end{align*}
The map $i_\xbb$ counts instantons in 0-dimensional moduli spaces on $\xbb$ with all limits irreducible;
the map $v_\xbb$ evaluates holonomy along a path $\gamma_\xbb$ from $\y_1$ to $\y_2$ on 
3-dimensional moduli spaces with irreducible limits on $\xbb$; the map $\delta'_\xbb$ counts 0-dimensional moduli 
spaces over $\xbb$ where the limit over $\y_1$ is trivial. The map $m_\xbb$ is a chain map because of the following.
First, we have the usual relation for the map involving only irreducibles, $i_\xbb\partial_{(12)}=\partial'i_\xbb$. Second, 
\begin{align*}
	& i_\xbb v_{(12)} + v_\xbb\partial_{(12)} + {\delta'_\xbb}(\delta\otimes 1) = \partial' v_\xbb.
\end{align*}
These relations are the same as before, except that all terms involving a trivial connection
on $\y_2$ do not arise. In particular, all diagrams in Figure \ref{fig:pieces} are 
relevant except the last. Third, we have the relation
\[
	i_\xbb(\delta'\otimes 1)  = \partial'{\delta'_\xbb}
\]
which again is the same as before but with the term involving a trivial connection 
on $\y_2$ absent. The map $m_\wbb$ is 
defined similarly to $m_\text{W}$, with the component involving the trivial connection on $\y_2$ 
thrown out.

We proceed as before. The first composite is $m_\xbb m_\wbb=i_\xbb v_\wbb + v_\xbb i_\wbb$, and this is 
chain homotopic to $m(\mathbb{Z},\gamma)$ where $\mathbb{Z}=\xbb\circ\wbb$ and 
$\gamma = \gamma_\wbb\cup\gamma_\xbb$ by stretching along $\y_1,\y_2$. The surgery theorem 
\ref{thm:surgery} applies, so $m_\xbb m_\wbb$ is chain homotopic to $m(\mathbb{Z}_\gamma)$, 
which is the identity. The other composite $m_\wbb m_\xbb$ now has only 9 components. We stretch the neck 
as before, so that $m_\wbb m_\xbb$ is chain homotopic to a map $f$; 
the terms of $f$ corresponding to the 9 components all vanish except the diagonal ones, 
which are the identity, and possibly $v_\wbb v_\xbb$. As before, $m_\wbb$ and $m_\xbb$ 
are chain homotopy inverses, and the proof follows through.

Next, we consider the case where both $\ybb_1$ and $\ybb_2$ are non-trivial. 
For $i=1,2$ let 
$\chain_{(i)}=\chain(\ybb_i)$ with maps $\partial_{(i)},v_{(i)}$. Define
\[
	\chain = \left( \chain_{(1)}\otimes\chain_{(2)}\right) \oplus \left( \chain_{(1)}[3]\otimes\chain_{(2)}\right)
\]
\[
\partial = \left(
\begin{array}{cc}
	\partial_{(12)} & 0 \\
	v_{(12)} & -\partial_{(12)} \\
\end{array}\right)
\]
with notation as before.

\begin{theorem} 
	 Let $\ybb_1$ and $\ybb_2$ be non-trivial admissible bundles. As $\mathbb{Z}/8$-graded $\ring$-modules, {\em $\ih(\ybb_1\#\ybb_2)\simeq H_\ast(\chain,\partial)$}.\label{thm:connect3}
\end{theorem}

\noindent This is the simplest case of all. Let $\chain'=\chain(\ybb_1\#\ybb_2)$ with differential $\partial'$.
Let
\[
	\xbb:\ybb_1\sqcup\ybb_2 \to \ybb_1\#\ybb_2
\]
be the cobordism bundle obtained 
from a boundary sum between $[0,1]\times\ybb_1$ and $[0,1]\times\ybb_2$ near $1$, 
making some inessential gluing choices. Let $\wbb$ be the cobordism in the reverse direction obtained
from the boundary sum near $0$. As before, we can define chain maps 
$m_\xbb$ and $m_\wbb$. Here $m_\xbb$ is given by two components, 
$v_\xbb :\chain_{(1)}\otimes\chain_{(2)} \to \chain'$ and 
$i_\xbb :\chain_{(1)}[3]\otimes\chain_{(2)}\to\chain'$.
As usual, $i_\xbb$ counts 0-dimensional moduli spaces on $\xbb$ with all limits irreducible,
and $v_\xbb$ takes holonomy along a path $\gamma_\xbb$ from $\y_1$ to $\y_2$ on 
3-dimensional moduli spaces with irreducible limits on $\xbb$. The relations that 
make $m_\xbb$ a chain map are just $i_\xbb\partial_{(12)}=\partial'i_\xbb$ and $i_\xbb v_{(12)} 
+ v_\xbb\partial_{(12)} = \partial' v_\xbb$, and follow from the previous cases, 
with the terms involving trivial connections thrown out. This latter relation 
is represented by Figure \ref{fig:pieces} with the last two diagrams omitted. 
The rest of the argument is the same 
as before.\\

\subsection{Framed homology for integral homology 3-spheres}\label{sec:inthomproof}

Now we apply the above results to compute $\ih^\#(\y)$ with $\ring$-coefficients 
for an integral homology 3-sphere $\y$, proving Theorem \ref{thm:integerhom}. 
Recall that $\ring$ is a field with char$(\ring)\neq 2$. Let 
$\tbb^3$ be a non-trivial bundle over $T^3$ geometrically represented 
by an $S^1$-factor of $T^3$. Let $V=\ring_0\oplus \ring_4$ 
be the chain complex that computes $\ih(\tbb^3)$. 
Write $\tau:V\to V$ for the $v$-map on $\tbb^3$, with which our normalization 
may be written as the degree 4 involution that multiplies by 4. Write $\chain=\chain(\y)$ 
and $\partial,v,\delta,\delta'$ for its relevant maps. Now let
\[
	\chainbf = \left(\chain\otimes V \right)\oplus \left(\chain[3]\otimes V \right)\oplus \left(\ring\otimes V \right)
\]
\[
\partialbf = \left(
\begin{array}{ccc}
	\partial\otimes 1 & 0 & 0\\
	v\otimes 1 + 1\otimes\tau& -\partial\otimes 1 &  \delta'\otimes 1 \\
	\delta\otimes 1 & 0  & 0\\
\end{array}\right)
\]
Theorem \ref{thm:connect2} tells us that
\[
	H_\ast(\chainbf,\partialbf)\simeq \ih(\ybb\#\tbb^3)=\ih^\#(\y)[4]\oplus\ih^\#(\y)
\]
where $\ybb$ is the trivial bundle over $\y$. 
Consider the filtration on $(\chainbf,\partialbf)$ given by
\[
	0 \subset \chain[3]\otimes V \subset \left(\chain[3]\otimes V\right)\oplus \left(\ring \otimes V\right)\subset \chainbf.
\]
This induces a spectral sequence with $E^2$-page
\[
	\left(\text{ker}(\boldsymbol{\delta})\otimes V\right)
	\oplus \left(\text{ker}(\boldsymbol{\delta}')/\text{im}(\boldsymbol{\delta})\otimes V\right)
	\oplus \left(\text{coker}(\boldsymbol{\delta}')[3]\otimes V\right)
\]
with the only non-zero component of the differential coming from
\[
	\phi:=\boldsymbol{v}\otimes 1 + 1\otimes\tau:\text{ker}(\boldsymbol{\delta})\otimes V\to\text{coker}(\boldsymbol{\delta}')[3]\otimes V.
\]
We are writing all modules as $\mathbb{Z}/8$-graded modules; for example, 
$\text{ker}(\boldsymbol{\delta})_i=\ih(\y)_i$ when $i\neq 1$, and similarly $\text{coker}(\boldsymbol{\delta}')_i=\ih(\y)_i$ when $i\neq 4$. Also note that the component $\text{ker}(\boldsymbol{\delta}')/\text{im}(\boldsymbol{\delta})$ is 
supported in grading $0$ and is either $\ring$ or $0$. Write
\[
	\phi_i = (\boldsymbol{v}\oplus\boldsymbol{v}+\sigma)_i:\text{ker}(\boldsymbol{\delta})_i\oplus\text{ker}(\boldsymbol{\delta})_{i+4}\to\text{coker}(\boldsymbol{\delta}')_{i+4}\oplus\text{coker}(\boldsymbol{\delta}')_{i}
\]
where $\sigma(x,y)=(4y,4x)$ is the degree 4 involution induced by $\tau$. Then
\begin{align*}
	&\ih^\#(\y)_0 \simeq \text{ker}(\phi_0)\oplus\text{coker}(\phi_1)\oplus\text{ker}(\boldsymbol{\delta}')/\text{im}(\boldsymbol{\delta}),\\
	&\ih^\#(\y)_i \simeq \text{ker}(\phi_i)\oplus\text{coker}(\phi_{i+1}), \quad i=1,2,3.
\end{align*}
Recall that $\widehat{v}$ is a degree 4 automorphism of $\widehat{\ih}(\y)$. We claim that
\begin{align*}
	&\text{ker}(\phi_0)\simeq\text{ker}(\widehat{v}^2-16)_0\oplus\text{im}(\boldsymbol{\delta}')_4,\\
	&\text{ker}(\phi_i)\simeq\text{ker}(\widehat{v}^2-16)_i, \quad i=1,2,3.
\end{align*}
To prove Theorem \ref{thm:integerhom} it 
suffices to consider the case in which $h(\y)\leq 0$, so that $\boldsymbol{\delta}'=0$ and 
$\widehat{\ih}_i=\ih_i$ for $i\neq 1,5$. For if $h(\y)>0$, then the theorem applies for $\overline{\y}$, 
which has $h(\overline{\y})=-h(\y)<0$, and the $\ring[\widehat{v}]$-module $\widehat{\ih}$ dualizes upon 
orientation reversal. Thus for $i=0,2,3$ we have
\[
	\phi_i:\ih_i\oplus\ih_{i+4}\to\ih_{i+4}\oplus\ih_i
\]
and each $\widehat{\ih}_i=\ih_i$ with $\boldsymbol{v}=\widehat{v}$ an isomorphism. The isomorphisms
$\text{ker}(\phi_i)\simeq\text{ker}(\widehat{v}^2-16)_i$ for $i=0,2,3$ are given by $(x,y)\mapsto x$ with inverse 
$x\mapsto (x,-\widehat{v}x/4)$. Next, consider
\[
	\phi_1:\text{ker}(\boldsymbol{\delta})_1\oplus\ih_5\to\ih_1\oplus\ih_5.
\]
We have an isomorphism $\text{ker}(\phi_1)\simeq\text{ker}(\boldsymbol{v}^2-16)_1\subset\text{ker}(\boldsymbol{\delta})_1$ given by $(x,y)\mapsto x$ with inverse $x\mapsto (x,-\boldsymbol{v}^{-1}x/4)$, 
using the isomorphism $\boldsymbol{v}:\ih_5\to\ih_1$. The natural inclusion $\widehat{\ih}_1\to\ih_1$ 
induces an injection $\text{ker}(\widehat{v}^2-16)_1\to\text{ker}(\boldsymbol{v}^2-16)_1$. Note here that $\widehat{v}$ is just the restriction of $\boldsymbol{v}$ to $\widehat{I}$. We show 
that this is surjective and hence an isomorphism. Given $x\in\ih_1$ with $\boldsymbol{\delta}x=0$ 
and $\boldsymbol{v}^2x=16x$ we must show $x\in\text{ker}(\boldsymbol{\delta}\boldsymbol{v}^{2k})$ 
for all $k>0$. But $\boldsymbol{v}^2x=16x$ implies $\boldsymbol{\delta}\boldsymbol{v}^{2k}x=4^k\boldsymbol{\delta}x=0$. Having computed $\text{ker}(\phi_i)$, dimension counting then yields
\begin{align*}
	&\text{coker}(\phi_1)\simeq\text{ker}(\widehat{v}^2-16)_1\oplus\text{im}(\boldsymbol{\delta}),\\
	&\text{coker}(\phi_i)\simeq\text{ker}(\widehat{v}^2-16)_i, \quad i=0,2,3.
\end{align*}
Using in our case that $\dim(\text{im}(\boldsymbol{\delta}))+\dim(\text{ker}(\boldsymbol{\delta}')/\text{im}(\boldsymbol{\delta}))=1$, 
we deduce that
\[
	\ih^\#(\y) \simeq \text{ker}(\widehat{v}^2-16)\otimes H_\ast(S^3)\oplus H_\ast(\text{pt.})
\]
where it is understood that $\widehat{v}^2-16$ is acting on $\bigoplus_{i=0}^3\widehat{\ih}_i$. This 
proves the first part of Thm. \ref{thm:integerhom}.\\

\subsection{Framed homology for non-trivial bundles}

Let $\ybb$ be a non-trivial admissible bundle over $\y$ 
geometricially represented by $\lambda \subset \y$.
We now write $\ih^\#(\y;\lambda)$ in terms of $\ih(\ybb)$. 
Let $V=\ring_0\oplus \ring_4$ 
and $\tau: V\to V$ be as before. Write $\chain=\chain(\ybb)$ 
and $\partial,v$ for its maps, and set
\[
	\chainbf = \left(\chain\otimes V \right)\oplus \left(\chain[3]\otimes V \right)
\]
\[
\partialbf = \left(
\begin{array}{cc}
	\partial\otimes 1 &  0\\
	v\otimes 1 + 1\otimes\tau& -\partial\otimes 1
\end{array}\right)
\]
Theorem \ref{thm:connect3} tells us that
\[
	H_\ast(\chainbf,\partialbf) \simeq \ih(\ybb\#\tbb^3)=\ih^\#(\y;\lambda)[4]\oplus\ih^\#(\y;\lambda).
\]
This is a degeneration of the computation in \S \ref{sec:inthomproof}. We want the kernel and cokernel of
\[
	\widehat{v}\oplus \widehat{v} + \sigma: \ih[4]\oplus\ih\to\ih\oplus\ih[4].
\]
The kernel is isomorphic to $\text{ker}(\widehat{v}^2-16)$ by the assignment $(x,y)\mapsto x$, inverse to $x\mapsto (x,-\widehat{v}x/4)$. 
The cokernel is of course the same, and we obtain
\begin{equation}
	\ih^\#(\y;\lambda) \simeq \text{ker}(\widehat{v}^2-16)\otimes H_\ast(S^3)\label{eq:nontriv}
\end{equation}
as relatively $\mathbb{Z}/4$-graded $\ring$-modules, where it is understood 
that $\widehat{v}^2-16$ is acting on four consecutively 
graded summands of $\ih(\ybb)$. This proves the second part of Thm. \ref{thm:integerhom}.\\

\subsection{Degenerations}

In this section we briefly consider a few cases in which the 
isomorphisms obtained degenerate into stricter relations between 
framed instanton homology and Floer's instanton homology, 
and in particular, we prove Corollaries \ref{cor:3}, \ref{cor:4} and \ref{cor:5}.

First, suppose $\ybb$ is a non-trivial admissible bundle over $\y$ 
geometrically represented by $\lambda$. As mentioned in \cite[\S 6]{froy}, 
when there exists a surface $\Sigma\subset\y$ of genus $\leq 2$ with
$\ybb|_\Sigma$ non-trivial, then $u^2=64$ on $\ih(\ybb)$. In this case
(\ref{eq:nontriv}) yields
\[
	\ih^\#(\y;\lambda)\otimes H_\ast(S^4) \simeq \ih(\ybb)\otimes H_\ast(S^3)
\]
as relatively $\mathbb{Z}/4$-graded $\ring$-modules. The term $H_\ast(S^4)$ appears because we take the full $\mathbb{Z}/8$-graded group $\ih(\ybb)$ on the right, instead of four consecutive summands as before. Now suppose 
$K$ is a knot in $S^3$ of genus $\leq 2$. Denote the result of $r$-surgery 
on $K$ by $\y_r$. For $r=1$, the exact triangle, combined with passing to the reduced groups $\widehat{I}$, yields a map
\begin{equation}
	\ih(\ybb_0)\to \widehat{\ih}(\y_{1})\label{eq:0to1}
\end{equation}
where $\ybb_0$ is a non-trivial bundle over $\y_0$. This map is an \textit{injection}. This follows from Fr{\o}yshov's observation after Thm. 10 in \cite{froy} that, in this situation, when passing to the reduced groups $\widehat{I}$, the surgery triangle retains exactness at the homology 3-spheres (but not $\ih (\ybb_0)$). If $r=-1$, we similarly obtain a \textit{surjection} $\widehat{\ih}(\y_{-1})\to \ih(\ybb_0)$. In either case, we
can form a surface $\Sigma$ in $\y_0$ by capping off a Seifert surface for $K$ of genus $\leq 2$ by a meridional
disk for the new framed knot in $\y_0$. The bundle $\ybb_0|_{\Sigma}$ is non-trivial, and so $u^2=64$ on $\ih(\ybb_0)$. Thus $u^2=64$ on $\widehat{\ih}(\y_{\pm1})$. With Theorem \ref{thm:integerhom}, this implies Corollary \ref{cor:3}.

Now we consider the proofs of Corollaries \ref{cor:4} and \ref{cor:5}. By the remarks in the introduction of \cite{froy}, we have
\[
	h(\Sigma(2,3,6k+1))=0, \quad h(\Sigma(2,3,6k-1))>0.
\]
Fintushel and Stern \cite{fs} compute
\[
	I(\Sigma(2,3,6k+1)) = F_1^{\lfloor k/2 \rfloor} \oplus F_3^{\lceil k/2 \rfloor}\oplus F_5^{\lfloor k/2 \rfloor} \oplus F_7^{\lceil k/2 \rfloor},
\]
from which Corollary \ref{cor:4} follows, as $\Sigma(2,3,6k+1)$ is $+1$-surgery on a twist knot with $k$ full twists, a knot of genus 1. On the other hand, $\Sigma(2,3,6k-1)$ is $-1$-surgery on a twist knot $K$ with $2k-1$ half twists. Since $K$ is also genus $1$, the inequality of Corollary 1 of \cite{froyh} yields $h(\Sigma(2,3,6k-1))=1$. Combined with Fintushel and Stern's computation from \cite{fs}, we obtain
\[
	\widehat{I}(\Sigma(2,3,6k-1))=F_1^{\lceil k/2 \rceil -1}\oplus F_3^{\lfloor k/2 \rfloor}\oplus F_5^{\lceil k/2 \rceil -1}\oplus F_7^{\lfloor k/2 \rfloor}.
\]
Now Corollary \ref{cor:5} follows from Corollary \ref{cor:3}.\\

%% file: euler.tex
\section{The Euler Characteristic}\label{sec:euler}

In this section we prove Cor. \ref{cor:euler}, which computes 
the Euler characteristic of $\ih^\#(\y;\lambda)$
where $\y$ is any closed, oriented 3-manifold and $\lambda$ is any 
unoriented closed 1-manifold in $\y$.
The claim is that $\chi(\ih^\#(\y;\lambda)) = |H_1(\y;\mathbb{Z})|$, 
where the expression on the right side means the cardinality of $H_1(\y;\mathbb{Z})$ if it is finite, 
and is zero otherwise.

\begin{proof}[of Cor. \ref{cor:euler}]
We make the abbreviations 
\[
	i(\y;\lambda)=\chi(\ih^\#(\y;\lambda)), \quad |\y|=|H_1(\y;\mathbb{Z})|. 
\]
Note that \S \ref{sec:prod} implies the multiplicativity
\begin{equation}\label{eq:multip}
	i(\y;\lambda) i(\y';\lambda')=i(\y\#\y';\lambda\cup\lambda').
\end{equation}
We also note that $i(\y)=1$ when $|\y|=1$ by 
Theorem \ref{thm:integerhom}.

Next, we claim the result is true for rational homology 
3-spheres $\y$ that are obtained by integral surgery on an algebraically split link. 
That is, $\y$ is the result of $(p_1,\ldots,p_k)$-surgery on a framed 
link $L=L_1\cup\cdots\cup L_k$ in $S^3$ whose pairwise linking numbers vanish.
Thus $|\y|=|p_1\cdots p_k|$. Assume the result true 
for $|\y|<n$. Since the case $|\y|=1$ has already been established, 
we may assume that $\y$ is not an integral homology 3-sphere, and (by reordering) that $|p_1| > 1$. 
Let $Z_p$ be $(p,p_2,\ldots,p_k)$-surgery on $L$. We have 
an exact sequence
\[
\cdots \ih^\#(Z_{\infty};{\lambda})\to\ih^\#(Z_{p_1-1};{\lambda\cup\mu})\to \ih^\#(Z_{p_1};{\lambda})
\to\ih^\#(Z_{\infty};{\lambda})\cdots
\]
The degree of the first map is odd, while the other two are even, cf. \cite[\S 42.3]{kmm}.
Observing that $Z_{p_1}=\y$, we obtain
\[
	i(\y;\lambda) = i(Z_{p_1-1};{\lambda\cup\mu}) + i(Z_\infty;\lambda).
\]
By the induction hypothesis, the right side is
\[
	|(p_1-1)p_2\cdots p_k| + |p_2\cdots p_k| = n,
\]
establishing the result for all rational homology 3-spheres which are obtained by
integral surgeries on algebraically split links.

We now prove the result for all rational homology 3-spheres $\y$. 
We use the fact that for any such 3-manifold, there is a framed algebraically 
split link $L\subset S^3$ such that some integral surgery on $L$ yields 
$Z=\y\# \y'$, where $\y'$ is a connected sum of lens spaces 
of type $L(p,1)$, cf. \cite[Cor. 2.5]{ohtsuki}. 
Since $\y'$ is integral surgery on an algebraically split link, 
$i(\y')=|\y'|$. Then (\ref{eq:multip}) yields
\[
	i(\y;\lambda) = i(Z;\lambda)/i(\y') = |Z|/|\y'| = |\y|,
\]
establishing the result for all rational homology 3-spheres.

Finally, we consider the case in which $b_1(\y)>0$.
We can always find $Z$ 
and a framed knot $\knot\subset Z$ such that $\y$ is 0-surgery on $K$ 
and $b_1(Z)+1=b_1(\y)$. We have an exact sequence 
\[
\cdots \ih^\#(Y;{\lambda}) \to \ih^\#(Z_{1};{\lambda\cup\mu})\to\ih^\#(Z;{\lambda})\to \ih^\#(Y;{\lambda})\cdots
\]
where $Z_1$ is the result of 1-surgery on $K$. 
The degree of the first two maps are even, while the third is odd, again cf. \cite[\S 42.3]{kmm}. Thus
\[
	i(\y;\lambda)=i(Z_{1};{\lambda\cup\mu})-i(Z;{\lambda}).
\]
The proof is again by induction. If $b_1(\y)=1$, then the right side is known, 
because $Z_{1}$ and $Z$ are rational homology spheres; we have $|Z_1|=|Z|$, so the right side vanishes.
Now suppose the result has been proven for $0<b_1<n$. If $b_1(\y)=n$, both terms on the right side vanish
by the induction hypothesis, and the proof is complete.
\end{proof}